\documentclass[11pt,a4paper,times]{amsart}
\usepackage{amsmath,amssymb,amsfonts,eucal}
\usepackage{amsthm}
\usepackage{bm}
\usepackage{bbm}
\usepackage{geometry}
\usepackage{array}
\usepackage{resizegather}
\usepackage{graphicx}
\usepackage{tabularx}
\usepackage{color}
\usepackage{cancel}
\usepackage[normalem]{ulem}
\usepackage[latin1]{inputenc}
\usepackage[shortlabels]{enumitem}
\usepackage{mathtools}
\usepackage[colorlinks=true]{hyperref}

\hypersetup{
    colorlinks=true,
    linkcolor=blue,  
    citecolor=blue,   
    urlcolor=blue     
}

\theoremstyle{plain}
\newtheorem{thm}{Theorem}[section]  
\newtheorem{thmx}{Theorem}
\newtheorem{corx}[thmx]{Corollary}

\newtheorem{lemma}[thm]{Lemma}
\newtheorem{proposition}{Proposition}[section]

\theoremstyle{definition}
\newtheorem{definition}[thm]{Definition}
 
\newtheorem{remark}{Remark}[section]

\numberwithin{equation}{section}

\DeclareMathOperator\id{id}

\allowdisplaybreaks

\newcommand{\ci}{c_{1}}
\newcommand{\bo}{{\rm O}}
\newcommand{\ds}{\displaystyle}
\newcommand{\dint}{\ds\int}
\newcommand{\dsum}{\ds\sum}

\newcommand{\eqskip}{ \vspace*{2mm}\\ }
\newcommand{\R}{\mathbb{R}}

\newcommand{\N}{\mathbb{N}}
\newcommand{\Z}{\mathbb{Z}}

\newcommand{\B}{{\rm B}} 

\newcommand{\Se}{\mathbb{S}}
\newcommand{\so}{{\rm o}}
\newcommand{\Cw}{{ C_W}}
\newcommand{\Cwnp}[2]{ C_{W #1,#2}}
\newcommand{\D}{\mathbb{D}}

\newcommand{\fr}[2]{\frac{\ds #1}{\ds #2}}
\newcommand{\pn}{p(n)}
\newcommand{\pin}[1]{p^*{#1}}
\newcommand{\pione}[1]{{p_1\!\!}^*{#1}}
\newcommand{\pnphi}{p_{\Phi}(n)}
\newcommand{\piphi}[1]{p^*_{\Phi}{#1}}
\newcommand{\lune}[2]{\mathbb{L}^{#1}_{#2}}
\newcommand{\luned}[2]{\raisebox{0.5pt}{${}$}\mathbb{L}^{#1}_{\pi/#2}}
\newcommand{\dbar}{d\hspace*{-0.08em}\bar{}\hspace*{0.1em}}

\newcommand{\modu}[1]{\;(\bmod\;{#1})}
\newcommand{\G}[1]{\mathcal{G}_{#1}}

\allowdisplaybreaks

\makeatletter
\@namedef{subjclassname@2020}{%
  \textup{2020} Mathematics Subject Classification}
\makeatother

\begin{document}

\title{Spherical  $n$-lunes: billiards and eigenvalues}

\author[P. Freitas]{Pedro Freitas}
\author[I. Salavessa]{Isabel Salavessa}
\address{Grupo de F\'{\i}sica Matem\'{a}tica, Instituto Superior T\'{e}cnico, Universidade de Lisboa, Av. Rovisco Pais, 1049-001 Lisboa, Portugal}
\email{pedrodefreitas@tecnico.ulisboa.pt}
\address{Grupo de F\'{\i}sica Matem\'{a}tica, Instituto Superior T\'{e}cnico, Universidade de Lisboa, Av. Rovisco Pais, 1049-001 Lisboa, Portugal}
\email{isabel.salavessa@tecnico.ulisboa.pt}

\date{\today}

\begin{abstract}
We characterise the periodic orbits of geodesic billiards on spherical lunes on $\Se^{n}$. In the case of angle openings
of the form $\pi/p$ for positive integer $p$ we fully determine their Dirichlet and Neumann spectra. We then show that
lunes with an angle opening smaller than $\pi$ which is not a rational multiple of $\pi$, or those with an angle opening
of the form $\pi/p$ for $p$ larger than one satisfy P\'{o}lya's conjecture eventually, independently of whether the
corresponding geodesic billiards satisfy the nonperiodicity condition or not. For lunes with an angle opening
$\pi/p$ we further provide a two-term asymptotic formula for the eigenvalues
based on sharp upper and lower bounds, together with a corresponding two-term counting function established using the
geoesic billiards approach. Finally,we give an explicit bound on $p$ in terms of the dimension ensuring the corresponding
lunes satisfy P\'{o}lya's conjecture for all eigenvalues.
\end{abstract}
\maketitle

\setcounter{tocdepth}{2}
  \tableofcontents

\section{Introduction}\label{INTROGERAL}
Let $\Se^n$ denote the $n-$dimensional sphere in $\R^{n+1}$
\[
 \Se^{n} = \left\{ x\in\R^{n+1}: \left\| x \right\| = 1\right\}
\]
with the canonical metric, and consider its open subsets consisting of the portion of
$\Se^n$ contained between two hyperplanes passing
through {the} origin. Denoting by $\theta$ the dihedral angle between these
hyperplanes, we will denote such a domain by $\lune{n}{\theta}$ and refer to it as the $n-$lune\footnote{In the
previous paper~\cite{fms} we referred to such domains as {\it wedges}, but {\it lunes} seems to be a more usual
designation.} with angle opening $\theta$. We are interested in the spectral properties of the Laplace-Beltrami
operator on such domains and their relations to certain properties of the corresponding geodesic billiards.
While these are specific domains, they display all sorts of possible extreme properties and are thus fundamental
in developing our understanding of spectral behaviour~\cite{horm,sava}.

To be more precise, let us consider the Dirichlet and Neumann eigenvalue problems for the Laplace-Beltrami
operator defined on a domain $\Omega\subset\Se^n$ with boundary $\partial \Omega$ by
\begin{equation}\label{eigprob}
\mbox{(D)}~ \left\{\begin{array}{ll}
\Delta u+\lambda u=0& x\in\Omega\\
u=0 & x\in\partial \Omega
\end{array}\right.
\quad\quad
\mbox{(N)}~\left\{\begin{array}{ll}
\Delta u+\mu u=0& x\in\Omega\\
\partial_{\nu}u=0 & x\in\partial \Omega.
\end{array}\right.
\end{equation}
The spectrum for such problems consists of eigenvalues with finite multiplicities which we will denote by
\[
 0 < \lambda_{1} \leq \lambda_{2} \leq \dots  \mbox{ and } 0=\mu_{0} \leq \mu_{1}\leq \dots,
\]
for Dirichlet and Neumann conditions, respectively, with both sequences converging to infinity as the order $k$
goes to infinity. Weyl's law for $\Omega$ reads as
\begin{equation}\label{weyl1}
\begin{array}{llll}
\lambda_{k}\left(\Omega\right) = \fr{4\pi^2 k^{2/n}}{\left(\omega_{n}|\Omega|\right)^{2/n}} + r_{D}(k) &
\mbox{ and } &
\mu_{k}\left(\Omega\right) = \fr{4\pi^2 k^{2/n}}{\left(\omega_{n}|\Omega|\right)^{2/n}} + r_{N}(k), & \mbox{ as } k\to\infty,
\end{array}
\end{equation}
for the Dirichlet and Neumann problems, respectively, where the remainder terms $r_{D}$ and $r_{N}$ are both $\so(k^{2/n})$ in this limit,
and $\omega_{n}$ denotes the
volume of the Euclidean $n-$ball of unit radius. Under certain geometric conditions, the second term in these
asymptotics is known to be of the form
\begin{equation}\label{2termweyl}
\lambda_{k}(\Omega) = \frac{\ds 4\pi^2k^{2/n}}{\ds \left(\omega_{n}|\Omega|\right)^{2/n}}
 + \frac{\ds 2\pi^2\omega_{n-1}|\partial\Omega|k^{1/n}}{\ds n \left(\omega_{n}|\Omega|\right)^{1+1/n}}+ {\rm o}(k^{1/n})
\end{equation}
and
\begin{equation}\label{2termweylNeumann}
 \mu_{k}(\Omega) = \frac{\ds 4\pi^2k^{2/n}}{\ds \left(\omega_{n}|\Omega|\right)^{2/n}}
 - \frac{\ds 2\pi^2\omega_{n-1}|\partial\Omega|k^{1/n}}{\ds n \left(\omega_{n}|\Omega|\right)^{1+1/n}}+ {\rm o}(k^{1/n}).
\end{equation}
The conditions for which these two-term asymptotics hold are related to the trajectories of the geodesic billiard
defined on $\Omega$, and are known in the literature as {\it nonperiodicity} and {\it nonblocking}
conditions~\cite{sava}. The first of these,
which will be one of the objects of our focus in this article, basically states that the set of periodic orbits of
the corresponding billiard has measure zero within the the set of all trajectories -- see~\cite{sava} for the
details.
These two-term asymptotic formulas then establish a relation between the nonperiodic behaviour of the geodesic
billiard and the asymptotics of eigenvalues. In particular, we see from~\eqref{2termweyl}
and~\eqref{2termweylNeumann} that, since the second term is of definite sign, a domain satisfying these conditions
will, for all {\it sufficiently large} $k$, have its eigenvalues above or below the first term in the Weyl
asymptotics for the Dirichlet and Neumann problems, respectively. This property relates to a famous open problem
in spectral theory, namely, the P\'{o}lya conjecture~\cite{poly1,poly2}, which states that the Dirichlet and
Neumann eigenvalues of the Laplacian on Euclidean domains are {\it always} above and below, respectively, the first
term in the Weyl asymptotics. More precisely,
\begin{equation}\label{polyaconject}
 \mu_{k}\left(\Omega\right) \leq \frac{\ds 4\pi^2k^{2/n}}{\ds \left(\omega_{n}|\Omega|\right)^{2/n}} \leq
 \lambda_{k}\left(\Omega\right), \hspace*{5mm} \mbox{ for all } k\in\N.
\end{equation}
We see that in the case where formulas~\eqref{2termweyl} and~\eqref{2termweylNeumann} hold, then P\'{o}lya's
conjecture is satisfied provided that $k$ is sufficiently large, say $k\geq k^{*}=k^{*}\left(\Omega\right)$ or,
as referred to in~\cite{fs}, that P\'{o}lya's conjecture is satisfied eventually.

In spite of this close connection between the existence of two-term asymptotics of the form~\eqref{2termweyl}
and~\eqref{2termweylNeumann}, and eigenvalues of the Laplacian satisfying P\'{o}lya's
conjecture~\eqref{polyaconject}, there exist (non-Euclidean) examples where the two-term
asymptotics~\eqref{2termweyl} and~\eqref{2termweylNeumann} do not
hold but P\'{o}lya's conjecture is satisfied, and examples for which the two-term asymptotics hold but the
conjecture is not satisfied. The $n-$hemisphere is an example of the former situation, where all geodesic orbits are
periodic but the conjecture is satisfied when $n$ equals two for both the Dirichlet and Neumann cases, while it
fails for the Dirichlet problem and still holds for the Neumann problem in higher dimensions~\cite{bb,blps,fms} --
due to this fact, and as explained below, in this article we will concentrate mostly on the Dirichlet problem. For
an example of the latter situation, consider a cylindrical surface of the form $\mathbb{S}^{1}\times(0,h)$ with $h$
sufficiently large, for which~\eqref{2termweyl} holds but where there exist low Dirichlet eigenvalues not
satisfying~\eqref{polyaconject}~\cite{fs} -- see also~\cite{frwa} for
a precise characterisation of the values of $h$ for which the conjecture holds.

It follows that for domains which do not satisfy the nonperiodicity condition but for which P\'{o}lya's conjecture
is satisfied, for all $k$ or eventually, this indicates that the remainder terms $r_{D}$ and $r_{N}$
in~\eqref{weyl1} are of one sign, for all $k$ or for sufficiently large $k$, respectively.

The main purpose of this article is to study this relation between the nonperiodicity condition and P\'{o}lya's
conjecture in the case where the  domain $\Omega$ is an $n-$lune, as described above. This is a natural continuation
of two previous papers, namely,~\cite{fms} and~\cite{fs}. In the former we studied this relation in the case of
hemispheres for which all orbits are periodic. In the case of Dirichlet boundary conditions we proved that in
dimensions higher than two the conjecture will, in fact, always fail for an infinite number of eigenvalues. This was
done by proving sharp inequalities satisfied by the eigenvalues and thus measuring in a precise way their
deviation from the first term in the Weyl asymptotics. More precisely, we proved the following two-sided
inequalities
\begin{equation}\label{uplowfms}
\hspace*{5mm} \left(n!\right)^{2/n}k^{2/n} -\fr{(n-1)(n-2)}{6}\leq \lambda_{k}\left(\Se^n_+\right) \leq
 \left(n!\right)^{2/n}(k-1)^{2/n} + 2\left(n!\right)^{1/n}(k-1)^{1/n}+n,
\end{equation}
valid for all positive integer $k$ and with both inequalities being asymptotically sharp, {in the sense that
in both cases there are sequences of eigenvalues behaving asymptotically like the first two terms in the given
bounds} -- see Theorems~B and~C in~\cite{fms} for details. This implies that the remainder term in the corresponding
Weyl asymptotics~\eqref{weyl1} which, in this case, reads as
\[
 \lambda_{k}\left(\Se^n_+\right) = \left(n!\right)^{2/n}k^{2/n} + r_{D}(k),
\]
is not symmetric in the sense that while it remains bounded from below, there is a strictly increasing sequence of
positive integers $k_{j}$ for which $r_{D}\left(k_{j}\right)\to +\infty$ as $j$ goes to infinity.  Moreover, since
the constant term in the lower bound is negative, this means P\'{o}lya's conjecture in these cases is not satisfied
eventually.

In the case of Neumann boundary conditions, we showed that P\'{o}lya's conjecture is satisfied for
$\Se_{+}^{n}=\lune{n}{\pi}$ in all dimensions -- see also~\cite{blps}. Because of this, it also follows that it is satisfied
for lunes with angles $\pi/p$, for integer $p$ larger than $2$.

In~\cite{fs} we showed that whenever P\'{o}lya's conjecture is satisfied eventually for a domain
$\Omega$, and for each positive integer $p$ it is possible to subdivide $\Omega$ into $p$ isometric copies which
tile $\Omega$, with $p$ allowed to become arbitrarily large, then there is an order $p_{0}$ after which these
copies will satisfy P\'{o}lya's conjecture~\cite[Lemma 1]{fs}. Hemispheres would be natural domains to apply this
result to, as $\Se_{+}^{n}$ can be divided into an arbitrarily large number of lunes of the form $\lune{n}{\pi/p}$.
However, and as we saw above, when $n$ is larger than two P\'{o}lya's conjecture is not satisfied eventually in the
Dirichlet case, while since the conjecture is satisfied for the Neumann problem it follows that it will also be
satisfied for any tiling of a hemisphere~\cite{fms}. Here we shall thus focus our study of  P\'{o}lya's conjecture on
the Dirichlet case for lunes $\lune{n}{\theta}$ with $n\geq 2$,  when either $\theta/\pi$ is irrational or
$\theta =\pi/p$ with $p\in \N_2$. In the former case we use the fact that the spectrum of $\lune{n}{\theta}$
satisfies~\eqref{2termweyl}, while in the latter we show that all geodesic orbits are periodic, and proceed by
explicitly determining the corresponding spectra using invariant theory with respect  to a group of reflections.

In order to proceed as described, we begin by studying the geodesic billiards on lunes to determine whether or not
the set of periodic orbits has measure zero. In contrast with convex spherical caps, for which the nonperiodicity
condition is always satisfied~\cite[Example 1.3.16]{sava}, here we see that this is never the case for angle openings
which are rational multiples of $\pi$. Note, however, that when this angle is smaller than $\pi$, these lunes
are also convex, thus providing further examples of convex domains not satisfying the nonperiodicity condition.
\begin{thmx}[Lune billiards]\label{thmx_blunes}
 On a lune $\lune{n}{\theta}$ all geodesic billiard trajectories are periodic when $\theta$ is {a rational multiple of
 $\pi$}. All these orbits have (minimal) period $2p$ {when $\theta=m \pi/p$ with $m,p \in \N$ relative prime}, except
 for the (single) trajectory which is orthogonal to all meridians which has period two. When $\theta$ is not a rational
 multiple of $\pi$ there are no periodic orbits except again for the single $2-$periodic orbit orthogonal to all
 meridians.
\end{thmx}
\begin{remark}
 Although most of the terms used above are fairly standard, for the benefit of the reader we define them in
 Section~\ref{lunedefinition}, while a more detailed version of the theorem is given in Theorem~\ref{thmA}.
\end{remark}

This immediately implies that (convex) lunes whose angle opening is an irrational multiple of $\pi$
will satisfy P\'{o}lya's conjecture eventually. More surprisingly, it turns out that this is also the case for
lunes whose opening angle is of the form $\theta=\pi/p$ with $p$ an integer larger than two -- we recall that
this was not the case for hemispheres ($p=1$). We shall refer to these by dihedral lunes, as they are associated
with the dihedral group $\mathbb{D}_p$.
\begin{thmx}[Generalised two-term asymptotic formula]\label{thmxeventually}
 On the dihedral lune $\lune{n}{\pi/p}$  $(p\in \N)$, the generalised $2$-term asymptotic formula
 \[\lambda_k= \left(n!\ p\right)^{2/n} k^{2/n}+  c(p,k)k^{1/n}+\so(k^{1/n}),\; \mbox{ as } k\to +\infty\]
 holds for $\lambda_k$ in the $K$-chain of the eigenvalue $K(K+n-1)$ with
 \[
 c(p,k)=\left(n!\ p\right)^{1/n}p+ \left(n!\ p\right)^{1/n}\left(1-2 \frac{k-k_-}{m_{n,p}[K(K+n-1)]-1}\right).
 \]
 Here $m_{n,p}[K(K+n-1)]$ denotes the multiplicity of $\lambda_{k}$, and $k_-$ is the lowest order of the corresponding chain.
 Moreover, $c(p,k)$  is sharp and bounded, with values in each chain between the corresponding minimum $\left(n!\ p\right)^{1/n}(p-1)$
 and maximum $\left(n!\ p\right)^{1/n}(p+1)$, taken at the highest and lowest orders $k=k_{+}=k_{-}+m_{n,p}[K(K+n-1)]-1$ and
 $k=k_-$, respectively.
\end{thmx}
\begin{remark}
In Appendix~\ref{ApC} we derive the corresponding formula for the associated counting function, following the approach in~\cite{sava} -- see, in particular, Proposition~\ref{2termsava}.
\end{remark}
As a consequence of Theorems~\ref{thmx_blunes} and~\ref{thmxeventually} we obtain that all lunes with angle
openings smaller than $\pi$ satisfy P\'{o}lya's conjecture eventually, independently of the dimension.
\begin{corx}[P\'{o}lya's conjecture for large orders] \label{CorC}
 Let $\lune{n}{\theta}$ be a lune with angle opening $\theta$ smaller than $\pi$ {that is either an
 irrational multiple of $\pi$ or of the form $\pi/p$ with $p\in \N_2$}. Then its Dirichlet spectrum
 satisfies P\'{o}lya's conjecture eventually.
 \end{corx}
As pointed out above, this implies that the remainder term $r_{D}$ in~\eqref{weyl1} is non-negative for these orders.
A direct application of~\cite[Lemma 1]{fs} then yields that for each $n$ there exists sufficiently large $p$ such
that $\lune{n}{\pi/p}$  satisfies P\'{o}lya's conjecture. We shall now provide a quantitative version of this result.
\begin{thmx}[P\'{o}lya's conjecture for all orders]\label{thmxquantitative}
 Given a positive integer $p$, the lune $\lune{n}{\pi/p}$ satisfies P\'{o}lya's conjecture for $3\leq n \leq
 8$ if and only if $p\in\N_{2}$. For $n\geq 9$, P\'{o}lya's conjecture holds if $p\geq n-1$, {while
 it fails if $p<\fr{-1+\sqrt{1+4e^{-2}}}{2}n \approx 0.120754n$.}
\end{thmx}
\begin{remark}
  Although the linear order of growth on the dependence of $p$ so that $\lune{n}{\pi/p}$ satisfies the
  conjecture is sharp, the actual values for which this happens may be improved. In
  Theorem~\ref{thmD} we show how this can be done and illustrate it by providing optimal quantitative results for dimensions up to $50$.
  From these examples we also see that it is possible for the conjecture to fail for a low eigenvalue other than the
  first (and doing so only for a finite number of eigenvalues), as is the case of $\luned{24}{4}$, for instance -- see Remark~\ref{remark_examples1}.
  A similar effect may also be observed in the case of a cylindrical strip~\cite{frwa}. This means that the threshold for the
  conjecture not to hold given above based on what happens to the first eigenvalue will not necessarily be optimal in general.
\end{remark}

The proofs of both Theorems~\ref{thmxeventually} and~\ref{thmxquantitative} rely partially on the explicit
knowledge of the Dirichlet eigenvalues of dihedral lunes, that is, those with $\theta=\pi/p$ for $p\in \N$
-- see Definition~\ref{lunedef} below.
In this case any eigenvalue of $\lune{n}{\pi/p}$ is also an eigenvalue of $\Se^{n}$, albeit with a different multiplicity
which will depend on $n$ and $p$. To the best of our
knowledge, the only case which had been studied in the literature was for lunes in $\Se^{2}$~\cite{gr}, and so we
provide a full description of the spectra of $\lune{n}{\pi/p}$ here. This is done by using the Hilbert--Poincar\'{e}
series of graded algebras of anti-invariant homogeneous polynomials with respect to dihedral
groups of reflections, following the method applied in~\cite{bb} to the particular case of $\Se^2_+$.

\begin{thmx}[Dihedral lune eigenvalues]\label{thmxmultiplicity} Let $n\in \N_2$ and $p\in \N$. The Dirichlet eigenvalues of $\lune{n}
{\pi/p}$ are of the form $K(K+n-1)$ with $K\geq p$, thus forming a proper subset of those of $\Se^n$. The corresponding multiplicities
are given by
\[m_{n,p}[K(K+n-1)]=\sum_{s=0}^m\frac{(sp+r+1)^{\overline{n-2}}}{(n-2)!},\]
{where $m = \left\lfloor\fr{K}{p}\right\rfloor-1$, $r=K\bmod p$, and the rising factorial $q^{\overline{t}}$ is defined by
$q(q+1)\cdots(q+t-1)$.}
\end{thmx}

The organisation of the paper is as follows. In the next section we provide the notation and definitions for lunes and spheres.
This has turned out to be more involved that expected, as there are many concepts such as meridians, and  polar and equatorial spheres,
for instance, that do not seem to have fully standard definitions in the literature. We have thus taken some care to set up clearly
in this section what the definitions we are using are. After this, we can then define the relevant billiard (geodesic) trajectories
on lunes and the related concepts that will be used in Section~\ref{proofbill} to prove Theorem~\ref{thmx_blunes}. We then
determine the eigenvalues and corresponding multiplicities of the lunes $\lune{n}{\pi/p}$. This again requires setting up several concepts
and techniques in order to proceed and associate to each lune with opening angle $\pi/p$ its eigenvalues and the corresponding
multiplicities. Once these are in place, we are ready to prove Theorem~\ref{thmxmultiplicity}, which we do in Section~\ref{Sec 4}.
In Section~\ref{introEIGEN} we derive conditions for P\'{o}lya's conjecture to hold on dihedral lunes, leading to the proofs
of Theorems~\ref{thmxeventually} and~\ref{thmxquantitative} in Section~\ref{Sec Third term}. 

As was already apparent from the proofs in~\cite{fms}, in order to derive inequalities such as~\eqref{uplowfms} or, in the
present paper, Theorem~\ref{thmxquantitative}, quite sharp algebraic inequalities are needed. These include upper and lower
bounds such as those given in~\cite[Lemma~A.1]{fms} for the rising factorial, which coincide with the first three and two terms in
the corresponding asymptotic expansions, respectively. Here, and apart from also using this inequality, we again needed new sharp
inequalities such as that given in Lemma~\ref{lmx1}. We believe this to be a feature of the spectrum of these problems, which may be
interpreted not only as a consequence of their large intrinsic eigenvalue multiplicities, but also as being indeed limiting geometric
domains where certain properties hold -- recall, for instance, that an infinite number of eigenvalues in the two-hemisphere give equality
in P\'{o}lya's conjecture.

In Appendix~\ref{ApA} we collect several results about reflection groups and Hilbert-Poincar\'{e} series which are used to calculate
the eigenvalue multiplicities of dihedral lunes in Section~\ref{Eigen}. In Appendix~\ref{Sec 8} we present the calculations leading
to the optimal values of $p$ for which $\lune{n}{\pi/p}$ satisfies P\'{o}lya's conjecture for $n$ up to fifty. Appendix~\ref{ApC}
follows the approach in~\cite{sava} to determine the two-term asymptotic formula for the counting function of dihedral lunes, where
the nonperiodicity condition is not satisfied. Finally, in Appendix~\ref{prooflemma64} we state and prove Lemma~\ref{lmx1}
which, as mentioned above, is instrumental in the proof of Theorem~\ref{thmxquantitative}.

\section{Lune billiards}
\subsection{Notation and definitions}\label{lunedefinition}
 The $n$-dimensional lune on the $n$-sphere $\Se^n$ is defined by two
 hyperplanes of $\mathbb{R}^{n+1}$ whose intersection with $\Se^n$ defines the boundary of the lune.  For
 definiteness, we consider the angle between the two  {defining hyperplanes} as being in the last two
 variables.
 \begin{definition}[Lune of opening angle $\theta$]\label{lunedef}
 Given $n\in\N$ and $\theta\in (0,2\pi]$, an $n-$lune is defined by
  \begin{equation}
 \lune{n}{\theta} = \left\{~(w,r\cos\alpha, r\sin\alpha)\in \Se^n:\, w\in \R^{n-1},\, |w|\leq 1,\,
  r=\sqrt{1-|w|^2},\, 
   \alpha\in [0,\theta]~\right\}. \label{lune}
  \end{equation}
 We refer to the angle $\theta$ as the angle opening of the lune. According to whether $\pi/\theta$ is
 rational or an irrational number, we refer to $\lune{n}{\theta}$ as a rational or an irrational $n$-lune,
 respectively. In the former case, if $\pi/\theta$ is an integer, we call these dihedral lunes.
  \end{definition}
\subsubsection{Meridians, the polar and equatorial spheres, and the equatorial geodesic}
  The boundary of $\lune{n}{\theta}$ is the (piecewise smooth) union of two hemi-hyperspheres of $\Se^n$ given by the equations
  $\alpha=0$ and $\alpha=\theta$,  and we take these as models for {a family of $(n-1)$-dimensional meridians
  of $\Se^n$ invariant by rotations of $\R^{n+1}$ that fixes an axis space of dimension $n-1$.}
\begin{definition}[Meridians]\label{(n-1)meridian}
Given $\theta\in {[0,2\pi]}$, we define the corresponding $(n-1)$-dimensional meridian of
$\Se^n$ by
\begin{equation}\label{partial}
\partial_{\theta}=  H_{\theta}^+\cap \Se^n
= \left\{ (w, r\cos\theta, r\sin\theta): w\in \R^{n-1}, |w|\leq 1,  r=r(|w|)=\sqrt{1-|w|^2}\right\}, 
 \end{equation} 
where the half-hyperplane $H^+_{\theta}=\R^{n-1} \times\R_0^+[(\cos\theta, \sin \theta)]$ is said to support the
meridian.
\end{definition}
In this way, the  boundary of $\lune{n}{\theta}$ consists of the union of two meridians $\partial_0\cup \partial_{\theta}$.
Two distinct meridians will intersect at an $(n-2)-$sphere and this intersection is, in fact, independent of the meridians
considered and we shall call it the polar sphere.
\begin{definition}[Polar sphere] Denote the union of two antipodal meridians by the \em meridian  $(n-1)$-sphere \em
$\Se^{n-1}_{\theta}:=\partial_{\theta}\cup\partial_{\theta+\pi}$ with support $H_{\theta}\cap \Se^n.$
We shall call the \em  space $\R^{n-1}\times \{(0,0)\}\subset \R^{n+1}$ the  $(n-1)$-dimensional \em polar axis.
 Given any two distinct angles $\theta$ and $\phi$, we define the polar $(n-2)-$sphere by
\begin{equation}\label{POLE}
\Se^{n-2}_{polar}:=
\partial_{\theta}\cap \partial_{\phi}
=\Se^{n}\cap (\R^{n-1}\times
 \{(0,0)\})= \bigcap_{\theta\in \R} \Se^{n-1}_{\theta}.
 \end{equation}
\end{definition}
It follows that the polar sphere consists on the intersection of all meridians and is also the intersection of
the polar axis space with $\Se^n$. In this setting, and provided the opening angle of the lune is not $\pi$, the
two meridians defining the boundary of the lune intersect at the polar sphere of $\Se^{n}$ which then constitutes
the singular set of the boundary.

Note also that $\partial_{\theta}=R_{\theta}\partial_0$ where $R_{\theta}$ is the rotation of angle $\theta$ around
the $(n-1)$-dimensional polar axis -- see~\eqref{RSphi} below.

Our last definition of subsets of $\Se^{n}$ is the equatorial sphere.
\begin{definition}[Equatorial sphere] A point $u=(w,0,0)$ in $\Se^{n-2}_{polar}$ together with its
 antipodal point $-u$ define a polar line axis $\R u\subset \R^{n-1}\times\{(0,0)\}$ and an
 \em equatorial $(n-1)$-sphere \em given by 
$\Se^{n-1}_{equator}(u):=(\R u)^{\bot}\cap \Se^n$.
The intersection of all such equatorial spheres is a one-dimensional sphere  orthogonal to all meridians, defining the
(positively oriented) equatorial geodesic
 \begin{equation}\label{equatgeod}
 \gamma_{e}(t)=(0^{n-1}, \cos t, \sin t).
 \end{equation}
\end{definition}
 
\begin{remark} A strongly convex domain $M$ \cite[Definition 1.3.14 ]{sava} is given by the condition ${\bf k}:=-\ddot{x}_n|
 _{t=0}> 0 $,
where $x_n=d(x,\partial M)$. Lunes are convex if $\theta<\pi$ but not strongly convex because 
 the boundary is piecewise totally geodesic, and so in this case  $x_n(t)=d(x(t),\partial M)=0$ for 
 any  ray $x(t)$
 tangent to the boundary.
\end{remark} 

\subsubsection{Isometries fixing the polar axis space $\R^{n-1}\times\{(0,0)\}$.} 
The unit normal to
$\partial_{\theta}$ at $b(w)$ is given by the constant vector,
\begin{equation}\label{normal}
  \nu_{\theta}= \nu_{\theta}(b(w))=\left(0^{n-1},\cos\left(\frac{\pi}{2}+\theta\right), \sin\left(\frac{\pi}{2}+\theta\right)\right)~\bot~~ H_{\theta}.
  \end{equation}

In the Euclidean plane let $R^{\phi}$ and $S^{\phi}$ be a counter-clockwise rotation by an angle $\phi$ and the reflection with respect to an axis
making an angle of $\phi/2$ with the positive part of the horizontal axis, respectively. Writing these transformations as
\[
R^{\phi}=\left[\begin{array}{cc}
\cos \phi &-\sin \phi\\ \sin\phi&\cos\phi\end{array}\right], \quad  S^{\phi}=\left[
\begin{array}{cc}\cos \phi &\sin \phi\\ \sin\phi&-\cos\phi\end{array}\right],
\]
we now extend them respectively as a rotation and a reflection to $\R^{n+1}$, as
\begin{equation}\label{RSphi}
R_{\phi}=\left[\begin{array}{cc}
\mathbbm{1}_{n-1}&0\quad 0\\ 
\begin{array}{c} 0\\0\end{array}& R^{\phi}\end{array}\right], \quad\quad
S_{\phi}=\left[\begin{array}{cc}
\mathbbm{1}_{n-1}&0\quad 0\\ 
\begin{array}{c} 0\\0\end{array}& S^{\phi}\end{array}\right],
\end{equation}
where $\mathbbm{1}_{n-1}=\id_{\R^{n-1}}$. The polar axis space $\R^{n-1}\times\{(0,0)\}$ is invariant by both $R_{\phi}$ and $S_{\phi}$,
while the usual properties of rotations and reflections in the plane extend to the transformations $R_{\theta}$ and
$S_{\theta}$ in $\R^{n+1}$. We will, in particular, make use of, 
\begin{gather}\label{ALI2}
\begin{array}{ll}
R_{\theta}\partial_0=\partial_{\theta}, &{S_{2\theta}}_{| \partial_{\theta}}=\id_{|\partial_{\theta}}\\
(S_{0\theta}S_{2\theta})^{\ell}=R_{-2\ell\theta},
&(S_{2\theta}S_{0\theta})^{\ell}=R_{2\ell\theta}.
\end{array}
\end{gather}
Note further that
\begin{equation}\label{R2redution}
S_{2\theta}(w,x,y)= (w, S^{2\theta}(x,y))= s_{\nu_{\theta}}(U)=U-2\langle U,\nu_{\theta}\rangle \nu_{\theta}
\end{equation}
where  $s_{\nu_{\theta}}(U)$
is the reflection  in $\R^{n+1}$ with respect to  $H_{\theta}$. Here, and in what follows, $\langle \cdot,\cdot \rangle$ denotes
the standard inner product in Euclidean space.
 \subsection{Billiard {trajectories}}
A geodesic billiard trajectory on a  domain $\Omega$ with a piecewise smooth boundary  of an $n$-dimensional
Riemannian manifold is a continuous curve $T(t)$ composed by unit-speed geodesic segments that hit the $(n-1)$-dimensional boundaries
transversely, satisfying the reflection principle, that is, the incidence and reflection angles are equal at the point where the
trajectory hits the boundary.
If it hits a corner of the boundary by convention {(classical case)} the trajectory is assumed to stop at this point~\cite{g}.
Trajectories that hit the boundary  tangentially at some point are non-transversal
and called grazing  \cite[Section 1.3]{sava}.

A trajectory that performs an infinite number of transversal reflections in finite time is called a dead-end
trajectory.
If the boundary is smooth, the set of these trajectories has measure zero, in the sense that the set of starting points  is
a set of measure zero in the co-sphere bundle~\cite[Lemma 1.3.11 and Appendix D]{sava} (See Remark~\ref{zeromeasure}~1).

\begin{definition}\label{periodic}
A trajectory $T(t)$ on a domain with boundary defined by an initial geodesic $\gamma(t)$, with $T(t_0)=\gamma(t_0)$ a boundary
point, is said to be {\it periodic of period  $q$} if it is a closed trajectory with $q$ geodesic segments, defining an oriented
closed {(geodesic)} polygon, possibly with self-intersections~\cite{g}. Equivalently, there are $q$ boundary collisions
such that the last collision point at $t=t_q$  equals the initial point $T(t_0)$ and the angles at $T(t_q)=T(t_0)$ of the
boundary {with the last and initial geodesic segments} also satisfying the reflection principle.The usual period in the variable
$t$, $t_q-t_0$,  is named {\it the $t$-period of $T(t)$}.
\end{definition}

Domains for which the sets of dead-end trajectories and of periodic trajectories have zero measure within the set of
all trajectories are said to satisfy the nonblocking  and nonperiodicity conditions,
respectively~\cite[Definition 1.3.22]{sava}. If the two conditions are
satisfied and if the boundary  is piecewise smooth with a finite number of connected $(n-1)$-dimensional
  components  making an angle $\theta_i\in (0, \pi)$ at corner boundary points  of dimension less than or equal to $(n-2)$,
  then the two-term Weyl asymptotic formula (\ref{2termweyl}) for the Dirichlet eigenvalues $\lambda_k$  holds
 -- see~\cite[Theorem~1.6.1 and Example~1.6.16]{sava} and~\cite{va} for the smooth and piecewise smooth cases, respectively. 

 Although lunes do not have a smooth boundary, they will satisfy the conditions in~\cite{va} provided their opening
 angle is smaller than $\pi$. In Theorem~\ref{thmA} we show that for rational lunes the nonblocking condition is satisfied
 but all trajectories are periodic, while {for irrational lunes } both the  nonperiodicity and nonblocking conditions are satisfied. In particular, if $\pi/\theta$ is an irrational number larger than one,~\eqref{2termweyl} holds for Dirichlet eigenvalues, extending the case $n=2$ obtained by Gromes~\cite{gr} -- note, however, that his formula  holds for any irrational $n$-lune with $\theta<2\pi$.

 A trajectory  of $\lune{n}{\theta}$ is defined by an initial unit-speed geodesic $\gamma(t)$ of $\Se^n$. Without loss of generality, we always assume
 $T(t)$ starts at $t_0$ with $T(t_0)=\gamma(t_0)\in \partial_0$. The condition that $T(t)$ is transversal to the $(n-1)$-dimensional boundary in all its
 segments turns out to be equivalent  to the  {initial geodesic  $\gamma$  not being  part of any meridian sphere or, equivalently, not crossing}
 the polar sphere.
 This is a consequence from the following observation on reflected geodesics at meridians.
  
Unit speed geodesics of $\Se^n$ are one-dimensional spheres $\Se^1$. Given a point $\gamma(t_{0})$ on $\mathbb{S}^{n}$ and
 a unit (direction) vector $\gamma'(t_{0})$, these trajectories are of the form $\gamma(t)=\cos(t-t_0) \gamma(t_0)+\sin(t-t_0) \gamma'(t_0)$.
 Then $\gamma(t+\pi)$ is the antipodal point $-\gamma(t)$,
 $\gamma'(t)=\gamma(t+\pi/2)$, and  $\gamma''(t)=-\gamma(t)$. The latter identity is equivalent to the geodesic equation
 $\nabla_{d/dt}(\gamma'{(t)})=(\gamma''(t))^{\top}=0$, with respect to the Levi-Civita connection $\nabla$ of $\Se^n$.
 {As we shall see in} Lemma~\ref{rigidity} below, if $\gamma$ does not cross the polar sphere, then $\gamma$ crosses meridians at a
 single point, {implying the existence of a first point $t_{1}$ in $(t_0, t_0+\pi]$ for which $\gamma(t_1)\in \partial_{\theta}$.}
 In particular,  by choosing the \em right orientation, \em  we may always assume the first segment of $\gamma(t)$ stays inside
 $\lune{n}{\theta}$ for $t\in (t_0,t_1)$.
  Since $S_{2\theta}=s_{\nu_{\theta}}$ as described by~\eqref{R2redution} is the reflection of $\R^{n+1}$ with respect to $H_{\theta}$,
  then $\gamma(t)$ reflects at $t=t_1$  as a unit-speed geodesic {$\gamma_{r}$} with initial conditions
$ \gamma_r(t_1)=\gamma(t_1)\in \partial_{\theta}$, and  $\gamma'_r(t_1)=
S_{2\theta}\gamma'(t_1)$.
  We write
 \begin{gather}\label{gammar1}
 \gamma_r(t)=_{t_1} S_{2\theta}\gamma(t), \quad \forall t\in {[t_1, t_1+0^+)}, 
 \end{gather}
 and have that $\gamma_{r}$ satisfies
\begin{equation}\label{Vr}
\langle \gamma'_r(t_1),\nu_{\theta}\rangle= -\langle \gamma'(t_1), \nu_{\theta}\rangle=:-\cos a=\cos(a+\pi).
\end{equation}
The angle $a$ is the  angle of incidence and reflection of the billiard trajectory at the hitting point $\gamma(t_1)$, that by construction
satisfies   $0\leq a\leq \pi/2$, and by the  transversality property it satisfies $0\leq a<\pi/2$. From~\eqref{Vr} we see that $\gamma_r'(t_1)$
is directed inwards, which means that  $\gamma(t_1^-):=\lim_{\tau\to 0^+}\gamma(t_1-\tau)$ and $\gamma_r(t_1^+):=\lim_{\tau\to 0^+}\gamma(t_1+\tau)$ are to the same side of the meridian $ \partial_{\theta}$,
{and we may repeat the same argument as before.}
 Hence,  the  transversality condition of $\gamma$  with respect to meridians, or equivalently, the  non-intersection with the polar sphere, is
 preserved by the reflection with respect to a meridian, that is, $\gamma_r$ satisfies the same condition {(see Lemma~\ref{elementar}~4))}.
\begin{definition}\label{admissibility}
An \em admissible \em  initial geodesic of a billiard trajectory on a lune is  a unit-speed geodesic pointing to the interior of the lune and which
does not intersect the polar sphere.
\end{definition}
\begin{remark}\label{zeromeasure}
1)  Each unit-speed geodesic $\gamma$ of $\mathbb{S}^n$ can be identified with the oriented two-dimensional vector subspace $\Pi_{\gamma}=\mathrm{Span}\{\gamma(t), \gamma'(t)\}$
for any chosen $t$, defining its orientation
as an  element of $\tilde{G}_2(n+1)=SO(n+1)/(SO(n-1)\times SO(2))$. This Grassmannian manifold, compact and of dimension $2(n-1)$, represents
\em the manifold of oriented geodesics \em  $\mathcal{G}(\Se^n)$ of $\Se^n$~\cite[Example~2.8]{Be}. 
 Moreover, the   unit tangent bundle of $\Se^n$,
 $U\Se^n= \{(x,v): x,v\in \Se^n, v\bot x \}$
 can be represented as the homogeneous space  $SO(n+1)/SO(n-1)$, and the map $P:U\Se^n\to \mathcal{G}(\Se^n)$,
 $P(x,v)= \Pi_{\gamma_{(x,v)}}$, where $\gamma_{(x,v)}(t):=\cos t \, x +\sin t\, v$, defines $U\Se^n$ as a principal bundle with structural
 group $\Se^1=SO(2)$ that acts as a free action of $\R$ on $U\Se^n$  by the geodesic flow, defining   $U\Se^n/\Se^1$  isometric tn
 $\mathcal{G}(\Se^n)$~\cite[2.5]{Be}, with $\Se^1$-invariant fibres diffeomorphic to $\Se^1$. Locally, on a chart $\Omega\subset \mathcal{G}(\Se^n)$,  $P^{-1}\Omega$
 is  diffeomorphic to  $\Omega \times \Se^1$, and so   $P^{-1}Y$ diffeomorphic to $Y\times \Se^1$ has measure zero in $U\Se^n$ if and only if
 $Y\subset \Omega$ has measure zero.  The co-sphere bundle $S^*\Se^n$ of the initial points  of the Hamiltonian geodesic flow is used to define
 the measure of subsets of trajectories in \cite{sava}. Hence, sets of geodesics with zero measure in $\mathcal{G}(\Se^n)$ define sets of trajectories
 with zero measure in $S^*\Se^n$ and vice-versa. {For example, the set of unit-speed  geodesics of $\Se^n$ that intersect the polar sphere has Hausdroff dimension {a.e.} $(n-2)\times 2 + 1$, and so of measure zero on $\mathcal{G}(\Se^n)$ (Lemma~\ref{elementar}~3). With no loss of generality
 we consider trajectories of a lune of opening angle $\theta$ starting at $\partial_0$, that covers the cases starting at $\partial_{\theta}$
 by using rotations or reflections that fix the polar axis. Moreover we may also assume $t_0=0$. In this case,  if we take  $\tilde{\gamma}(t)=\gamma(-t)$
 then $\Pi_{\tilde{\gamma}}$ and $\Pi_{\gamma}$ span the same 2-plane but with the opposite orientations. Trajectories with admissible initial geodesic
 starting at an interior point of the lune are uniquely defined by any  hitting boundary point. }\\[1mm]
 2) On a rational lune $\lune{n}{m \pi/p}$ all admissible trajectories but one
  are periodic  of $t$-period $\mathbb{T}=2m\pi $  (Theorem~\ref{thmA}).
Any smooth segment of $T$  may be parameterised as  $T(t)=\cos(t-\mathbb{T})y+\sin(t-\mathbb{T})\eta^{\sharp} $ for $t\in (\mathbb{T}-\epsilon, \mathbb{T}+\epsilon)$, 
where $(y, \eta)$ belongs to $S^*\mathbb{S}^n$
and $T(\mathbb{T})=y$ is not a boundary point of the lune. Set $(x^*(t; y,\eta),\xi^*(t;y, \eta))
=(T(t),T'(t)^{\flat})$ ( $\sharp$ and $\flat$ are the isometric vector bundle  musical isomorphisms). Then $ F(t,y, \eta):=|x^*(t; y,\eta)-y|^2 + |\xi^*(t;y,\eta)-\eta|^2$
is given by $ 4(1-\cos(t-\mathbb{T}))$.   Therefore,  $F(\mathbb{T},y,\eta)$ has an infinite order zero at $(y,\eta)$.
This means  $T(t)$ is \em absolutely periodic \em in the sense given in \cite[pag.\ 229]{va} (\cite[Definition 1.3.2.]{sava} in the no boundary case).
In Theorem~\ref{thmA} we choose the initial $T(t_0)$ belonging to the boundary $\partial_0$.
\end{remark} 
Unit-speed spherical geodesics satisfy the geodesic equation $\gamma''(t)=-\gamma(t)$,  and
are thus of the form  $\gamma(t)=\cos(t-t_0)\gamma(t_0)+\sin(t-t_0)\gamma'(t_0)$.
We {also} use the splitting $\R^{n+1}= \R^{n-1}\times \R^2$, giving a representation in the form
\begin{equation}\label{representation}
\gamma(t)=(w(t), f(t)e^{i\phi(t)}),
\end{equation}
with $w(t)=P_{n-1}(\gamma(t))$ and $f(t)e^{i\phi(t)}=P_2(\gamma(t))$ smooth real functions, where $P_m$ is the orthogonal projection on $\R^m$, $f(t)\geq 0$, and  $\phi(t)\in \R$. Then $f^2= 1-|w|^2$, and $\phi$ is not well defined at points $t$ such that $f(t)=0$.
Nevertheless, for all $t$, $\gamma(t)\in {\partial_{\phi(t)}}$. The geodesic equation implies a system of equations on $w(t)$, $f(t)$ and $\phi(t)$ that provides information on the trajectories as we will see in the proof of Theorem~\ref{thmA}.

\section{Proof of Theorem~\ref{thmx_blunes}}\label{proofbill}

We characterise the periodicity of every geodesic billiard trajectory $T(t)$ on $\lune{n}{\theta}$
by unfolding reflections. Starting from an admissible initial geodesic $\gamma$ of $\Se^n$ with
$\gamma(t_0)\in\partial_0$, we construct inductively a sequence of collision times
\[
t_0<t_1<t_2<\cdots,\qquad \mathbb{I}_j=[t_j,t_{j+1}]\subset [t_0,+\infty),
\]
so that $T(t)$ alternately hits $\partial_0$ and $\partial_\theta$, and $\gamma(t_j)$ lies on the meridian
$\partial_{j\theta}$. On each interval $\mathbb{I}_j$ the trajectory is obtained from $\gamma$ by an isometry
$Q_{\theta,j}$ of $\R^{n+1}$ fixing the polar axis. In particular, any periodic billiard trajectory must
close after an even number of geodesic segments.

\begin{thm}\label{thmA}
Let $\theta\in (0,2\pi)$, $n\ge 2$, and let $T(t)$ be a geodesic billiard trajectory on $\lune{n}{\theta}$
defined by an initial admissible unit-speed geodesic $\gamma$ of $\Se^n$ with $\gamma(t_0)\in \partial_0$.
Then the following holds.
\begin{enumerate}[{\rm 1)}]
\item There exists a unique strictly increasing sequence $t_j\to +\infty$ such that
$T(t_{2k})\in \partial_0$, $T(t_{2k+1})\in\partial_{\theta}$, and
$T(t)\notin \partial_0\cup\partial_\theta$ for $t\in (t_{2k},t_{2k+1})$, $k=0,1,\ldots$.
In particular $T(t)$ is defined for all $t\in [t_0,+\infty)$, i.e.\ there are no dead-end trajectories.\\[-3mm]

\item If $T(t)$ is periodic, its period must be even. Moreover, $T(t)$ is periodic of period $2p$
if and only if
\[
R_{2p\theta}\gamma(t)=\gamma(t+t_{2p}-t_0)\qquad\forall\,t,
\]
where $t_0$ and $t_{2p}$ are as in {\rm 1)}. In this case
$\Pi_{\gamma}=\Pi_{R_{2p\theta}\gamma}=R_{2p\theta}\Pi_{\gamma}$.\\[-3mm]

\item[{\rm 3a)}] If $\pi/\theta$ is irrational, the set of periodic trajectories reduces to the single
right-oriented equatorial geodesic $\gamma_e(t)=(0,e^{it})$, $t\ge 0$.
If moreover $\theta<\pi$, then the two-term asymptotic formula \eqref{2termweyl} holds.\\[-3mm]

\item[{\rm 3b)}] If $\pi/\theta$ is rational, with $\theta=m\pi/p$ and $(m,p)=1$ ($m$ and $p$ are relatively prime), then all admissible
trajectories are periodic of period $2p$ and of $t$-period $2m\pi$.
More precisely, $T(t)$ is obtained by folding $\gamma$ by suitable isometries of $\R^{n+1}$ fixing the
polar axis: over $[t_0,t_0+2m\pi]$ it is a folded partition of $\gamma$ into $2p$ geodesic segments.\\[-3mm]

\item[{\rm 3c)}] For any $\theta$, if $\gamma(t)=\gamma_e(t)$ then $T(t)$ is periodic of period $2$ and $t$-period  $2\theta$. If $\theta=m\pi/p$, then $2$ is the minimal possible
periodicity.
\end{enumerate}
\end{thm}

We split the proof of Theorem~\ref{thmA} into two lemmas. The admissibility assumption implies that
$\gamma$ and all reflected segments intersect each meridian at a single point modulo $2\pi$
(see Lemmas~\ref{rigidity}~3) and \ref{elementar}~4)), which makes the unfolding procedure global.

\begin{lemma}\label{thmAA}
The following properties are satisfied by the trajectory $T(t)$ on $\lune{n}{\theta}$, $\theta\in(0,2\pi)$,
corresponding to an initial admissible geodesic $\gamma(t)$. Here $j,k\in\N_0$ and $m,q\in\N$.
\begin{enumerate}[{\rm 1)}]
\item There exists a strictly increasing sequence $(t_j)_{j\ge0}$ such that $T(t_{2k})\in\partial_0$ and
$T(t_{2k+1})\in\partial_\theta$. Moreover,
\begin{equation}\label{gamma2k}
\begin{array}{lll}
&T(t)=(S_0S_{2\theta})^k\,\gamma(t)& \mbox{ for } t\in \mathbb{I}_{2k},\eqskip
&T(t)=(S_{2\theta}S_0)^k\,S_{2\theta}\gamma(t) & \mbox{ for } t\in \mathbb{I}_{2k+1}.
\end{array}
\end{equation}
The intervals are optimal in the sense that $t_{j+1}$ is the smallest real number greater than $t_j$
with the same boundary property, and $t_j$ is characterised by the crossing condition
$\gamma(t_j)\in\partial_{j\theta}$, which is unique modulo $2\pi$.\\[-3mm]

\item In the rational case $\theta=m\pi/q$, we have
$t_0,\ldots,t_{2q}\in [t_0,t_0+2m\pi]$ and $t_{2q}=t_0+2m\pi=t_0+2q\theta$.
Moreover $t_{j+2q}=t_j+2m\pi$.\\[-3mm]

\item If $T(t)$ is periodic, then it closes after an even number $2q$ of segments, i.e.\
$T(t_{2q})=T(t_0)$, and the reflection principle holds at the closing point:
$S_0(T'(t_{2q}^-))=T'(t_0^+)$. The trajectory is periodic with period $2q$ if and only if
either (hence both) of the following equivalent conditions holds:
\begin{eqnarray}\label{closure1}
&&(S_0S_{2\theta})^q\gamma(t_{2q})=\gamma(t_0)\quad \mbox{and}\quad
(S_0S_{2\theta})^q\gamma'(t_{2q})=\gamma'(t_0),\\
\label{closure2}
&&(S_0S_{2\theta})^q\gamma(t-t_0+t_{2q})=\gamma(t)\quad\forall\,t.
\end{eqnarray}
\end{enumerate}
\end{lemma}

\begin{proof}
1) Define the family of unit-speed geodesics
\[
\gamma_{2k}(t):=(S_0S_{2\theta})^k\gamma(t),\qquad
\gamma_{2k+1}(t):=(S_{2\theta}S_0)^kS_{2\theta}\gamma(t),\qquad k\ge0.
\]
Since $(S_0S_{2\theta})^k=R_{-2k\theta}$ and $(S_{2\theta}S_0)^k=R_{2k\theta}$, induction in $k$
together with \eqref{R2redution} and \eqref{ALI2} yields
\begin{equation}\label{family}
\begin{array}{ll}
\gamma_{2k}(t)=S_0\gamma_{2k-1}(t)=R_{-2k\theta}\gamma(t), &\quad
\gamma_{2k+1}(t)=S_{2\theta}\gamma_{2k}(t)=R_{2k\theta}S_{2\theta}\gamma(t).
\end{array}
\end{equation}

We start with $T(t)=\gamma(t)$ on $[t_0,t_1]$, where $t_1>t_0$ is the first time with
$\gamma(t_1)\in\partial_\theta$ (cf.\ \eqref{gammar1}). Reflecting at $t_1$ across $\partial_\theta$
produces a geodesic $\gamma_1$ and defines $T=_{t_1}\gamma_1$ on $[t_1,t_2]$, where $t_2>t_1$ is the
first time with $\gamma_1(t_2)\in\partial_0$. Iterating this procedure gives a strictly increasing
sequence $t_j$ and intervals $\mathbb{I}_j=[t_j,t_{j+1}]$ such that
$T(t)=\gamma_{2k}(t)$ on $\mathbb{I}_{2k}$ and $T(t)=\gamma_{2k+1}(t)$ on $\mathbb{I}_{2k+1}$.
Equivalently, each segment is obtained from the previous one by reflection at the collision point:
\[
(\gamma_{2k})_r=_{t_{2k-1}}\gamma_{2k-1},\qquad
(\gamma_{2k+1})_r=_{t_{2k}}\gamma_{2k},\]
and
\[
\begin{array}{ll}
\gamma_{2k-1}(t_{2k})=\gamma_{2k}(t_{2k})\in\partial_0,
&\qquad \gamma_{2k}(t_{2k+1})=\gamma_{2k+1}(t_{2k+1})\in\partial_\theta,
\eqskip
\end{array}
\]
This proves~\eqref{gamma2k}. Since $T(t_{2k})\in\partial_0$ and
$T(t_{2k+1})\in\partial_\theta$, the identities in \eqref{ALI2} imply
$\gamma(t_{2k})\in\partial_{2k\theta}$ and $\gamma(t_{2k+1})\in\partial_{(2k+1)\theta}$, and the
minimality of each $t_{j+1}$ follows from the uniqueness modulo $2\pi$ (Lemma~\ref{rigidity}~3)).

\smallskip
2) Assume first $\theta=\pi/q$. By admissibility, $\gamma$ meets each meridian at a unique point
modulo $2\pi$. Define inductively $t^*_j>t^*_{j-1}$ as the smallest time with
$\gamma(t^*_j)\in\partial_{j\theta}$. Then $t^*_{2q}=t^*_0+2\pi$, and
$t^*_{j+2kq}=t^*_j$ on $[2k\pi,2(k+1)\pi]$, $k\in\N$. Using \eqref{family} and \eqref{ALI2}, we see
$\gamma_{2k}(t^*_{2k})\in\partial_0$ and $\gamma_{2k+1}(t^*_{2k+1})\in\partial_\theta$, hence the
trajectory constructed in 1) satisfies $t_j=t^*_j$.

If $\theta=m\pi/q$, the same argument gives $\gamma(t_{2q})\in\partial_{2qm\pi}=\partial_0$, and
$t_0,\ldots,t_{2q}\in[t_0,t_0+2m\pi]$ with $t_{2q}=t_0+2m\pi=t_0+2q\theta$. In particular
$t_{2q}-t_0=2m\pi=2q\theta$.

\smallskip
3) If $T(t)$ is periodic, transversality forces it to hit $\partial_0$ and $\partial_\theta$ alternately,
hence it must close after $2q$ segments for some $q\in\N$. From 1), on $\mathbb{I}_{2q-1}$ and
$\mathbb{I}_{2q}$ we have
\begin{equation}
\begin{array}{rcl}\label{fact0}
 T(t) &=& (S_{2\theta}S_0)^{q-1}S_{2\theta}\gamma(t),\quad  t\in \mathbb{I}_{2q-1},
\end{array}
\end{equation}
\begin{equation}
\begin{array}{rcl}\label{fact1}
T(t) & = & S_0(S_{2\theta}S_0)^{q-1}S_{2\theta}\gamma(t)\eqskip
     & = & (S_0S_{2\theta})R_{-2(q-1)\theta}\gamma(t)\nonumber\eqskip
     & = & (S_0S_{2\theta})^q\gamma(t),\qquad \quad t\in \mathbb{I}_{2q}.\nonumber
\end{array}
\end{equation}
Evaluating at $t=t_{2q}$ gives $T(t_{2q})=(S_0S_{2\theta})^q\gamma(t_{2q})=\gamma(t_0)$, which is
the first identity in \eqref{closure1}. The reflection principle at the closing point reads
$S_0(T'(t_{2q}^-))=T'(t_0^+)$. Differentiating \eqref{fact0} at $t=t_{2q}$, applying $S_0$, and
using \eqref{fact1} yields $(S_0S_{2\theta})^q\gamma'(t_{2q})=\gamma'(t_0)$, the second identity in
\eqref{closure1}. Finally, \eqref{closure2} is equivalent to \eqref{closure1} because both sides are
geodesics of $\Se^n$ with the same initial data.
\end{proof}


\begin{lemma}\label{thmAB}
Under the hypothesis in Lemma~\ref{thmAA}, let $t_j$ be defined as in Lemma~\ref{thmAA}{\rm~1)}. Then we have periodicity in the following cases.
\begin{enumerate}[{\rm 1)}]
\item If $\theta=m\pi/q$, then $T(t)$ is a periodic trajectory of period $2q$, and of $t$-period $2m\pi$, thus,  for all $k\in \N$, it satisfies
$T_{|[t_{2kq},\, t_{2(k+1)q}]}(t+2km\pi)=T_{|[t_0,\,t_{2q}]}(t)$, 
with $t_{2kq}-t_{2(k-1)q}=2m\pi$, and $t_{j+2kq}= t_j +2km\pi$ for all $j=0,1,\ldots,2q-1$.
More precisely, $T(t)$ is obtained by partitioning 
$\gamma(t)$ on $[t_0,t_0+2m\pi]$ into $2q$ segments, each folded by an isometry fixing the polar axis.
 If $(m,q)\neq 1$ we may reduce 
$\theta=m'\pi/q'$ with $(m',q')=1$, obtaining smaller periods. $2q'$ and $2m'\pi$.
If $\gamma(t)=\gamma_e(t)=(0, e^{it})$ 
is the equatorial geodesic, the period is reduced to minimal $2$.\\[-3mm]

\item If $\pi/\theta$ is irrational, the only periodic trajectory is the one defined by the equatorial geodesic $\gamma(t)=\gamma_{e}(t)$.
For any non-periodic trajectory, $\{\gamma(t_j)\}_{j\ge0}$ is an infinite set of distinct points.\\[-3mm]

\item Given any $\theta$, the equatorial geodesic $\gamma(t)=\gamma_{e}(t)\in \partial_{t}$ intersects all meridians of $\Se^n$ perpendicularly.
In this case, $t_{j}=j\theta$ at the boundary of $\lune{n}{\theta}$, and
\[
(S_0S_{2\theta})\gamma(t+t_2)=(0,e^{it})=\gamma(t),
\]
that is,~\eqref{closure2} holds with $q=1$. Hence $T(t)$ is periodic with period $2$ and of $t$-period  $2\theta$
Explicitly, for $k\in \Z$,
\[
\begin{array}{llll}
T(t) & = & \left\{
\begin{array}{ll}
\gamma(-2k\theta+t), & \text{for } t\in [2k\theta, (2k+1)\theta] \eqskip
\gamma(2(k+1)\theta-t), & \text{for } t\in [(2k+1)\theta,(2k+2)\theta].\eqskip
\end{array}
\right.
\end{array}
\]
Moreover $R_{2\theta}\gamma(t)=\gamma(t+2\theta)$, and 
$\Pi_{\gamma}=\Pi_{R_{2\theta}\gamma}=R_{2\theta}\Pi_{\gamma}$.
In the case considered in {\rm 1)}, $\theta=m\pi/q$, this reduces the periodicity to its minimum possible value $2$.\\[-2mm]
\end{enumerate}
\end{lemma}

\begin{proof}
1) Clearly,   $\id_{\R^{n+1}}=(S_0S_{2\theta})^q=R_{-2q\theta}$ is equivalent to $\theta=m\pi/q$ for some $m\in \N$, that is, to $\theta/\pi$ being rational. Moreover,  condition $(m,q)=1$ is equivalent to $(S_0S_{2\theta})^p\neq \id_{\R^{n+1}}$ for all $1\leq p<q$.
From Lemma~\ref{thmAA}{\rm~2)} we have $t_{2q}=t_0+2m\pi$, and by Lemma~\ref{thmAA}{\rm~3)} and its proof \eqref{closure2} holds, the closure at $t_{2q}$ satisfies the reflection principle with respect to $\partial_0$, namely
$S_0(T'(t^-_{2q}))=T'(t_0^+)$.
The $t$-periodicity follows immediately.
The second  statement will be proved in the proof of Theorem~\ref{thmA}{\rm~3b)}. The last statement is proved in 3). The reduction of period and $t$-period follows trivially.  The periodicity may not depend  on $(S_0S_{2\theta})^q$ to be the identity, by  on a  combination  equation \eqref{closure2}.
\smallskip

2) If $\pi/\theta$ is irrational, then the meridians $\partial_{j\theta}$ never repeat.
In particular, $T(t)$ cannot be closed. Moreover, $\gamma(t_j)$ is an infinite sequence of distinct points.
\smallskip

3) We take $t_0=0$.
The trajectory $T(t)$ of the equatorial geodesic $\gamma(t)=(0,e^{it})$ is periodic with period $q=2$, whatever $\theta$ is.
To show this, recall first that for all $j\geq 0$, $\gamma(t_j)\in \partial_{j\theta}$ if and only if $t_j=j\theta\modu{2\pi}$.
In particular, from \eqref{ALI2} we have
$S_{2\theta}\gamma(t)=(0, e^{i(2\theta-t)})$  and $ (S_0S_{2\theta})^k\gamma(t)=  (0,e^{i(-2k\theta+t)})$.
Therefore, for $k\geq 1$ we get
\[
(S_0S_{2\theta})^k\gamma(t+t_2)=(S_0S_{2\theta})^k\gamma(t+2\theta)=(0,e^{i(-2(k-1)\theta+ t)}),
\]
that is, \eqref{closure2} holds if and only if $q=k=1$. In particular $T(t)$ is periodic of period $2$, and the explicit values of $T(t)$ follow trivially. Moreover,
for all $k\geq 0$, $T(t_{2k})=\gamma(0)\in \partial_0$ and $T(t_{2k+1})=\gamma(\theta)\in \partial_{\theta}$, and the same holds for $k\in \Z$.
Hence $T(t)$ closes after two geodesic segments.
Since $\gamma'(t_j)= \nu_{j\theta}$, by \eqref{Vr} we have $a=0$, meaning that $\gamma(t)$ and the reflected $\gamma_r(t)$ meet the boundary orthogonally, and so the same holds for $T(t)$.
\end{proof}

\begin{proof}[Proof of Theorem~\ref{thmA}]
1) Take the sequence $t_j$ obtained in Lemma~\ref{thmAA}{\rm~1)}. We prove that $t_j\to +\infty$.
Assuming $t_j$ bounded, there exist a subsequence $t_{j}\nearrow t^*<+\infty$.
Let $\theta^*$ be such that $\gamma(t^*)\in \partial_{\theta^*}$.
Then $\|\gamma(t_{2k+1})-\gamma(t_{2k})\|\to 0$. 
Using the representation \eqref{representation} we have $w(t_j)\to w(t^*)$ and $f(t_j)e^{ij\theta}\to f(t^*)e^{i\theta^*}$,
with $|w(t_j)|\neq 1$ and $0\neq f(t_j)\to f(t^*)$.
Therefore, for $k\to +\infty$, we have
\[
\|w(t_{2k})-w(t_{2k+1})\|\to 0, \qquad
\|f(t_{2k})e^{i2k\theta}-f(t_{2k+1})e^{i(2k+1)\theta}\|\to 0.
\]
From the latter limit we get
\[
f(t^*)\|(1-e^{i\theta})\|
=f(t^*)\times \lim_{k\to +\infty}\left\|\left(1-\frac{f(t_{2k+1})}{f(t_{2k})}e^{i\theta}\right)\right\|\to 0,
\]
and hence $f(t^*)=0$. This means that $\gamma(t)$ crosses the polar sphere, yielding a contradiction.
\smallskip

2) This follows from Lemma~\ref{thmAA}{\rm~3)}. Applying \eqref{R2redution} and \eqref{representation} implies the equalities with $\Pi_{\gamma}$.
\smallskip

3) From Lemma~\ref{thmAB}{\rm 2)} the set $\{\gamma(t_j)\}_{j\ge0}$ is infinite and consists of distinct points, and from 1) we have $t_j\to +\infty$.
So there are no dead-ends. Assume $T(t)$ is periodic. For simplicity we take $t_0=0$.
By \eqref{closure2} we have
\[
\gamma(t)=R_{-2q\theta}\gamma(t+t_{2q}).
\]
Using this identity, with $\gamma(0)\in \partial_0$, $\gamma(t_{2q})\in \partial_{2q\theta}$, and \eqref{representation}, with $f(t)\neq 0$ for all $t$ by admissibility,  we obtain $\phi(0)=0$, $\phi(t_{2q})=2q\theta$,  where
\[
\begin{array}{ll}
\gamma(t)=(w(t),f(t)e^{i\phi(t)}), &
\gamma'(t)=(w'(t), (f'(t)+i\phi'(t)f(t))e^{i\phi(t)}),\eqskip
R_{-2q\theta}\gamma(0)= (w(0),f(0)e^{-2iq\theta}),&
R_{-2q\theta}\gamma'(0)= (w'(0),(f'(0)+i\phi'(0)f(0))e^{-2iq\theta}),
\end{array}
\]
and
\[
\begin{array}{lll}
\gamma(t) & = & \cos t \,\gamma(0) + \sin t\,\gamma'(0)\eqskip
& = & (\cos t\, w(0)+ \sin t\, w'(0), \cos t\, f(0)+\sin t\, f'(0)+ i\sin t\, \phi'(0)f(0)).
\end{array}
\]
On the other hand, the periodicity condition gives
\[
\begin{array}{lll}
\gamma(t)&=&  R_{-2q\theta}\gamma(t+t_{2q})\eqskip
&=& \cos (t+t_{2q}) R_{-2q\theta}\gamma(0)+ \sin (t+t_{2q})R_{-2q\theta}\gamma'(0)\eqskip
&=& \cos (t+t_{2q})\left( (w(0),f(0)e^{-2iq\theta})\right)+ \sin (t+t_{2q})\left((w'(0),(f'(0)+i\phi'(0)f(0))e^{-2iq\theta})\right)\eqskip
&=& \left( \cos (t+t_{2q})w(0)+ \sin (t+t_{2q})w'(0)~,~ (\cos (t+t_{2q})f(0)
+ \sin (t+t_{2q}) f'(0))e^{-2iq\theta} \right.\eqskip
&& \hspace*{5mm} \left.+\sin(t+t_{2q})\phi'(0)f(0)ie^{-2iq\theta}\right).
\end{array}
\]
Therefore we obtain the following identities,
\[
\begin{array}{rcl}
\label{c}
\cos t\, w(0)+ \sin t\, w'(0) & = &
 \cos (t+t_{2q})w(0)+ \sin (t+t_{2q})w'(0)\eqskip
 \cos t\, f(0)+\sin t\, f'(0)+ i\sin t\, \phi'(0)f(0)
 & = & \left(\cos (t+t_{2q}) f(0)+
\sin (t+t_{2q}) f'(0)\right)e^{-2iq\theta}\nonumber \eqskip
& &\hspace*{5mm}+\sin(t+t_{2q})\phi'(0)f(0)ie^{-2iq\theta}.\nonumber
\end{array}
\]
Writing the last identity in real variables,
\begin{equation}\label{d}
 \begin{array}{rcl}
\cos(t)\, f(0)+\sin t\, f'(0) & = & (\cos (t+t_{2q}) f(0)+ \sin (t+t_{2q}) f'(0))\cos(-2q\theta)
-\sin(t+t_{2q})\phi'(0)f(0)\sin(-2q\theta), \eqskip
\sin(t)\, \phi'(0)f(0) & = & (\cos (t+t_{2q}) f(0)+
\sin (t+t_{2q}) f'(0))\sin(-2q\theta)+\sin(t+t_{2q})\phi'(0)f(0)\cos(-2q\theta).\nonumber
\end{array}
\end{equation}
The resulting equations at $t=0$ then yield
\begin{equation}\label{a}
\begin{array}{cll}
w(0) & = &\cos(t_{2q})w(0) + \sin(t_{2q})w'(0)\eqskip
 1 & = & \left(\cos(t_{2q})+A\sin(t_{2q})\right)\cos(-2q\theta)-\sin(t_{2q})\phi'(0)\sin(-2q\theta)\eqskip
 0 & = & \left(\cos(t_{2q}) +A\sin(t_{2q})\right)\sin(-2q\theta) +\sin(t_{2q})\phi'(0)\cos(-2q\theta),
\end{array}
\end{equation}
where $A=f'(0)/f(0)$, while at $t=-t_{2q}$ we obtain
\begin{equation}\label{b}
\begin{array}{cll}
 w(0) & = & \cos(-t_{2q})w(0)+\sin(-t_{2q})w'(0)\eqskip
 \cos(-2q\theta) & = & \cos(-t_{2q})+\sin( {-}t_{2q})A\eqskip
 \sin(-2q\theta) & = & \sin(-t_{2q})\phi'(0).
\end{array}
\end{equation}
From the first identities in \eqref{a} and \eqref{b}, we get $\sin(t_{2q})w'(0)=0$.
If $\sin(t_{2q})=0$ then by the third equation in \eqref{a} we obtain $\sin(-2q\theta)=0$.
Next we continue the proof separating the cases {\rm 3a)} and {\rm 3b)}.
\smallskip

3a) Since $\pi/\theta$ is irrational, we have $\sin(-2q\theta)\neq 0$. Hence $w'(0)=0$ and by the first identity in \eqref{b}
$(1-\cos(t_{2q}))w(0)=0$, with $\cos(t_{2q})\neq 1$, and thus $w'(0)=w(0)=0$.
On the other hand, the geodesic equation for $\gamma(t)$ is $\gamma''(t)=-\gamma(t)$, which implies $w''(t)=-w(t)$.
The initial conditions at $t=0$ then yield $w(t)=0$, hence $f(t)=1$ and $\gamma(t)=(0, e^{i\phi(t)})$.
By the geodesic equation $\gamma''=-\gamma$ we get $\phi''=0$ and $(\phi')^2=1$. Hence $\phi(t)=\pm t$, and so
$\gamma(t)=(0, e^{it})=\gamma_{e}(t)$ is the equatorial geodesic.
Hence, the set of periodic geodesics has measure zero, so the nonperiodicity condition is satisfied and from 1) the nonblocking condition is also satisfied.
Then if $\theta<\pi$, the two-term asymptotic formula \eqref{2termweyl} holds.
\smallskip

3b) Now we take $\theta=m\pi/p$. The first part is proved in Lemma~\ref{thmAB}~{\rm 1)} (with $q$ and $p$ switched).
We assume $(m,p)=1$ and $1\leq q<p$, and so $R_{-2q\theta}\neq \id$.
Let us assume $T(t)$ is also of period $2q$, for some $1\leq q<p$.
By Lemma~\ref{thmAA}{\rm~3)} this means that $R_{-2q\theta}\gamma(t+t_{2q})=\gamma(t)$, and so $R_{-2q\theta}\gamma(t_{2q})=\gamma(0)$.
If $\sin(t_{2q})\neq 0$ then $w'(0)=0$ and, following the proof of {\rm~3a)}, either $w(0)=0$ or $\cos(t_{2q})=1$.
In the first case we conclude $\gamma(t)=\gamma_e(t)$.
In the second case, by \eqref{b} we have $\cos(-2q\theta)=1$, which is equivalent to $p/m=q/k$ for some $k\in \N$ with $q<p$.
This contradicts the assumption that $(m,p)=1$.
Hence there are no periodic solutions other than the equatorial geodesic.
\smallskip

3c) Follows from Lemma~\ref{thmAB}{\rm~3)}.
\end{proof}

\subsection{Intersection results}\label{intersection}
Here we justify the transversality condition of a unit-speed geodesic $\gamma(t)$ of $\Se^n$ with respect to a family of meridians $\partial_{\phi}$ assumed in Theorem~\ref{thmA}.

In Remark~\ref{zeromeasure}~{\rm 1)} we introduced the space of oriented unit-speed geodesics $\mathcal{G}(\Se^n)$ of the $n$-sphere.
Each unit-speed geodesic $\gamma$ is identified with an oriented $2$-plane $\Pi_{\gamma}$ spanned by initial conditions $\{\gamma(t_0), \gamma'(t_0)\}$.
By definition, this $2$-plane satisfies the property $\Se^n\cap \Pi_{\gamma}=\gamma =\Se^1$.

In the next two lemmas we state some useful elementary facts {about geodesics of $\Se^{n}$.

\begin{lemma}\label{elementar}
 Let $\gamma$ be a unit-speed geodesic of $\Se^n$.
\begin{enumerate}[{\rm 1)}]
\item If $\gamma$ intersects an $m$-sphere $\Se^m\subset \Se^n$ at two linearly independent points then $\gamma\subset \Se^m$.\\[-3mm]
 \item For each $\phi$, either $\gamma$ intersects transversely $\partial_{\phi}$ at a single point or $\gamma\subset \Se^{n-1}_{\phi}$.\\[-3mm]
 \item If $n=2$ only the constant geodesic identically equal to North or South poles is part of the polar sphere.
If $n\geq 3$ the space of unit-speed geodesics that are part of the polar sphere has dimension $2\times (n-3)$.
If $n=3$ the latter is a set of two polar geodesics $\gamma(t)=(e^{\pm it}, 0^2)$.
The space of geodesics that are part of a given meridian sphere has dimension $2\times(n-2)$.
The set of geodesics that are part of some meridian sphere, or equivalently that intersect the polar sphere, is of Hausdorff dimension { $2\times (n-2)+1$ a.e.}.\\[-3mm]
\item Any rotation or reflection fixing the polar axis transforms meridians of $\Se^n$ into meridians of $\Se^n$.
In particular if $\gamma(t)$ does not cross the polar sphere, then neither does $R_{\phi}\gamma(t)$ nor $S_{\phi}\gamma(t)$.
 \end{enumerate}
\end{lemma}

\begin{proof}
 1) If $\gamma(t_1),\gamma(t_2)\in \Se^m$ and are linearly independent, then they span $\Pi_{\gamma}$ and the latter is contained in the vectorial support of $\Se^m$.
Then $\gamma= \Pi_{\gamma}\cap \Se^n\subset \Se^m$.\\[1mm]
 2) If $\dim(H_{\phi}+\Pi_{\gamma})=n$ then $\gamma$ is part of $H_{\phi}$, and so
 $\gamma=\Pi_{\gamma}\cap \Se^n\subset H_{\phi}\cap \Se^n=\Se^{n-1}_{\phi}$.
Now assume $\gamma$ is not part of $\Se^{n-1}_{\phi}$.
Then
 \[
 \dim(H_{\phi}\cap \Pi_{\gamma})=\dim(H_{\phi})+ \dim(\Pi_{\gamma})-\dim(H_{\phi}+\Pi_{\gamma})\geq n+2-(n+1)\geq 1,
 \]
 we conclude that $\gamma\cap H_{\phi}= \Se^n\cap \Pi_{\gamma}\cap H_{\phi}\neq \emptyset$.
Hence $\gamma$ intersects $\Se^{n-1}_{\phi}$ at least twice with $\gamma(t_1)\in\partial_{\phi}$, and its antipodal $\gamma(-t_1)\in \partial_{\phi+\pi}$.
From 1) there are no more intersection points.\\[1mm]
3) It only remains  to prove the last statement. 
For each $0\leq \phi<\pi$, set $\G{\phi}:=\mathcal{G}(\Se^{n-1}_{\phi})$. 
Then  for $\phi\neq 0$ we have a sequence  of submanifolds $\mathcal{G}(\Se^{n-2}_{polar})=\G{0}\cap \G{\phi}\subset \G{\phi}\subset\mathcal{G}(\Se^{n}) $ of strictly increasing Hausdorff measure, $2(n-3)$ (zero if $n=2$), $2(n-2)$ and  $2(n-1)$, respectively (cf. Remark~\eqref{zeromeasure})
 We now claim there exists a smooth map \ $\tau:[0, 2\pi)\times \G{0}\to \mathcal{G}(\Se^{n})$,
$\tau(\phi, \Pi_{\gamma_{(x,v)}} )=\Pi_{\gamma_{(R_{\phi}x,R_{\phi}v)}}$,  whose image $Y=\cup_{\phi}\G{\phi}$ has necessarily Hausdorff measure at most $2(n-2)+1$. To
see it is smooth, we recall that  $\mathcal{G}(\Se^n)$ is a double cover of the usual Grassmannian $G_2(\R^{n+1})$, and so they have the same dimension. The 2-dimensional subspace
$F=\Pi_{\gamma}$ is an element of $G_2(\R^{n+1})$ that can be identified with  the orthogonal projection $\pi_F$  onto
$F$~\cite{ms}. Since $\Pi_{\gamma_{(x,v)}}=\Pi_{\R[x,v]}$, 
then $\pi_{\R[x,v]}(X)=\langle X, x\rangle x+\langle X, v\rangle v$, $\forall X\in \R^{n+1}$, and 
\[\tau(\phi, \pi_{\R[x,v]})(X)=\langle X, R_{\phi}x\rangle R_{\phi}x+\langle X, R_{\phi}v\rangle R_{\phi} v, \]
is clearly smooth and injective. By a classical result on Lipschitz maps, the dimension of the image $Y$ has Hausdorff dimension at most
$1+ 2(n-2)$. Using local principal charts $\Omega\times \mathbb{S}^1$ of $U\Se^n$, and taking $Y'\subset Y\cap \Omega$ to be a sufficiently small neighbourhood  of any given point of $Y$ (cf.\ Remark~\ref{zeromeasure}), we get that the dimension (a.e) of $Y$ is indeed  $2n-3$,
as shown in what follows.
Associated to $\tau$ is the injective map $\bar{\tau}:\Se^1\times  U\Se^{n-1}_0\!\to \!U\Se^n$ defined by,
\[\bar{\tau}\left(e^{i\phi},((w, \sqrt{1-|w|^2}(1,0)), (h, \sqrt{1-|h|^2}e^{i\frac{\pi}{2}})\right)= \left((w, \sqrt{1-|w|^2}e^{i\phi}), (h, \sqrt{1-|h|^2}e^{i(\phi+\frac{\pi}{2})})\right), \]
where we use~\eqref{representation} for $x$ and  $v\bot x$, that clarifies the injectivity. Now we show  its derivative at $(e^{i\phi},(x,v))$ is also injective if {$(x,v)$  is away from  $U\mathbb{S}^{n-2}_{polar}$}. 
The derivative with respect to $\phi$ is given by
\[\frac{\partial \bar{\tau}}{\partial \phi}=((0, \sqrt{1-|w|^2}ie^{i\phi}), (0, \sqrt{1-|h|^2}ie^{i\phi})), \]
that vanishes only if $|w|=|h|=1$, that is at elements $(x,v)\in U\Se^{n-2}_{polar}$, that have measure zero on $U\Se_0^{n-1}$. Taking curves $ w_s$ and $ h_s$ with
vector variation at $s=0$, $W$ and $H$ respectively, we obtain the derivative of $\tau$ in the variables $w$ and $h$ away from {geodesics} of the polar sphere,
\begin{gather*}
\frac{d \bar{\tau}}{ds}\Big{|}_{s=0}\!\left(e^{i\phi}, \left((w_s, \sqrt{1-|w_s|^2}(1,0)), (h_s, \sqrt{1-|h_s|^2}e^{i\frac{\pi}{2}})\right)\right)=
\left((W, \frac{-\langle w_0,W\rangle}{\sqrt{1-|w_0|^2}}e^{i\phi}),
(H, \frac{-\langle h_0,H\rangle}{\sqrt{1-|h_0|^2}}e^{i(\phi+\frac{\pi}{2})})                                                                                                                                                                                                                                                                                                                                                                                                                                                                                                                                                                                                                                                                                                              \right), 
\end{gather*}
this vanish if and only if $W=H=0$.  This implies that the image $Z:=\bar{\tau}( \Se^1\times U\Se_0^{n-1})=
\cup_{\phi\in \Se1} U\Se^{n-1}_{\phi}$ has Hausdorff dimension
$2n-3 +1 =2n-2$ away  from $|w_0|=|h_0|=1$. Since $P^{-1}(Y)=\cup_{\phi} P^{-1}(\mathcal{G}_{\phi})$,  and for each $\phi$,  $P^{-1}(\mathcal{G}_{\phi})=U\mathbb{S}^{n-1}_{\phi}$, then $P^{-1}(Y)=Z$. This implies $Y$ has dimension $2n-1=2(n-2)+1$  away from $\mathcal{G}({S}^{n-2}_{polar})$.
\smallskip

 4) The first and second statements follow by applying~\eqref{ALI2} and identities $\langle R_{\phi}\gamma'(t), \nu_{\theta}\rangle$
 $= \langle \gamma'(t), \nu_{-\phi+\theta}\rangle$,and  $\langle S_{\phi}\gamma'(t), \nu_{\theta}\rangle= \langle \gamma'(t), \nu_{\phi- \theta}\rangle$, respectively.
 Note that admissibility of $\gamma(t)$ is equivalent to $ \langle \gamma'(t), \nu_{\theta}\rangle \neq 0$ for all $t$  and $\theta$.
\end{proof}

 \begin{lemma}\label{rigidity}
Let $\gamma$ be a unit-speed geodesic of $\Se^n$. Then we have.
\begin{enumerate}[{\rm 1)}]
\item  $\gamma$ is part of the polar sphere $\Se^{n-2}_{polar}$ if and only if it is part of  $\Se^{n-1}_{\phi}$ for any $\phi$.\\[-3mm]
\item $\gamma$ intersects the polar sphere  transversely (at two antipodal points $\pm\gamma(t_1)$ {only})
 if and only if  there exists a single $\phi\modu{\pi}$ such that $\gamma\subset \Se^{n-1}_{\phi}$. In this case $\gamma$ intersects transversely $\Se^{n-1}_{\psi}$ at the same two antipodal points $\pm\gamma(t_1)$ for all $\psi\neq \phi\modu{\pi}$.\\[-3mm]
 \item $\gamma$ does not intersect the polar sphere
 if and only if it intersects $\Se^{n-1}_{\phi}$
 transversely (at two antipodal points {only}) for all $\phi$. If this is the case, then for each $k\in \N_0$, there exists unique  $t_k\in[t_0, t_0+2\pi)$ such that $\gamma(t_k)\in \partial_{k\theta}$.
 \end{enumerate}
\end{lemma}
\begin{proof}
1) Follows immediately from (\ref{POLE}).
2) Let $t_1$ such that $\pm \gamma(t_1)\in \Se^{n-2}_{polar}$, that is $\gamma(t_1)=(w_0,0^2)$ and $\gamma'(t_1)=(w', he^{i\phi})\bot \gamma(t_1)$, with $\langle w_0,w'\rangle =0$.  Note that $h\neq 0$ otherwise $\gamma(t)$ would be tangent to the polar sphere, and so   part of it.   Then for all $t$,
\begin{eqnarray*}
\gamma(t)
&=& (\cos(t-t_1)w_0+ \sin(t-t_1)w',  \sin(t-t_1)he^{i\phi})\in \Se^{n-1}_{\phi}.
\end{eqnarray*}
In particular, $\gamma(t_1)=(w_0,0^2) \in \Se^{n-1}_{\psi}$ for any $\psi$ and the intersection is transversal for $\psi\neq \phi$
otherwise $\gamma$ would be part of  $\Se^{n-1}_{\psi}$ as well, and so of the polar sphere by 1).  The reciprocal follows from 1) as well.
3) Follows easily by contradiction from 1) and 2).
\end{proof}
  \begin{remark} In Lemma~\ref{rigidity}~3), $t_k\in [t_0,t_0+2\pi)$ are not necessarily ordered by increasing order, while in
  Theorem~\ref{thmA} $t_k\in [t_0, +\infty)$ are taken in a strictly increasing order, interactively by construction using the
  trajectory $T(t)$ in $\lune{n}{\theta}$.
  \end{remark}
\section{Dirichlet and Neumann eigenvalues of $\lune{n}{\pi/p}$\label{Eigen}}
As mentioned in the Introduction, eigenvalues of lunes with opening $\pi/p$, where $p\in \N_2$, are also eigenvalues of $\Se^n$,
the main question being what the corresponding multiplicities are. In order to fully describe these spectra, we will use the
fact that $\lune{n}{\pi/p}$ may be associated to {the dihedral group} $\mathbb{D}_{p}$
of order $2p$, that is the group of symmetries of a regular polygon with $p$ sides, defining a Coxeter system.
This fact allows us  to compute the multiplicities of the Dirichlet and the Neumann eigenvalues of $\lune{n}{\pi/p}$ using invariant
theory, {and thus fully describe both spectra}. This does, in turn, allow us to derive a formulation
of P\'{o}lya's conjecture depending only on the lowest and highest order eigenvalues in each chain in the Neumann and Dirichlet cases,
respectively, namely inequality~\eqref{PolyaExtreme} below. We then show this to be  equivalent to the non-negativeness of a family
of $p$  polynomial functions in one variable and with coefficients depending on $n$, $p$ and a parameter $r$ with integer values
between $0$ and $p-1$. Given $n\geq 2$, we determine which dihedral $n$-lunes satisfy P\'{o}lya's conjecture by defining a
slightly stronger P\' {o}lya-type inequality, $\Phi_n(p,R)\geq 0$, {where $\Phi_n$ denotes a one-parameter family of functions on
the variables $p\in \N$, and $R\in \N_0$, indexed by the dimension $n$.} The relation with $M_{n,p,r}(m)$ comes from
the decomposition $R+p=K=(m+1)p+r$.

\subsection{Background and notation.}\label{Notations}
Consider the set of distinct Dirichlet eigenvalues $\bar{\lambda}_K=\bar{\lambda}_K(\Omega)$, $K=1,2, \ldots$, with
$0<\bar{\lambda}_1<\bar{\lambda}_2<\ldots$, with corresponding multiplicities denoted by $m(K):=m[\bar{\lambda}_K]$. We further denote
the sum of multiplicities by
\[\sigma(K):=\sigma[\bar{\lambda}_K] =m(1)+m(2)+\ldots+m(K)= \sigma(K-1)+m(K),\]
where $m(1)=\sigma(1)=1$,  and we make the convention that  $\sigma(0)=0$. Then, for each $K\geq 1$, the $K$-chain of eigenvalues is
composed by $m(K)$ eigenvalues $\lambda_k$ that are equal to $\bar{\lambda}_K$ and can be described by the integers $k=\sigma(K-1)+{q}$
such that $ k_-(K)=\sigma(K-1)+1\leq k \leq k_+(K)=\sigma(K)$, where $q=1,\ldots,m(K)$, and $k_-(K)$ and $k_+(K)$ denote the lowest and
the highest order of the $K$-chain,  respectively.
In a similar way we consider the distinct Neumann eigenvalues $\bar{\mu}_K$, $K=0,1, \ldots$, ordered increasingly
$0=\bar{\mu}_0<\bar{\mu}_1<\ldots$, with multiplicities $m'(K):=m'[\bar{\mu}_K]$ and sum of multiplicities
\[\sigma'(K):= \sigma'[\bar{\mu}_K]=m'(0)+ m'(1)+\ldots+m'(K)=\sigma'(K-1)+m'(K),\] where $m'(0)=\sigma'(0)=1$, and again
make the convention that $m'(-1)=\sigma'(-1)=0$. The $K$-chain is given by the $m'(K)$ eigenvalues $\mu_k$ equal to $\bar{\mu}_K$
where $k=\sigma'(K-1)+{q}$ and  $k'_-(K)=\sigma'(K-1)\leq k\leq k'_+(K)=\sigma'(K)-1$,  for ${q}=0,\ldots,m'(K)-1$.

With this notation, each of the two P\'{o}lya inequalities in~\eqref{polyaconject} may be written in terms of the distinct eigenvalues
$\bar{\mu}_K$ and $\bar{\lambda}_K$ by noting that if they hold for the lowest and highest orders on the $K$-chains in the Neumann
and Dirichlet cases, respectively, then they hold for all eigenvalues in the corresponding chain. More precisely, P\'{o}lya's conjecture
may be written in an equivalent way as
\begin{equation}\label{PolyaExtreme}
\begin{array}{lll}
 	\bar{\mu}_K \leq \fr{4\pi^2 (k'_-(K))^{2/n}}{(\omega_n|\Omega|)^{2/n}} & \mbox { and } &
 	\bar{\lambda}_K \geq \fr{4\pi^2 (k_+(K))^{2/n}}{(\omega_n|\Omega|)^{2/n}},
\end{array}
 \end{equation}
in the Neumann and Dirichlet cases, respectively.

As mentioned in the Introduction, dihedral $n$-lunes satisfy P\'{o}lya's conjecture in the Neumann case, as they tile the corresponding
hemisphere which is, in turn, known to satisfy P\'{o}lya's conjecture~\cite{blps,fms}. This is a consequence of spectral comparison
results for tiling domains obtained by P\'{o}lya in \cite{poly2}.  It thus remains to study the
Dirichlet problem, which we will do by considering invariant tilings with respect to Coxeter groups. This allows us to derive further
consequences for eigenvalues, by applying invariant theory to dihedral lunes.

If $M=\mathbb{S}^n$, and  $M'=\mathbb{S}^n_+$,  assuming that $M''$ tiles $M'$ or $\mathbb{S}^n$,  the Dirichlet and Neumann  eigenvalues $\lambda''_k$, $\mu''_k$  of $M''$ are considered in \cite{bb},   namely the solutions of the  respective eigenvalue problems (\ref{eigprob}) with $\Omega=M''$ 
under the condition the tiling is invariant with respect to a finite reflection group  of $\mathbb{R}^{n+1}$
of  the form $\tilde{\mathcal{W}}=\{\id_{\mathbb{R}^{n+1-l}}\}\times\mathcal{W}$, given by a trivial extension of a {Coxeter} group of isometries $\mathcal{W}$ generated by $l$  reflections of $\mathbb{R}^l$, for some $1\leq l\leq n+1$.
 In this case,   $M''=(\mathbb{R}^{n+1-l}\times \mathfrak{C})\cap \mathbb{S}^{n+1}$ where $\mathfrak{C}$ is a chamber defined by a Coxeter system of rank $l$ of  $\mathcal{W}$. The fact that  $\mathcal{W}$ acts transitively and faithfully among the Weyl chambers, and so covering all tiles, implies that given a smooth function $u$ on $M''$, there is a unique  $\tilde{\mathcal{W}}$-anti-invariant
($\tilde{\mathcal{W}}$-invariant, respectively) function $\tilde{u}$ on $\mathbb{S}^n$ that extends $u$, namely $\tilde{u}(w(x))=\epsilon(w)u(x)$, where $x\in M''$, $w\in \tilde{\mathcal{W}}$ and $\epsilon(w)=\det(w)$ ($=1$, respectively).     
If  $u$ is a solution  on $\Omega=M''$ of (\ref{eigprob})(D) ((N), respectively), the  anti-invariant (invariant, respectively) extension  $\tilde{u}$  is an eigenfunction
of the Laplacian on $\mathbb{S}^n$ for the closed problem  with respect to the same eigenvalue $\lambda$ ($\mu$, respectively).
The reciprocal also holds, by taking $\tilde{u}$ restricted  to $M''$ {as long as $\tilde{u}=0$ on $\partial M''$}~\cite[Prop.\ 7]{bb}. 

{We observe that the above framework used in~\cite{bb}  {to completely  describe the spectrum of $M''$} holds in the same way for finite Coxeter groups $\mathcal{W}$ as the dihedral group of any order $2p$, $p\geq1$, that are Weyl groups only for $p=2,3,4, 6$ (more details are given  in Appendix~\ref{ApA})}.
The model tile $M''$ we consider here 
is given by a  lune of angle opening   $\pi/p$,  defined by a sector
$\mathfrak{C}$  of  $\mathbb{R}^2$, 
$\luned{n}{p}= (\mathbb{R}^{n-1}\times \mathfrak{C})\cap \mathbb{S}^{n}$,  where
 $\mathfrak{C} =\left\{ re^{i\varphi}\in \mathbb{R}^2:  \varphi\in \left[0,\pi/p\right], r\geq 0 \right\}$ is just a chamber defined by  the dihedral group $\mathbb{D}_p$,  
 giving a tiling by $2p$ tiles of a $n$-dimensional 
 sphere by slices that contains two fixed antipodal points of the boundary. This is also a $p$-tiling
 of the  hemisphere. Each slice is an extension of a chamber in $\mathbb{R}^2$ defined by the  $p$ axes of symmetry of a regular $p$-polygon.
 The corresponding reflection group is  $\tilde{\mathcal{W}}=\{\id_{\mathbb{R}^{n-1}}\}\times \D_p$.
  Deriving the Hilbert-Poincar\'{e} series of graded algebras of $\tilde{\mathcal{W}}$--invariant and anti-invariant homogeneous polynomials  (details are given in  Appendix~\ref{ApA}), we compute  the multiplicity of each eigenvalue of a lune from the  coefficients of the series.
  
This method of computing multiplicities and its application to proving P\'{o}lya conjecture are described in~\cite{bb}, {but this was carried out
  explicitly in that paper only for the $2-$hemisphere $H_2=\mathbb{S}^2_+$, being also mentioned that other domains tiling $\Se^{2}$, namely
  $H_4=\mathbb{S}^2\cap\{(x,y,z): x\geq 0, y\geq 0\}$ and $H_8=H_4\cap\{(x,y,z):z\geq 0\}$ (respectively $\lune{2}{\pi/2}$ and $\lune{2}{\pi/4}$ in
  our notation), can be treated in the same way.}  Note that $H_8$ tiles $H_4$
  and $H_4$ tiles $H_2$, but what is relevant is that these three cases  correspond to  reflection groups associated to a root system of type $A_1$,
  $A_1\times A_1$, and $A_1 \times A_1\times A_1$, respectively.
  We also observe that while $\luned{2}{1}=H_2$, and  $\luned{2}{2}=H_4$  contain two antipodal points, as is the case for all lunes, 
  this is no longer the case for $H_8$.
  
  We use a corresponding notation for the eigenvalues of $M=\mathbb{S}^n$, $M'=\mathbb{S}_+^n$,  and $M''=\lune{n}{\pi/p}$, denoting
  the zero eigenvalue of $\mathbb{S}^n$ as the $0$-th eigenvalue, and the same for the zero Neumann eigenvalue of $M'$ and $M''$. 
  In Lemma~\ref{lemmaA} we show that the  eigenvalues of $\luned{n}{p}$, $p\geq 1$, are exactly the eigenvalues of $\mathbb{S}_+^n$,
  albeit with different multiplicities, so of the form $K(K+n-1)$, with $K\geq p$ for the Dirichlet case and $K\geq 0$ for the Neumann
  case. We denote the Dirichlet eigenvalues by $\bar{\lambda}_{K}=K(K+n-1)$,  with multiplicity $m_{n,p}[K(K+n-1)]$, and denote the  sum of the multiplicities from the first $\bar{\lambda}_p$  till $\bar{\lambda}_K$
  by $\sigma_{n,p}[K(K+n-1)]$. Similarly for the  Neumann eigenvalues $\bar{\mu}_{K}=K(K+n-1)$, where $K\geq 0$, where
  we adopt the notation $m'_{n,p}$, and $\sigma'_{n,p}$ for the multiplicities. We also extend the notation used
  in~\cite{fms} for the case $p=1$ to $p\geq 2$, as
\[
 m_{n,p}(K)=m_{n,p}[K(K+n-1)]
\]
for the Dirichlet eigenvalues, and
\[
 m'_{n,p}(K)=m'_{n,p}[K(K+n-1)]
\]
for the Neumann eigenvalues. Similar notations will be used for $\sigma_{n,p}(K)$ and $\sigma'_{n,p}(K)$.

{To prove Theorem~\ref{thmxmultiplicity} we will use the following representation of non-negative integers to describe the
multiplicities of eigenvalues of a lune $\lune{n}{\pi/p}$ given in Lemma~\ref{lemmaA} below.}
\begin{definition}
Fix $n,p\geq 2$. Given $K\geq 0$ we take the following representation 
 \begin{equation}\label{K(m,r)}
 K=K(m,r):=mp+r,  \quad \mbox{where~} m=\left\lfloor \frac{K}{p}\right\rfloor
 \mbox{~and~} r=K-mp, 
 \end{equation}
 where $m\geq 0$ is an integer, $r=0, 1, \ldots, p-1$. 
 The lexicographic order,
$K(m,r)> K(m',r')$ if and only if $m>m'$ or $m=m'$ and $r>r'$,
agrees with the natural order of the eigenvalues of $\Se^n$,  
$\bar{\lambda}_{K(m,r)}$.
 \end{definition}

Recall now the rising factorial, given for $k\geq 0$ and $n\geq 1$ by
$k^{\bar{n}}=k(k+1)\ldots (k+n-1)$, with $k^{\bar{0}}=1$, and the following relations,
the second one named the parallel summation identity
\begin{equation}\label{parallelsum}
\fr{k^{\overline{n}}}{n!}=\binom{n+k-1}{n} = \dsum_{j=1}^{k}\binom{n+j-2}{j-1} = \dsum_{j=0}^{k}\frac{j^{\overline{n-1}}}{(n-1)!}.
\end{equation}

The multiplicities of the  $K^{\rm th}$  Dirichlet eigenvalue of $\mathbb{S}^n_+$, $\bar{\lambda}_K=K(K+n-1)$, $K\geq 1$, and of the
Neumann eigenvalue $\bar{\mu}_K=K(K+n-1)$, $K\geq 0$, are given by, respectively~\cite{bsj,bb},
\begin{gather}\label{hemimulti}
\begin{array}{ll}
m_{n,1}(K) = \fr{K^{\overline{n-1}}}{(n-1)!},
&\quad \sigma_{n,1}(K) = m_{n,1}(1)+\ldots + m_{n,1}(K)=  \fr{K^{\overline{n}}}{n!},\eqskip
 m'_{n,1}(K) = \fr{(K+1)^{\overline{n-1}}}{(n-1)!},&\quad
 \sigma'_{n,1}(K) = m'_{m,1}(0)+\ldots +m'_{n,1}(K)=  \fr{(K+1)^{\overline{n}}}{n!}.
 \end{array}
 \end{gather}
The $K$-chain of the Dirichlet eigenvalues of  $\luned{n}{p}$ is given by $\lambda_k=\bar{\lambda}_{K}$,   where  
 \[k_-=k_-(K)\leq k\leq k_+(K)=k_+,\]
 with $k_-=\sigma_{n,p}(K-1)+1$, and $k_+=\sigma_{n,p}(K), $ the lowest and highest orders in that chain, respectively, for each  $K\geq p$.
The $K$-chain of the Neumann eigenvalues of  $\luned{n}{p}$, 
$\mu_k=\bar{\mu}_K=K(K+n-1)$,  
 is given by 
 \[k'_-=k'_-(K)\leq k\leq 
 k'_+(K)=k'_+,\]
  where $k'_-=\sigma'_{n,p}(K-1)$, $k'_+=\sigma'_{n,p}(K)-1 $, for each $K\geq 0$.
 The volume and  area of the boundary are given by 
 \begin{equation}\label{areap}
 \begin{array}{l}
 |\lune{n}{\pi/p}|=\fr{1}{2p}|\mathbb{S}^n|=\fr{(n+1)}{2p}\omega_{n+1},\eqskip
 |\partial \lune{n}{\pi/p}| = 2|\{(x_1, \ldots, x_{n-1},r)\in \mathbb{S}^{n-1}, r\geq 0\}|=2|\mathbb{S}^{n-1}_+|=
 |\mathbb{S}^{n-1}|.
 \end{array}
 \end{equation}
From
$\omega_n |\mathbb{S}^n|=\omega_n (n+1)\omega_{n+1}=
\frac{2^{n+1}}{n!}\pi^n$ we get
\begin{equation}\label{cwp}
\Cw_{n,p}=\Cw[\lune{n}{\pi/p}]:= \frac{ 4\pi^2}{(\omega_n|\lune{n}{\pi/p}|)^{2/n}}= (n!p)^{2/n} =  p^{2/n}\Cw_{n,1}.    \end{equation}

\subsection{Relations between multiplicities: Proof of Theorem~\ref{thmxmultiplicity}\label{Sec 4}}
 In this section we compute the coefficients of the Hilbert-Poincar\'{e} series
of $\tilde{\mathcal{W}}=\{\id_{\R^{n-1}}\}\times \mathbb{D}_p$ derived in  Appendix~\ref{ApA}, obtaining the multiplicity of the
eigenvalues of $\lune{n}{p}$ in terms of sums of the multiplicity of a family of $p$-consecutive eigenvalues of the
$(n-1)$-hemisphere, $S_{+}^{n-1}$.
For each $j\geq 0,$  $0\leq r\leq p-1$,  and $ m\geq 0$ define
\begin{equation}\label{dj}
\begin{array}{lll}
d_j :=  {\ds \binom{n-2+j}{j}} = \fr{(j+1)^{\overline{n-2}}}{(n-2)!},\;  d_0=1,\eqskip
S_j :=\dsum_{0\leq j'\leq j}d_{j'} = \binom{n-1+j}{j}  = \fr{(j+1)^{\overline{n-1}}}{(n-1)!}
\end{array}
\end{equation}
and
\begin{equation}
\begin{array}{lll}
C_{(m,p,r)} & : = & d_r+ d_{p+r}+ d_{2p+r}+\cdots+d_{mp+r}\eqskip
& \ = &~\dsum_{s=0}^{m}{{n-2+s p +r}\choose{s p+ r}}\eqskip
& \ = & \dsum_{s=0}^m\fr{(sp+r+1)^{\overline{n-2}}}{(n-2)!}~=~\dsum_{s=0}^m m_{n-1,1}(sp+r+1).\label{Cmpr}
\end{array}
\end{equation}
%

Now we develop the series obtained in Proposition~\ref{Inv-Anti} for any $p\geq 1$.

\begin{lemma}\label{PSlune}  Let  $p\geq 1$, and $n\geq 2$.
The Hilbert-Poincar\'{e} series of the $\tilde{\mathcal{W}}$-invariant and anti-invariant harmonic homogeneous polynomials on $\mathbb{R}^{n+1}$,  $H^i(n+1)$ and $H^a(n+1)$, respectively, are given by 
\begin{eqnarray*}
\begin{array}{ll}
\mathcal{HP}_{H^i(n+1)}(T) &=
\fr{1}{(1-T^p)(1-T)^{n-1}} = \dsum_{m\geq 0}\dsum_{ r\geq 0}^{ p-1}\, C_{(m,p,r)}\,\, T^{mp+r}\eqskip
\mathcal{HP}_{H^a(n+1)}(T) &= \fr{T^p}{(1-T^p)(1-T)^{n-1}}=\dsum_{m\geq 0}\dsum_{r\geq 0}^{  p-1}\,C_{(m,p,r)}\,\,T^{(m+1)p+r}
\end{array}
\end{eqnarray*}
\end{lemma}

\begin{proof}
The first equalities are given in  \eqref{Fi}.
 The second equality for the first identity is obtained developing~\eqref{Fi} by applying~\eqref{PS1} and using
 $(1-T^m)^{-1}=\dsum_{k\geq 0}T^{mk}$,
\[
 \begin{array}{lll}
\mathcal{HP}_{H^i(n+1)}(T) & = & \dsum_{m\geq 0}T^{pm}\sum_{R\geq 0}d_RT^R\eqskip
& = & \left(1+ T^p+ T^{2p}+\ldots\right)\left(d_0 +d_1T^1+ \ldots d_{s}T^{s}+\ldots \right)\eqskip
& = & d_0+d_1T+d_2T^2+\ldots + d_{p-1}T^{p-1}\eqskip
& & \hspace*{5mm} + (d_p+d_0)T^p +(d_{p+1}+d_1)T^{p+1}\eqskip
& & \hspace*{10mm}+ (d_{p+2}+d_2)T^{p+2}+\ldots + (d_{p+p-1} +d_{p-1})T^{p+p-1}\eqskip
& & \hspace*{15mm}+ (d_{2p}+d_{p}+d_0)T^{2p+0}
+ (d_{2p+1}+d_{p+1}+d_1)T^{2p+1}\eqskip
& & \hspace*{20mm}+\ldots+ (d_{2p+p-1} +d_{p+p-1}+ d_{p-1})T^{2p+p-1}\eqskip
& &  \hspace*{25mm}+ \ldots\\
&=& \dsum_{m\geq 0}\, \dsum_{0\leq r\leq p-1}(d_r+ d_{p+r}+ d_{2p+r}+\ldots+d_{mp+r})\,
T^{mp+r},
\end{array}
\]
The second identity follows trivially from the first.
\end{proof}

 Theorem~\ref{thmxmultiplicity} is part of next lemma. 
\begin{lemma}\label{lemmaA}
Let $p\geq 2$ and $n\geq 2$.
\begin{enumerate}[{\rm 1)}]
\item The  Dirichlet eigenvalues of $\lune{n}{\pi/p}$ are the eigenvalues $K(K+n-1)$ of $\mathbb{S}^n$
with $K\geq p$. For each $K=K(m+1,r)\geq p$, $K(K+n-1)$ is the
 $(K-p+1)$-th eigenvalue of $\lune{n}{\pi/p}$
 and its multiplicity and
the  sum of multiplicities up to the $(K-p+1)$-th  eigenvalue,  $\dsum_{j=p}^{K}m_{n,p}[j(j+n-1)]$,
are  given by the following equivalent identities, respectively
\[
\begin{array}{lll}
m_{n,p}[K(K+n-1)] & = &
 \dsum_{s=0}^m\frac{(sp+r+1)^{\overline{n-2}}}{(n-2)!}\eqskip
& = & \dsum_{s=0}^mm_{n-1,1}(sp+r+1).
\end{array}
\]
\[
\begin{array}{lll}
\sigma_{n,p}[K(K+n-1)]
 & = & \dsum_{s=0}^m\frac{(sp+r+1)^{\overline{n-1}}}{(n-1)!}\eqskip
 & = & \dsum_{s=0}^m m_{n,1}(sp+r+1).
\end{array}
\]
\noindent
\item The  Neumann eigenvalues  of $\lune{n}{\pi/p}$ are the  eigenvalues  $K(K+n-1)$ of $\mathbb{S}^n$ with $K\geq 0$.
For each $K=K(m,r)\geq 0$, the multiplicity of the $K$-th  Neumann eigenvalue $K(K+n-1)$ and the sum of multiplicities up to the $K$-th eigenvalue, $\sum_{j=0}^K m'_{n,p}[j(j+n-1)]$, are given by the following equivalent identities, respectively
\[
\begin{array}{lll}
m'_{n,p}[K(K+n-1)] & = &\dsum_{s=0}^{m}\frac{(sp+r+1)^{\overline{n-2}}}{(n-2)!}\eqskip
& = & \dsum_{s=0}^{m}m'_{n-1,1}(sp+r).
\end{array}
\]
\[
\begin{array}{lll}
\sigma'_{n,p}[K(K+n-1)] & = &  \dsum_{s=0}^m\frac{(sp+r+1)^{\overline{n-1}}}{(n-1)!}\eqskip
& = & \dsum_{s=0}^m m'_{n,1}(sp+r).
\end{array}
\]
\end{enumerate}
\end{lemma}
\begin{proof}
 1)
A Dirichlet eigenvalue of $\lune{n}{\pi/p}$  is an eigenvalue of $\mathbb{S}^n$, 
$K(K+n-1)$,  such that the  $T^K$-coefficient of the Hilbert-Poincar\'{e} series $\mathcal{HP}_{H^a(n+1)}(T)$ given in Lemma~\ref{PSlune} is not zero,  that is $K\geq p$.   In this case it  gives $\dim(H^a_K(n+1))$.
Using the notation in (\ref{K(m,r)}) we  take $K=K(m+1,r)=(m+1)p+r$
with $m\geq 0$ and $0\leq r\leq p-1$. 
The first eigenvalue of $\luned{n}{p}$ is when
$K(1,0)$, i.e. 
$\bar{\lambda}_p(\mathbb{S}^n)=p(p+n-1)$. 
Then, $m_{n,p}(K(m+1,r))= C_{(m,p,r)}$
defined in (\ref{Cmpr}).
\[
\begin{array}{lll}
\sigma_{n,p}(K(m+1,r)) & = & m_{n,p}(K(1,0))+ \ldots + m_{n,p}(K(1,p-1))\eqskip
&& \hspace*{2mm} +m_{n,p}(K(2,0))+\ldots+ m_{n,p}(K(2,p-1))\eqskip
&& \hspace*{4mm}+\ldots\eqskip
&& \hspace*{6mm} + m_{n,p}(K(m,0))+ \ldots +m_{n,p}(K(m,p-1))\eqskip
&& \hspace*{8mm}  + m_{n,p}(K(m+1,0))+\ldots + m_{n,p}(K(m+1,r))
\eqskip
& = &  \dsum_{\scriptsize \begin{array}{c}
0\leq l\leq m-1\\ 0\leq r'\leq p-1\end{array}} m_{n,p}(K(l+1,r'))\eqskip
& & \hspace*{5mm}  + m_{n,p}(K(m+1,0))+\ldots + m_{n,p}(K(m+1,r))\eqskip
& = &\dsum_{ 0\leq l\leq m-1}\dsum_{\scriptsize \begin{array}{c}  0\leq r'\leq p-1\end{array}}
(d_{r'}+ d_{p+r'}+ d_{2p+r'}+\ldots+d_{lp+r'})\eqskip
&&
\hspace*{5mm} + \dsum_{0\leq r'\leq r} (d_{r'}+ d_{p+r'}+ d_{2p+r'}+\ldots+d_{mp+r'}),
\end{array}
\]
where the first term in $\sum_{l,r'}$ is ignored if $m=0$. Rearranging the sums,

\[
\begin{array}{lll}
\sigma_{n,p}(K(m+1,r))&=& \Big{[}d_0+ d_1+\ldots+ d_{p-1}\Big{]}\eqskip
&&\hspace*{2mm} + \Big{[}   \left[d_0+ d_1+\ldots+ d_{p-1}\right]+
\left[d_p+ d_{p+1}+\ldots+ d_{p+ (p-1)}\right]\Big{]}\eqskip
&& \hspace*{4mm} + \ldots+ \Big{[} \left[d_0+ d_1+\ldots+ d_{p-1}\right]+
\left[d_p+ d_{p+1}+\ldots+ d_{p+ (p-1)}\right]\eqskip
&&\hspace*{8mm}
+ \ldots+\left[d_{(m-1)p}+ d_{(m-1)p+1}+\ldots+ d_{(m-1)p+ (p-1)}\right] \Big{]}\eqskip
&& \hspace*{12mm}+
\dsum_{0\leq r'\leq r} (d_{r'}+ d_{p+r'}+ d_{2p+r'}+\ldots+d_{mp+r'})
\end{array}
\]
\[
\begin{array}{lll}
\phantom{\sigma_{n,p}(K(m+1,r))}
&=& m \left(d_0+ d_1+\ldots+ d_{p-1}\right)
\hspace*{2mm}+ (m-1) \left(d_{p}+ d_{p+1}+\ldots+ d_{p+ (p-1)}\right) \eqskip
&& \hspace*{4mm}+(m-2)\left(d_{2p}+ d_{2p+1}+\ldots+ d_{2p+ (p-1)}\right)\eqskip
&& \hspace*{6mm}+\ldots\eqskip
&& \hspace*{8mm}+ \left(d_{(m-1)p}+ d_{(m-1)p+1}+\ldots+ d_{(m-1)p+ (p-1)}\right)\eqskip
&& \hspace*{10mm}+ \dsum_{0\leq r'\leq r} (d_{r'}+ d_{p+r'}+ d_{2p+r'}+\ldots+d_{mp+r'})
\end{array}
\]
\[
\begin{array}{lll}
\phantom{\sigma_{n,p}(K(m+1,r))}
&=&  \Big{(}\left(d_0+d_1+\ldots+d_{p-1}\right) +\left(d_p+d_{p+1}+\ldots+d_{p+ (p-1)}\right)\eqskip
&&\hspace*{4mm} +\ldots+
\left(d_{(m-1)p}+d_{(m-1)p+1}+\ldots+d_{(m-1)p+(p-1)}\right)\Big{)}\eqskip
&&\hspace*{8mm}+  \Big{(} \left(d_0+d_1+\ldots+d_{p-1}\right)
+\left(d_p+d_{p+1}+\ldots +d_{p+ (p-1)}\right)\eqskip
&& \hspace*{12mm}+\ldots+\left(d_{(m-2)p}+d_{(m-2)p+1}+\ldots+d_{(m-2)p+(p-1)}\right)\Big{)}\\
&& \hspace*{14mm}+\ldots\eqskip
&& +  \Big{(} \left(d_0+d_1+\ldots+d_{p-1}\right)\Big{)}
+ \dsum_{0\leq r'\leq r} \Big{(}d_{r'}+ d_{p+r'}+ d_{2p+r'}+\ldots+d_{mp+r'}\big{)}
\end{array}
\]
\[
\begin{array}{lll}
\phantom{\sigma_{n,p}(K(m+1,r))}
& = & \dsum_{j=1}^m  \Big{(} d_0+d_1+\ldots +d_p +\ldots +  d_{(m-j)p+(p-1)}\Big{)}\eqskip
&& \hspace*{4mm} + \dsum_{0\leq r'\leq r} \Big{(}d_{r'}+ d_{p+r'}+ d_{2p+r'}+\ldots+d_{mp+r'}\Big{)}\eqskip
& = & \dsum_{j=1}^m  S_{(m-j)p+(p-1)}
+S_r + (S_{p+r}-S_{p-1}) +(S_{2p+r}-S_{2p-1})\eqskip
& & \hspace*{4mm} + \ldots + (S_{mp+r}-S_{mp-1})\eqskip
& = & \fr{1}{(n-1)!}\Big{(}\dsum_{j=1}^m  ((m-j+1)p)^{\overline{n-1}}+ \dsum_{s=0}^m\Big{(}(r+sp+1)^{\overline{n-1}}-(sp)^{\overline{n-1}}\Big{)} \Big{)}\eqskip
& = &\fr{1}{(n-1)!}\Big{(}\dsum_{\mu=1}^m  (\mu p)^{\overline{n-1}}+ \dsum_{s=0}^m\Big{(}(r+sp +1)^{\overline{n-1}}-(sp)^{\overline{n-1}}\Big{)} \Big{)}\eqskip
& = &  \fr{1}{(n-1)!}\dsum_{s=0}^m(sp+r+1)^{\overline{n-1}}.
\end{array}
\]
\noindent
2) The proof for Neumann eigenvalues is similar, where now we take $K=K(m,r)$, with $m\geq 0$, i.e $K\geq 0$. The multiplicity of $K(K+n-1)$ as a Neumann eigenvalue of $\lune{n}{\pi/p}$ is the  non-zero $T^K$-coefficient of  the  series $\mathcal{HP}_{H^i(n+1)}(T)$, given by $\dim(H^i_K(n+1))$ that is,
\[m'_{n,p}(K(m,r))= C_{(m,p,r)}= m_{n,p}(K(m+1,r))=
(d_r+ d_{p+r}+ d_{2p+r}+\ldots+d_{mp+r}).\]
and the sum of multiplicities by
\begin{gather*}
\sigma'_{n,p}(K(m,r)) =
\!\!\!\! \dsum_{\scriptsize
\begin{array}{c}
0\leq l\leq m-1\\[-1mm] 0\leq r'\leq p-1\end{array}} \!\!\!\!\!\!m'(K(l,r'))~~~ + m'(K(m,0))+\ldots m'(K(m,r))~=~\sigma_{n,p}(K(m+1,r)).
\end{gather*}
\end{proof}

\begin{remark}
The above expressions extend naturally to the case $p=1$ with $r=0$, for $K=(m+1)$ in the Dirichlet case and  $K=m$ in the Neumann case, for all $m\geq 0$.
Extending the case $p=1$~\cite[Sections 2.1 and 2.2]{fms}, for  all $p\geq 1$, $n\geq 2$ and $K\geq 0$, we have
\[ \begin{array}{l}
m'_{n,p}[K(K+n-1)]=m_{n,p}[(K+p)(K+p+n-1)]
\eqskip
\sigma'_{n,p}[K(K+n-1)]=\sigma_{n,p}[(K+p)(K+p+n-1)].\end{array}
\]
\end{remark}
\begin{remark}
In Lemma~\ref{lemmaA}  the multiplicity of an  eigenvalue $K(K+n-1)$ of $\lune{n}{\pi/p}$  with $K=(m+1)p+r$ is given in terms of a sum
of the multiplicities of $m+1$  
  eigenvalues $sp+r+1$ of the $(n-1)$-hemisphere by consecutive   multiples of $p$.  A converse also holds, namely,
 the multiplicity of any   eigenvalue of the hemisphere with $K\geq p$ can be expressed in terms of the sum of the multiplicities of $p$
 consecutive  eigenvalues of the lune. Let $K'=(m+1)p+r_1$ where $m\geq 0$, $0\leq r_1\leq p-1$, and consider the $p$-consecutive eigenvalues
 $K_i(K_i+n-1)$, $i=0,\ldots, p-1$, where $K_i=K'+i$ with $K_{0}=K'$.
Then the following relations hold for the Dirichlet eigenvalues.
\begin{equation}\label{consecutivep-sum2}
\begin{array}{lllll}
\dsum_{i=1}^{p}m_{n,p}[K_i(K_i+n-1)]
&=& \fr{((m+1)p+r_1)^{\overline{n-1}}}{(n-1)!} & = & m_{n,1}[K'(K'+n-1)],\eqskip
\dsum_{i=1}^{p}\sigma_{n,p}[K_i(K_i+n-1)]
&=& \fr{((m+1)p+r_1)^{\overline{n}}}{n!} & = & \sigma_{n,1}[K'(K'+n-1)].
\end{array}
\end{equation}
  Moreover, the Neumann version is similar to the above two equalities, replacing  $m_{n,p}$ by $m'_{n,p}$,  and $\sigma_{n,p}$ by
$n'_{n,p}$ and taking now $ K'=mp+r_1$
where $m\geq 0$ and $0\leq r_1\leq p-1$.
 
The proof consists in  parameterizing $K_i$ as
$K(m+1,r)$ for $r_1\leq r\leq  p-1$, and as $K(m+2,r)$ for $0\leq r\leq r_1-1$, where    $K(m+2,r)$ is assumed empty if $r_1=0$.
Assuming first that $r_1=0$, that is, $K'=(m+1)p=K_0$,  the
parallel summation formula (\ref{parallelsum}) gives
\begin{gather*}
\begin{array}{lll}
\dsum_{r=0}^{p-1}m_{n,p}(K(m+1,r)) & = & \dsum_{r=0}^{p-1}\dsum_{s=0}^{m}\frac{(sp +r+1)^{\overline{n-2}}}{(n-2)!}
 = \dsum_{l=1}^{(m+1)p} \frac{l^{\overline{n-2}}}{(n-2)!}\eqskip
 & = & \fr{\left( (m+1) p\right)^{\overline{n-1}}}{(n-1)!}=m_{n,1}(K')
 \end{array}
\end{gather*}
and
\begin{gather*}
\begin{array}{lll}
\dsum_{r=0}^{p-1}\sigma_{n,p}(K(m+1,r))
 & = &\fr{\left( (m+1)p\right)^{\overline{n}}}{n!}=\sigma_{n,1}(K'),
 \end{array}
\end{gather*}  
We get the same equality for $1\leq r_1\leq p-1$, following the case $r_1=0$, and using it, splitting the sum at $r=r_1$.
The sum of the multiplicities follows as well.
The case of Neumann eigenvalues can be obtained  from the Dirichlet case using  the relations  in the previous remark.
 \end{remark}
\section{P\'{o}lya inequalities for $\lune{n}{\pi/p}$}\label{Sec 5}

\subsection{P\'{o}lya inequalities on dihedral lunes\label{introEIGEN}}

Consider the Dirichlet eigenvalues of the $K$- chain of $\luned{n}{p}$, $\lambda_k=\lambda_k(\luned{n}{p})=K(K+n-1)$,  where $K=K(m+1,r)$, with  $m\in \mathbb{N}_0$, $r=0, \ldots, p-1$.
 Each eigenvalue $\lambda_k$ of the $K$-chain is defined by its order $k$ such that $k_-=k_-(K)\leq~ k~ \leq k_+(K)=
k_+$. Conversely, writing the order $k$ {as} $k_-(K)+j$ defines $K$ and $j\in\{0, \ldots, m_{n,p}(K)-1\}$ as follows
(see Section~\ref{Notations}). From  Lemma~\ref{lemmaA},   we have 
\begin{gather}\label{k+}
\begin{array}{l}
k_+(K)=\sigma_{n,p}(K(m+1,r))=\frac{1}{(n-1)!}\Big{(} \dsum_{s=0}^m(sp +r+1)^{\overline{n-1}} \Big{)},\eqskip
k_-(K)=\sigma_{n,p}(K(m+1,r)-1)+1
= \frac{1}{(n-1)!}\left(\dsum_{s=0}^{m}(sp +r)^{\overline{n-1}}\right) +1,
\eqskip
k =k_-(K)+j=k_+(K)-m_{n,p}(K)+1+j \quad\quad\mbox{for}~0\leq j\leq m_{n,p}(K)-1.
\end{array}
\end{gather}
Since $\Cw_{n,p}= (n!p)^{2/n}$, for $p\geq 2$, $m\geq 0$, and  $0\leq r\leq p-1$, we have the identities
\begin{gather}\label{cwsnp}
\begin{array}{l}
\Cw_{n,p}\left(k_+(K)\right)^{2/n}=
(np)^{2/n}\Big{(} \dsum_{s=0}^m(sp +r+1)^{\overline{n-1}} \Big{)}^{2/n},
\eqskip
\Cw_{n,p}\left(k_-(K)\right)^{2/n}=
\Big{(} \dsum_{s=0}^m(np)(sp +r)^{\overline{n-1}}+(n!p) \Big{)}^{2/n}.
\end{array}
\end{gather}

Given $n\geq 2$, $p\geq 1$, $r=0,\ldots, p-1$, $m\geq 0$, we define
\begin{gather*} 
\left\{\begin{array}{ll}
A_n(m,p,r)=(\bar{\lambda}_{(m+1)p+r})^n=\Big{(}\big{(}(m+1)p+r\big{)}\big{(}(m+1)p+r+n-1\big{)}\Big{)}^{n}, \quad  \eqskip
B_n(m,p,r)= (np)^2\left(\dsum_{s=0}^m(sp+r+1)^{\overline{n-1}}\right)^2,
\end{array}\right.
\end{gather*}
and a family of $p$ polynomial functions on the variable $m$, parameterised by $r=0, \ldots, p-1$,
 \begin{equation}\label{M}
M_{n,p,r}(m):=A_n(m,p,r)-B_n(m,p,r).
\end{equation}
The next lemma follows immediately from the first identity in~\eqref{cwsnp}.
%
\begin{lemma}[P\'{o}lya's inequality] \label{3.4} The Dirichlet eigenvalue $K(K+n-1)$ of $\luned{n}{p}$, where  $K=(m+1)p+r$,
satisfies P\'{o}lya's inequality~\eqref{PolyaExtreme} if and only if $M_{n,p,r}(m)\geq 0$, that is
%
\begin{equation}\label{polyanp}
\big{(}(m+1)p+r\big{)}\big{(}(m+1)p+r+n-1\big{)}^{ n}\geq (np)^2\Big{(}\sum_{s=0}^m(sp+r+1)^{\overline{n-1}}\Big{)}^2.
\end{equation}
\end{lemma}
\begin{proposition} [First eigenvalue] \label{First}
Let {$p\geq 1$}. The first Dirichlet eigenvalue of $\luned{n}{p}$, $p(p+n-1)$,  satisfies P\'{o}lya  inequality  if and only if  $M_{n,p,0}(0)\geq 0$,
that is, the following inequality holds
\begin{equation}\label{BigBang}
 p^{n-2}(p+n-1)^n\geq (n!)^2. 
 \end{equation}
 This is the case when $n=2$. If $3\leq n\leq 8$, the first eigenvalue satisfies P\'{o}lya's inequality
 if and only if $p\geq 2$.  It is not
 satisfied for $p\leq  \Big{(}\frac{-1+\sqrt{1+4 e^{-2}}}{2}\Big{)} n \approx 0.120754 \, n$, for any $n\geq 2$.
 On the other hand if $p$ is such that
 \begin{equation}\label{firstyes}
  p > \left[ -\fr{1}{2} +\fr{\sqrt{4(2\pi n^3 e^3)^{1/n}+e^2}}{2e}\right] (n-1),
 \end{equation}
 then the first Dirichlet eigenvalue of $\luned{n}{p}$ satisfies P\'{o}lya's conjecture.
\end{proposition}
\begin{proof}
 From  Lemma~\ref{3.4}, P\'{o}lya's condition for the first eigenvalue reads as
\[ p\times (p+n-1)\geq (np)^{2/n}((n-1)!)^{2/n} = (n!p)^{2/n},\]
which is equivalent to condition~\eqref{BigBang}.
The case $n=2$ follows directly, and the cases $n=3, \ldots, 8 $  by direct inspection of (\ref{BigBang}), together with the proof of Lemma~\ref{Dlemma2} and Theorem~\ref{thmC}.
 
 Assume $p=\alpha n$ for some  $0<\alpha\leq (-1+\sqrt{1+4e^{-2}})/2$ and  (\ref{BigBang}) holds.
 Then, using the Stirling lower bound for $n!$, $(\alpha n)^{n-2}((\alpha+1)n)^n> 2\pi e^{-2n}n^{2n-1}$, that is equivalent to
 $e\alpha \big{(}e\alpha +e\big{)}> (2\pi n^3\alpha ^2)^{1/n}>1$. That is $\alpha^2+\alpha -e^{-1}>0$ implying $\alpha > (-1 +\sqrt{1+4e^{-2}})/2$ contradicting the assumption.
 
  When $n$ equals $2$ or $3$ the right-hand side of~\eqref{firstyes} yields approximately $1.63312$ and $2.02424$, for which we
 already know the result holds. We will thus consider $n$ larger than three from now on. Assume $p\geq \alpha (n-1)$ with
 \[
 \alpha = \alpha_{n} = -\fr{1}{2} +\fr{\sqrt{4(2\pi n^3 e^3)^{1/n}+e^2}}{2e}.
 \]
 Then
 \[
 \begin{array}{lll}
  p^{n-2} (p+n-1)^{n} & \geq & \fr{\left[\alpha(\alpha+1)\right]^n}{\alpha^2} (n-1)^{2(n-1)}.\eqskip
  & \geq &  \fr{\left[\alpha(\alpha+1)\right]^n}{\alpha^2} \fr{e^{2(n-1)-1/(6(n-1))}}{2\pi n^2(n-1)}\left(n!\right)^2\eqskip
  & \geq & \left[\alpha(\alpha+1)\right]^n \fr{e^{2(n-1)-1/(6(n-1))}}{2\pi n^3}\left(n!\right)^2,
  \end{array}
 \]
 where the second inequality follows from the Stirling approximation inequality
 \[
  n! \leq \sqrt{2\pi n} \left(\fr{n}{e}\right)^n e^{1/(12n)}
 \]
 applied to $(n-1)^{n-1}$, while the third is a consequence that for $n$ larger than or equal to four the factor $\alpha$ is smaller than one. We thus see that it is enough to show that the factor multiplying $(n!)^2$ on the right-hand
 side is larger than one. This is equivalent to
 \[
  \begin{array}{llll}
    & \left[\alpha(\alpha+1)e^2\right]^n & \geq & 2\pi n^3 e^{2+1/(6(n-1))}\eqskip
    \Leftrightarrow & \alpha(\alpha+1)e^2 & \geq & \left( 2\pi n^3 e^{2+1/(6(n-1))}\right)^{1/n}.
  \end{array}
 \]
 The above inequality is satisfied if we replace $2+1/(6(n-1))$ by $3$ and solve the resulting inequality for $\alpha$,
 yielding the desired result.
 \end{proof}
\begin{remark}
(1) \label{remark_examples1} 
 The following examples illustrate possible different behaviours, including the fact that having the first eigenvalue
 satisfy~\eqref{polyaconject} is not sufficient for the conjecture to hold.
 All eigenvalues of $\luned{9}{2}$ satisfy P\'{o}lya inequality except for the first. Moreover $\luned{9}{p}$ satisfies P\'{o}lya
 conjecture for all $p\geq 3$. In case of $\luned{32}{5}$, if $r=0,1,2$ then $M_{32,5,r}(0)<0$, but for $r=3,4$, $M_{32,5,r}(m)>0$
 for all $m\geq 0$, that is, $r$ is not a secondary parameter. We also note that $\luned{32}{p}$ satisfies P\'{o}lya conjecture for
 all $p\geq 6$. On the other hand the first and the third eigenvalues of $\luned{24}{4}$ satisfy P\'{o}lya inequality, namely
 $M_{24,4,r}(0)>0$, for $r=0,2,3$ respectively, but  $M_{24,4,1}(0)<0$, i.e the second does not.
\end{remark}

\begin{definition}\label{DEF5.2} Define an integer $\pn\geq 1$ with respect to $M_{n,\cdot,\cdot}(\cdot)$ given by~\eqref{M} if it is such that for all $p\geq \pn$ and  $r=0, \ldots,p-1$,  P\'{o}lya's inequality \eqref{polyanp} holds for all $m\geq 0$. The smallest of such integers is denoted by $\pin(n)$. Similarly, we define $p_1(n)$ and $p_1^*(n)$ with respect to the first eigenvalue.
\end{definition}
\begin{remark}
Proposition~\ref{First} implies that $\pione(n)$ grows to infinity linearly in $n$, with
 \[\fr{1}{2}\left(-1+\sqrt{1+4e^{-2}}\right)n~\leq ~\pione(n) ~< ~\fr{1}{2}\left(-1+\sqrt{1+4e^{-2}(2\pi n^3e^3)^{1/n}}\right)(n-1),\]
 and we also see that the limit
 of the factor $\alpha_{n}$ on the right is
 $ \fr{1}{2}\left( \sqrt {1 + 4e^{-2}} - 1 \right)\approx 0.120754.$
\end{remark}

\section{A generalised two-term asymptotic formula\label{Sec Third term}}
The two-term asymptotic formulas (\ref {2termweyl}) and  (\ref{2termweylNeumann}) hold on domains satisfying the nonperiodicity and nonblocking
conditions. This is not the case of the  rational lunes whose geodesic billiards    are all periodic (Theorem~\ref{thmx_blunes}).
We derive {a generalised two-term} asymptotic formula of the type of  \cite[1.7.1]{sava} for the Dirichlet eigenvalues of  dihedral lunes $\lune{n}{\pi/p}$
replacing the constant coefficient $\sqrt{\Cw_{n,p}}\,p$ of the second term in~\eqref {2termweyl} by a  bounded function  $c(k,p)$
that depends on $n$ and the order  $k$ of $\lambda_k$ in the corresponding $K$-chain,  given by a  parameter $j$ that runs from  zero up to $m_{n,p}
(K)-1$. Moreover, $c(p,k)$  is sharp, being constant along lower and higher  orders where it takes its maximum and minimum values, respectively.
Due to the fact that the formula and its proof for lunes on the two-dimensional sphere is considerably simpler than the general case
and also allows for the derivation of sharp two-sided inequalities, we consider this case first.

We use the notations given in Section~\ref{Notations}, and the formulas in  Lemma~\ref{lemmaA} and in the introduction of  Section~\ref{introEIGEN}.

 \subsection{${n=2}$ and ${p\geq 1}$}
 \label{C6.1}
The multiplicity of $\bar{\lambda}_K=K(K+1)$ for $K=K(m+1,r)=(m+1)p+r$, $r=0,\ldots p-1$, is given by
$ m_{2,p}(K)=m+1$. Let $\lambda_k$ in the $K$-chain. Then $k= k_-(K)+ j= k_+(K)-m+ j$ for a unique $j=0, \ldots m_{2,p}(K)-1$. Therefore
 \begin{eqnarray*}\begin{array}{ll}
 \lambda_k=m^2p^2 +mp(2p+2r+1)+(p+r)(p+r+1)=:Q_2(m)&
 \eqskip
 k_+(K)=\dsum_{s=0}^m(sp+r+1)=
 \fr{m(m+1)}{2}p +(m+1)(r+1),& \eqskip
 k_-(K)= \fr{m(m+1)}{2}p +(m+1)r+1, & \eqskip
 k=
 \fr{m(m+1)}{2}p +(m+1)r+1+j. &\\
   \end{array}
 \end{eqnarray*}
 Recall that  $\Cwnp{2}{p}=2p$ and the second constant in~\eqref{2termweyl} is $\sqrt{C_{W 2,p}}p=\sqrt{2p}p$.
 \begin{proposition} \label{PropC.1}
 Assume {$p\geq 1$}. The following generalised 2-term asymptotic formula holds on $\lune{2}{\pi/p}$ for $k\to +\infty$,
 \begin{equation}\label{Gen2term}
 \lambda_k=\Cwnp{2}{p} k + c(p,k)k^{1/2} +\so(k^{1/2}),
 \end{equation}
 with 
 \[c(p,k)= \sqrt{\Cwnp{2}{p}}\, p+\sqrt{\Cwnp{2}{p}}\left(1-2\frac{j}{m}\right),\]
 where $k$ is in the $K$-chain and
$j=k-k_-(K)\in \{0, \ldots, m_{2,p}(K)-1\}$. This
 second coefficient function $c(p,k)$ is bounded and  sharp along highest and lowest orders where equality is achieved, 
 \[\sqrt{\Cwnp{2}{p}}(p-1) = c(p,k_+(K))\leq c(p,k)= \sqrt{2p}\left(p+1-2\frac{j}{m}\right)\leq c(p,k_-(K))=\sqrt{\Cwnp{2}{p}}(p+1).\]
\end{proposition}
\begin{proof}
   Consider the following function
 \begin{gather*}
 C_{p}(r,j,m) :=
  \frac{\lambda_k-\Cwnp{2}{p}k}{k^{1/2}}={
 \frac{
 m\left[ p^2+p - \frac{2p j}{m}\right]+p(p-1)+r(r+1)}{\sqrt{\frac{m^2}{2}p+m\left(\frac{p}{2}+r+\frac{j}{m}\right)+r+1}}}.
 \end{gather*}
  Take $k\to +\infty$, that is $m\to +\infty$. Then,
\begin{eqnarray*}
\lefteqn{\lim_{k\to +\infty} \frac{\lambda_{k}-\Cwnp{2}{p}k -c(p,k)k^{1/2}}{k^{1/2}}
=\lim_{k\to +\infty}\left(C_p(r,j,m)-c(p,k)\right) }\\
&=& 	\lim_{k\to +\infty}\left({\frac{m\left(p^2+p-2p\frac{j}{m}\right)}{\frac{m\sqrt{p}}{\sqrt{2}}}}-c(p,k)\right)=
\lim_{k\to +\infty}\left(c(p,k)-c(p,k)\right)=0
\end{eqnarray*}
The last statement follows  trivially using $j=m$ and $j=0$.
 \end{proof}
 Next we prove two-sided, two-term, inequalities holding for all Dirichlet eigenvalues.
\begin{proposition}
\begin{enumerate}[{\rm 1)}]
\item The following two inequalities hold for all $p\geq 1$ and $k\geq 1$,
\begin{gather}\label{twotermineq}
\quad\Cwnp{2}{p} k+\sqrt{\Cwnp{2}{p}}(p-1)k^{1/2}~\leq ~ \lambda_k~
 \leq~ C_{W 2,p} k+\sqrt{\Cwnp{2}{p}}(p+1)k^{1/2}+ p(p-1)+r(r+1).
\end{gather}
 The right-hand side is asymptotically sharp along the lowest orders when $k\to +\infty$ in the sense 
 \[\frac{\lambda_k- \Cwnp{2}{p} k_--\sqrt{\Cwnp{2}{p}}(p+1)k_-^{1/2}}{k_-^{1/2}}\to 0,\]
 when $k_-=k_-(K)\to +\infty$. The  left-hand side holds for all $p\geq 1$, and is also asymptotically
 sharp along the highest orders when $k\to +\infty$ (in the same sense). \\[1mm]
 \item For $1\leq p\leq 8$ the right-hand side in the inequality above may be improved to,
 \begin{eqnarray}\label{ddag}
  \lambda_k &\leq &\Cwnp{2}{p} k+\sqrt{\Cwnp{2}{p}}(p+1)
  k^{1/2}+ \frac{p(p-3)}{2}+r(r+1).
\end{eqnarray}
\end{enumerate}
\end{proposition}
\begin{proof}
We look for the largest $c'(p)> 0$,  such that for all $K\geq p$ along the highest order eigenvalues $k=k_+=k_+(K)$ we have
\begin{eqnarray*}
\bar{\lambda}_{K}-\Cwnp{2}{p}k_+
&=&(m+1)p^2-(m+1)p+r(r+1)\nonumber \\
&\geq&
\label{lowerbound} c'(p)\sqrt{\left(\sum_{s=0}^m(sp+r+1)\right)}, 
\end{eqnarray*}
that is
\begin{equation}\label{c'1}
\left((m+1)(p(p-1))+r(r+1)\right)^2 \geq {c'}(p)^2\left(\frac{m(m+1)}{2}p + (m+1)(r+1)\right).
\end{equation}
Note that,
\[
\lim_{m\to +\infty}\frac{\left((m+1)p(p-1)+r(r+1)\right)}{
\sqrt{\left(\frac{m(m+1)}{2}p+ (m+1)(r+1)\right)}}= \sqrt{2p}(p-1)= \sqrt{\Cwnp{2}{p}}(p-1).
\]
Then we take $c'(p)=c^-(p)$, and verify  inequality (\ref{c'1}) holds, or equivalently,
\begin{equation*}
(m+1)p^2(p-1)^2+r^2(r+1)^2 \geq 2(m+1)p(p-1)(r+1)(p-(r+1)).
\end{equation*}
The later inequality holds if we prove the following  inequality holds, for all $0\leq r\leq p-1$
\begin{equation*}
p(p-1)\geq 2(r+1)(p-(r+1)).
\end{equation*}
This inequality holds for $p=1$ (when both sides equal zero), and for $p\geq 2$ since $x(p-x)$ takes its maximum at $x= p/2$ when  $x\in [1,p]$, with equality only if $p=2$ and $r=0$.
This proves the right-hand side inequality, that is strict except when $p=2$ and  $m=r=0$. {Since we used highest order eigenvalues the same inequality (\ref{twotermineq}) holds for all $k\geq 1$.  From the  above limit} follows the asymptotic sharpness when $m\to +\infty$ along highest order eigenvalues.

Similarly, for the left-hand inequality we have along the lower order eigenvalues
\[\bar{\lambda}_K-C_{W 2,p}k_-=\left(mp(p+1)+p(p-1)+r(r+1)\right),\]
 and so
\[
\begin{array}{lll}
{\ds \lim_{m\to +\infty}} \fr{\bar{\lambda}_K-C_{W 2,p}k_-}{k_-^{1/2}} & = &
{\ds \lim_{m\to +\infty}}\fr{\left(mp(p+1)+p(p-1)+r(r+1)\right)}{\sqrt{\left(\frac{m(m+1)}{2}p+ (m+1)r+1\right)}}\eqskip
& = & \sqrt{2p}(p+1)\eqskip
\end{array}
\]
We have the following inequalities
\begin{eqnarray}
\lefteqn{\bar{\lambda}_K-C_{W 2,p}k_- -\sqrt{\Cwnp{2}{p}}(p+1)k_-^{1/2}}\nonumber\\
&=&\left(mp(p+1)+p(p-1)+r(r+1)\right) 
-\sqrt{2p}(p+1)\sqrt{\frac{m(m+1)}{2}p+(m+1)r+1}\quad \label{dag}\eqskip
&\leq& mp(p+1)+p(p-1)+r(r+1)
-\sqrt{2p}(p+1)\frac{m\sqrt{p}}{\sqrt{2}}\sqrt{1+\frac{1}{m}+ \frac{2}{m^2p}}\nonumber\eqskip
&\leq&  mp(p+1)+p(p-1)+r(r+1)
-\sqrt{2p}(p+1)\frac{m\sqrt{p}}{\sqrt{2}}\nonumber\eqskip
&=& p(p-1)+r(r+1)\nonumber 
\end{eqnarray}
If $m=0$ we get  
$(\ref{dag})=p(p-1)+r(r+1)-\sqrt{2p}(p+1)\sqrt{r+1}$, so the stated inequality still holds.
If $p\leq 8$, we may take the following step valid for $m\geq 1$ (and $m=0$ since $2\sqrt{2}\sqrt{r+1}\geq \sqrt{p}$),
\begin{eqnarray*}
(\ref{dag})&\leq & mp(p+1)+p(p-1)+r(r+1)
-\sqrt{2p}(p+1)\frac{m\sqrt{p}}{\sqrt{2}}\sqrt{\left(1+\frac{1}{2m}\right)^2}\eqskip
&=& mp(p+1)+p(p-1)+r(r+1)
-\sqrt{2p}(p+1)\frac{m\sqrt{p}}{\sqrt{2}}\left(1+\frac{1}{2m}\right)\eqskip
&=& p(p-1)+r(r+1)-\frac{p(p+1)}{2}
\end{eqnarray*}
\end{proof}
\begin{remark} When $p=1$ the above inequalities~\eqref{twotermineq} and~\eqref{ddag} reduce to
\begin{equation}
 \label{*}2 k \leq ~ \lambda_k~\leq 2 k+2\sqrt{2}k^{1/2}-1
\end{equation}
A inequality on the (late) counting function $N(z)=\#\{k: \lambda_k\leq z\}$ for  the Dirichlet eigenvalues of $\mathbb{ S}^2_+$ is given in \cite[Theorem 3.2.1]{blps}, writing any positive real as $z=w(w+1)$, so we have $N(w(w+1))=N(K(K+1))=k_+(K)=K(K+1)/2$ where $K=\lfloor w \rfloor$, and setting  $\varepsilon:=\varepsilon(w)=w-\lfloor w \rfloor$
\begin{gather}
\frac{w(w+1)}{2}\left(1-\frac{\varepsilon(w)}{\sqrt{w(w+1)}}\right)^2-\frac{\varepsilon(w)}{8\sqrt{w(w+1)}}~\leq ~N(w(w+1))~
 \leq~ \frac{w(w+1)}{2}\left(1-\frac{\varepsilon(w)}{\sqrt{w(w+1)}}\right)^2, \quad\quad \label{**}
 \end{gather}
When $z=K(K+1)=\bar{\lambda}_K$, $K\in \N$, the inequalities are just a double equality $N(K(K+1))=k_+(K)$. By taking $z=(K+\varepsilon)(K+\varepsilon +1)= K(K+1)+\epsilon(2K+1)+\epsilon^2$,
with $\varepsilon\in [0,1)$,   
next we analyse which of the two $(\ref{*})$ and $(\ref{**})$ is sharper, where $j=0, \ldots, K-1$, and $w=K+\varepsilon$. First we compare the l.h.s. of (\ref{*})
with the r.h.s. of (\ref{**}),
that is, 
$(1+j)\leq K$, with   $ 0\leq  \varepsilon^2$. They are equally sharp.
Now  we compare the other two.
Since $k_-=K(K-1)/2 +1$, we are comparing 
$(K-1/2)\leq \sqrt{(K-1)K+2}$, i.e. $0\leq 8$, with
\[(K(K+1)+\varepsilon(2K+1)+\varepsilon^2)\left((2K+1+2\varepsilon)^2-1\right)\leq
4\left(K(K+1)+\varepsilon(2K+1)+\varepsilon^2\right)^2+\frac{1}{16}\]
that turns out to be equivalent to $0\leq \frac{1}{64}$. 
This means they are also sharply equivalent.
\end{remark} 
\noindent
\subsection{$n\geq 3$ and {$p\geq 1$}}\label{C6.2}
  The multiplicity of  $\bar{\lambda}_K=K(K+n-1)$ with  $K=K(m+1,r)=(m+1)p+r$, is given by
  $ m_{n,p}(K)=\frac{1}{(n-2)!}\sum_{s=0}^m(sp +r+1)^{\overline{n-2}}$. Recall the following equalities
 \begin{gather*}\begin{array}{l}
 \bar{\lambda}_K= K(K+n-1)=m^2p^2+mp( 2p+2r+n-1)+(p+r)(p+r+n-1)=: Q_2(m)\\
 k_+(K)=\frac{1}{(n-1)!}\sum_{s=0}^m(sp +r+1)^{\overline{n-1}}\\
 k_-(K)= k_+(K)-m_{n,p}(K)+1,
 \end{array}
 \end{gather*}
and  if $k$ is in the $K$-chain, then  $k= k_-(K)+ j$, where $j=0, \ldots, m_{n,p}(K)-1$. Recall that
 $\Cwnp{n}{p}=(n!p)^{2/n}$ and the second constant of (\ref{2termweyl}) is $\sqrt{\Cwnp{n}{p}}p$. We use simplified notation $k_{\pm}=k_{\pm}(K)$. Next we prove Theorem~\ref{thmxeventually}.
\begin{proposition} \label{PropC3} Assume $p\geq 1$ and $n\geq 3$.
The following generalised $2$-term asymptotic formula on $\luned{n}{p}$ holds for $k\to +\infty$,
\begin{eqnarray}\label{o-3term}
 \lambda_k &=&\Cwnp{n}{p}k^{2/n}+  c(p,k)k^{1/n} +\so(k^{1/n})
\end{eqnarray}
 where
 \[c(p,k)=\sqrt{\Cwnp{n}{p}}\, p+ \sqrt{\Cwnp{n}{p}}\left( 1-2\frac{j}{m_{n,p}(K)-1}\right),\]
satisfies
\begin{gather*}
\sqrt{\Cwnp{n}{p}}(p-1)=~c(p,k_+)~\leq ~ c(p,k)~\leq ~  c(p,k_-)~=~\sqrt{\Cwnp{n}{p}}(p+1).
\end{gather*}
This extends the expression of $c(p,k)$ obtained for $n=2$ in Proposition~\ref{PropC.1}
 \end{proposition}
 \begin{proof}
Denoting by $\sigma_i=\sigma_i(1,\ldots, n-1)$ the elementary symmetric polynomial of degree $i$ in $(n-1)$ variables,  we have
\[
(x+1)^{\overline{n-1}}=(x+1)(x+2)\ldots (x+n-1)=
\dsum_{i=0}^{n-1}\sigma_ix^{n-1-i},
\]
obtaining 
\[
\begin{array}{lll}
k_+(K) & = &  \fr{1}{(n-1)!}\dsum_{s=0}^m(sp+r+1)^{\overline{n-1}}\eqskip
& = & \dsum_{s=0}^{m}\left(  \dsum_{i=0}^{n-1}\fr{\sigma_i}{\sigma_{n-1}}(sp +r)^{n-1-i}\right) \eqskip
& = & \dsum_{s=0}^{m}\left(  \dsum_{i=0}^{n-1}\dsum_{\alpha=0}^{n-1-i}\fr{\sigma_i}{\sigma_{n-1}}\binom{n-1-i}{\alpha}(sp)^{\alpha} r^{n-1-i-\alpha}\right)\eqskip
& = & \dsum_{s=0}^{m}\left(  \dsum_{\alpha=0}^{n-1}\dsum_{i=0}^{n-1-\alpha}\frac{\sigma_i}{\sigma_{n-1}}\binom{n-1-i}{\alpha}(sp)^{\alpha} r^{n-1-i-\alpha}\right)\eqskip
& = & \dsum_{\alpha=0}^{n-1}\left(\dsum_{s=0}^{m}(sp)^{\alpha}\left[\sum_{i=0}^{n-1-\alpha}
 \frac{\sigma_i}{\sigma_{n-1}}\binom{n-1-i}{\alpha} r^{n-1-i-\alpha}\right]\right).
\end{array}
\]
Applying Faulhaber's formula we get
 \begin{eqnarray}
k_+(K) &=&  \sum_{s=0}^m(sp)^0f_0
 +\sum_{s=1}^m s f_1+
\sum_{\alpha=2}^{n-1}\left(\frac{m^{\alpha+1}}{\alpha+1}+\frac{m^{\alpha}}{2}+\frac{1}{\alpha+1}\sum_{l=2}^{\alpha}\binom{\alpha+1}{l}\B_lm^{\alpha+1-l}\right)f_{\alpha}\nonumber\eqskip
&=& (m+1)f_0+\frac{m(m+1)}{2}f_1+ \sum_{\alpha=2}^{n-1}\left(\frac{m^{\alpha+1}}{\alpha+1}+\frac{m^{\alpha}}{2}+\frac{1}{\alpha+1}
\sum_{l=2}^{\alpha}\binom{\alpha+1}{l}\B_lm^{\alpha+1-l}\right)f_{\alpha}\nonumber\eqskip
&=&
\left(\fr{m^n}{n}+\frac{m^{n-1}}{2}+\frac{1}{n}\sum_{l=2}^{n-1}\binom{n}{l}\B_lm^{n-l}\right)f_{n-1}\nonumber\eqskip
&&\hspace*{5mm}+\left(\frac{m^{n-1}}{n-1}+\frac{m^{n-2}}{2}+\frac{1}{n-1}\sum_{l=2}^{n-2}\binom{n-1}{l}\B_lm^{n-1-l}\right)f_{n-2}\nonumber\eqskip
&&\hspace*{10mm}+\ldots    \label{kmaisn}\eqskip
&&\hspace*{15mm}+\left(\frac{m^3}{3}+\frac{m^2}{2}+\frac{1}{3} \binom{3}{2}\B_2m\right)f_2+ \frac{m(m+1)}{2}f_1+(m+1)f_0\nonumber
\end{eqnarray}  
where $B_l$ are the Bernoulli numbers and    $f_{\alpha}$ are constants that  depend on $n,p,r$  given by
\[\begin{array}{ll}
f_{\alpha}=p^{\alpha}\dsum_{i=0}^{n-1-\alpha}
 \fr{\sigma_i}{\sigma_{n-1}}\binom{n-1-i}{\alpha} r^{n-1-i-\alpha}, & \alpha=0,1,2,  \ldots, n-1.\eqskip
 f_{n-1}=\fr{p^{n-1}}{(n-1)!}.
 \end{array}\]
 If $n=3$ in~\eqref{kmaisn}, then  $f_{n-1}=f_2$, $f_{n-2}=f_1$, and the corresponding coefficients coincide since the summation on $2\leq l\leq n-2=1$ is empty.
 Therefore, for any  $n\geq 3$ the following holds
\begin{eqnarray}\label{kmais2}
k_+(K) =  m^{n}\left(\frac{f_{n-1}}{n}\right) +m^{n-1}\left(\frac{f_{n-1}}{2}+\frac{f_{n-2}}{n-1}\right) +\bo(m^{n-2}),\\
\label{tilde2term}
m_{n,p}(K) = m^{n-1}\left(\frac{\tilde{f}_{n-2}}{n-1}\right) +m^{n-2}\left(\frac{\tilde{f}_{n-2}}{2}+\frac{\tilde{f}_{n-1}}{n-2}\right) +{\bo(m^{n-3})},
\end{eqnarray}
where $\bo(m^{k})$ is a polynomial of degree ${k}$ on the variable $m$, and  we use the tilde $\tilde{f}_{\alpha}$  with respect to
$m_{n,p}(K)$, and 
  $\tilde{\sigma}_i$ the $i-$elementary symmetric function on $n-2$ variables applied to $(1,2, \ldots, n-2)$. That is,  $\tilde{f}_{\beta}$ are given by
\[\begin{array}{l}
\tilde{f}_{\beta}=p^{\beta}\dsum_{i=0}^{n-2-\beta}
 \fr{\tilde{\sigma}_i}{\tilde{\sigma}_{n-2}}\binom{n-2-i}{\tilde{\beta}} r^{n-2-i-\beta}, \quad \beta=0,1, 2, \ldots, n-2.
 \end{array}\]

 Next we compute $k= k_-(K)+ j= k_+(K)-m_{n,p}(K)+1+j$, $j=0, \ldots, m_{n,p}(K)-1$.
 Observe that $m_{n,p}(K)/m^{n-1}$ is bounded. Now we have,
\begin{eqnarray}\nonumber
k^{2/n}&=&\big{(}k_+-m_{n,p}(K)+1+j \big{)}^{2/n}\eqskip
&=&\left[ m^n\left(\frac{f_{n-1}}{n}\right)+ m^{n-1}\left(\frac{f_{n-1}}{2}
+\frac{(f_{n-2}-\tilde{f}_{n-2})}{n-1} +\frac{1+j}{m^{n-1}}\right)+ \bo(m^{n-2})\right]^{2/n}\nonumber
 \eqskip \label{upper}
 &=& \left(m^n\,\frac{f_{n-1}}{n}\right)^{2/n}\left( 1 +\frac{1}{m}\left( n/2
+\frac{n(f_{n-2}-\tilde{f}_{n-2})}{(n-1)f_{n-1}} +\frac{(1+j)n}{m^{n-1}f_{n-1}}\right)+\bo\left(\frac{1}{m^2}\right)\right)^{2/n}.
 \end{eqnarray}
  Note that
 \begin{equation}\label{cut}
 (n!p)^{2/n}\left(m^n\,\frac{f_{n-1}}{n}\right)^{2/n}=m^2p^2,
 \end{equation}
 and from~\eqref{tilde2term} we obtain, 
 \begin{equation}\label{almost}
 \frac{1}{m_{n,p}(K)}-\frac{(n-1)!}{m^{n-1}p^{n-2}}=\frac{\bo(m^{n-2})}{\bo(m^{2n-2})}=\bo\left(\frac{1}{m^{n}}\right).
 \end{equation}
Using~\eqref{upper},~\eqref{cut}) and~\eqref{almost}, and the two-term binomial expansion, we have,
\[
\begin{array}{lll}
\lambda_k-C_{W,n,p}k^{2/n} & = & Q_2(m) - m^2p^2\left[ 1 +\fr{1}{m}\left(\fr{n}{2}
+\fr{n(f_{n-2}-\tilde{f}_{n-2})}{(n-1)f_{n-1}} +\fr{(1+j)n}{m^{n-1}f_{n-1}}\right)\right.\eqskip
& &\hspace*{5mm}\left.+\bo\left(\fr{1}{m^2}\right)\right]^{2/n}\eqskip
&=& Q_2(m) - m^2p^2\left[] 1 +\fr{1}{m}\left(1+\fr{2(f_{n-2}-\tilde{f}_{n-2})}{(n-1)f_{n-1}} +\fr{2(1+j)}{m^{n-1}f_{n-1}}\right)\right.
\eqskip
& &\hspace*{5mm} \left.+\fr{2}{n}\bo(\frac{1}{m^2})+\so
\left(\fr{1}{m}\right)
\right]\eqskip
&=& mp(2p+2r+n-1)+(p+r)(p+r+n-1)\eqskip
& &\hspace*{5mm}-{mp^2}\left(1 +\fr{2(f_{n-2}-\tilde{f}_{n-2})}{(n-1)f_{n-1}} + \fr{2(1+j)}{m^{n-1}f_{n-1}}\right)+\so(m).
\end{array}
\]
Now,
 $f_{n-2}=\frac{p^{n-2}}{(n-2)!}\left(r+\frac{n}{2}\right)~$, $\frac{f_{n-2}}{f_{n-1}}=\frac{(n-1)(r+\frac{r}{2})}{p}~$, $\tilde{f}_{n-2}=\frac{p^{n-2}}{(n-2)!}$,
 and taking $k\to +\infty~$ we have
\begin{gather*}
\begin{array}{lll}
\lefteqn{\lim_{m\to +\infty}\frac{\lambda_k-C_{W,n,p}k^{2/n}}{k^{1/n}
}}\nonumber\eqskip
&=& {\ds\lim_{m\to +\infty}}\frac{ mp(2p+2r+n-1)+(p+r)(p+r+n-1)-{mp^2}\left(1
+\frac{2(f_{n-2}-\tilde{f}_{n-2})}{(n-1)f_{n-1}} +
\frac{2(1+j)}{m^{n-1}f_{n-1}}\right)+\so(m)}{
\left(m^n\,\frac{f_{n-1}}{n}\right)^{1/n}\left( 1 +\frac{1}{m}\left(\frac{1}{2}
+\frac{(f_{n-2}-\tilde{f}_{n-2})}{(n-1)f_{n-1}} +\frac{(1+j)}{m^{n-1}f_{n-1}}\right)+\bo\left(\frac{1}{m^2}\right)\right)
}\nonumber\eqskip
&=&\!\! {\ds\lim_{m\to +\infty}}
 \frac{mp(2p+2r+n-1)-{mp^2}
\left(1+2\frac{f_{n-2}}{(n-1)f_{n-1}} +
\frac{2(1+j)}{m^{n-1}f_{n-1}}
-2\frac{\tilde{f}_{n-2}}{(n-1)f_{n-1}}\right)
}
{m\left(\fr{f_{n-1}}{n}\right)^{1/n}}\eqskip
&=& {\ds\lim_{m\to +\infty}} \fr{mp(p-1)-\fr{2(1+j)(n-1)!}{m^{n-2}p^{n-3}}+2mp}{m\left(\fr{p^{n-1}}{n!}\right)^{1/n}}\eqskip
&=& {\ds\lim_{m\to +\infty}} \left[
(pn!)^{1/n}(p+1)-(pn!)^{1/n}\left(\fr{2(1+j)(n-1)!}{m^{n-1}p^{n-2}}\right)\right]\eqskip
&=&  {\ds\lim_{m\to +\infty}}
\sqrt{C_{W n,p}}\left( p+1-\fr{2(1+j)(n-1)!}{m^{n-1}p^{n-2}}\right)\eqskip
&=& {\ds\lim_{m\to +\infty}}
\sqrt{C_{W n,p}}\left( p+1-\fr{2(1+j)}{m_{n,p}(K)}+(1+ j)\times \bo\left(\fr{1}{m^n}\right)\right)  \eqskip
&=& {\ds\lim_{m \to +\infty}}
\sqrt{C_{W n,p}}\left( p+1-\fr{2j}{m_{n,p}(K)-1}\right)
\end{array}
\end{gather*}
where the second but last identity comes from~\eqref{almost},~\eqref{tilde2term}, and the fact that
 $j\leq m_{n,p}(K)-1$.
 Note that this limit may not be defined for different sublimits of $j/(m_{n,p}(K)-1)$, but if we take the following quotient we get the existence of the generalised two-term asymptotics,
\begin{eqnarray*}
\lim_{k\to +\infty}\frac{\lambda_k-C_{W,n,p}k^{2/n}- c(p,k)k^{1/n}}{k^{1/n}} &=&\lim_{k\to +\infty}\left(\frac{\lambda_k-C_{W,n,p}k^{2/n}}{k^{1/n}}- c(p,k)\right)~
 =~0,
\end{eqnarray*}
and the Proposition is proved for $n\geq 3$, extending the  same formula for the case $n=2$. 
\end{proof}

\subsection{Stronger P\'{o}lya inequalities $\Phi_n\geq 0$ and $\Psi_n\geq 0$.}
In order to obtain an estimate from above for  $\pin(n)$ we  introduce {the functions}
 \begin{gather}\label{PhiPsi}
 \begin{array}{lll}
\Phi_n(p,R) &:= & \Big{(}(R+p)(R+p+n-1)\Big{)}^{n/2}\!\!\!-
 \Big{(}R(R+n-1)\Big{)}^{n/2}\!\!\!- (np)(R+1)^{\overline{n-1}},
 \eqskip
\Psi_n(p,R) & := & \Big{(}(R+p)(R+p+n-1)\Big{)}^{n}\!-
 \Big{(}R(R+n-1)\Big{)}^{n}\!- \left((np)(R+1)^{\overline{n-1}}\right)^2\eqskip
 & & \hspace*{5mm}  -2(R(R+n-1))^{n/2}(np)(R+1)^{\overline{n-1}}.
 \end{array}
 \end{gather}
{ We will now consider the inequalities $\Phi_n(p,R)\geq 0$, or, equivalently $\Psi_n(p,R)\geq 0$. Although there are slightly stronger than
$M_{n,p,r}(m)\geq 0$, we will find them easier to handle.}
 \begin{definition} Given $n\geq 2$ denote by $\pnphi$ any integer larger than or equal to $2$ for which $\Phi_n(p,R)\geq 0$ for all
 non-negative integers $R$, and $p\geq \pnphi$. We denote the smallest possible value of $\pnphi$ by $\piphi(n)$.
 \end{definition} 
 
 \begin{lemma}  \label{REDUCE}
Given integers $n\geq 2$,  $r=0,1,\ldots, p-1$, and $m\geq 0$,  consider $K=(m+1)p+r$. Assume 
  $\Phi_n(p,R)\geq 0$ holds for any integer $R$ (real, respectively) such that $r\leq R\leq K-p$. Then {the eigenvalue $K(K+n-1)$} satisfies a refined P\'{o}lya inequality, 
\begin{equation}\label{strongpPolya}
\Big{(}(K(K+n-1)\Big{)}^{n/2}-
 \Big{(}r(r+n-1)\Big{)}^{n/2}\geq  (np)\left(\sum_{s=0}^m(sp+r+1)^{\overline{n-1}}\right).
 \end{equation}
 In particular, $K(K+n-1)$ satisfies P\'{o}lya inequality (\ref{polyanp}). Consequently, if $\Phi_n(p,R)\geq 0$ holds for all {non-negative integer (real, respectively) $R$},  $\lune{n}{\pi/p}$ satisfies P\'{o}lya's conjecture
 $M_{n,p,r}(m)\geq 0$ for all $r=0,1,\ldots, p-1$,
 $m\geq 0$ integer (real, respectively).
  In particular  $\piphi(n)\geq \pin(n)$.
 \end{lemma}
\begin{proof}
By assumption, for each $s=0,1, \ldots m$, $\Phi_n(p,R)\geq 0$,  where $R=sp+r$, yielding
\begin{eqnarray*}
\lefteqn{\sum_{s=0}^{m}\left\{(np)(sp+r+1)^{\overline{n-1}}\right\}}\eqskip
&\leq&
\sum_{s=0}^m\left\{\left[ \big{(}(s+1)p+r\big{)}\big{(}
(s+1)p+r+n-1\big{)}\right]^{n/2}-\left[\big{(}sp+r\big{)}\big{(}sp+r+n-1\big{)}\right]^{n/2}\right\}\eqskip
&=&\left[\big{(}(m+1)p+r\big{)}\big{(}(m+1)p+r+n-1\big{)}\right]^{n/2}-\left[r(r+n-1)\right]^{n/2}
\end{eqnarray*}
This implies 
(\ref{strongpPolya}) holds. In particular $\bar{\lambda}_K$  satisfies P\'{o}lya inequality (\ref{polyanp}).
\end{proof}

 \begin{lemma}\label{Dlemma1}
 Given $n\geq 2$,  $R\geq 0$ and  $p\geq 1$,  the derivative 
 \[\frac{\partial \Phi_n}{\partial p}(p,R) =
 n\left(p+R+\frac{n-1}{2}\right)\left((R+p)(R+p+n-1)\right)^{\frac{n}{2}-1}-n(R+1)^{\overline{n-1}}\]
is non-negative for all $R\geq 0$ if $p\geq
\hat{p}(n)=\left\lceil\frac{(n-1)}{2}z(n)\right\rceil$,
where $z(n)=\sqrt{1+\rho(n)}-1$ with $\rho(n)=\frac{2n(n+1)}{3(n-1)^2}$ is decreasing in $n$,  and  range given by
$(\sqrt{5}-1)=z(2)\geq z(n)> (\sqrt{5/3}-1)=\lim_{n\to \infty}z(n)$.  In particular if 
 $p\geq (n-1)$  then for each $R\geq 0$, $\Phi_n(p,R)$ is increasing in $p$.
 \end{lemma}
\begin{proof}
 From \cite[Lemma A.1]{fms} we have  {for all $n\geq 2$ and $R\geq 0$}
 \begin{gather}\label{A.1}
 (R+1)^{\overline{n-1}}\leq \left((R+1)(R+n-1)+\frac{(n-2)(n-3)}{6}\right)^{\frac{(n-1)}{2}}.
 \end{gather}
It is thus enough to show
\begin{gather*}
\sqrt{(R+p)(R+p+n-1)}\left((R+1)(R+n-1)+\fr{(n-2)(n-3)}{6}\right)^{\frac{(n-1)}{2}}\\
\leq \left(p+R+\frac{n-1}{2}\right)
\left((R+p)(R+p+n-1)\right)^{\frac{n-1}{2}}
\end{gather*} 
holds, in order to prove $\frac{\partial \Phi_n}{\partial p}\geq 0$.
This is true if  the following two inequalities hold for any $R\geq 0$ and $p\geq \hat{p}(n)$
\begin{eqnarray*}\begin{array}{l}
\sqrt{(R+p)(R+p+n-1)}\leq p+R+\frac{n-1}{2},\eqskip
  (R+1)(R+n-1)+\fr{(n-2)(n-3)}{6} \leq (R+p)(R+p+n-1).
  \end{array}
  \end{eqnarray*}
The first inequality is always true for $R, p\geq 0$. The second one holds only if
for all $R\geq 0$
\[(2p-1)R +p(p+n-1)\geq \left((n-1)+\frac{(n-2)(n-3)}{6}\right). \]
This is the case if we demand
$p^2+(n-1)p\geq \left((n-1)+{(n-2)(n-3)}/{6}\right)$, 
that is,
\[p\geq \frac{(n-1)}{2}\left(\sqrt{1+ \rho(n) }
-1\right)=\frac{(n-1)}{2}\times z(n).\]
Since
$(n-1)> \hat{p}(n)\geq \frac{(n-1)}{2}\sup_n z(n)= \frac{\sqrt{5}-1}{2}(n-1)$, the last statement follows.
 \end{proof}
 \begin{remark} \label{8.1} From  $z(n)$ given in Lemma~\ref{Dlemma1} we obtain a list of $\hat{p}(n)$ for $n\leq 53$
 shown in the following table.
{\footnotesize
\begin{table}[h!]
\centering
\begin{tabular}{|c|c||c|c||c|c||c|c|}
\hline
$n$ & $\hat{p}(n)$ & $n$ & $\hat{p}(n)$ & $n$ & $\hat{p}(n)$ & $n$ & $\hat{p}(n)$\\
\hline
\rule{0pt}{10pt}
$2-5$ &$1$ &$13-18$ &$3$ &$26-32$ &$5$ &$40-47$ &$7$ \\
\hline
\rule{0pt}{10pt}
$6-12$ &$2$ &$19-25$ &$4$ &$33-39$ &$6$ &$44-53$ &$8$\\
 \hline
 \end{tabular}\vspace*{3mm}
 \caption{Values of $\hat{p}(n)$ for $n=1,\dots,53$.}
 \end{table}
 }
\end{remark}

\begin{proposition}\label{start} For all $n\geq 2$ we have  $\Phi_n(n-1,R)\geq 0$ for all $R\geq 0$. In particular $\Phi_n(p,R)\geq 0$ for all $p\geq n-1$ and all $R\geq 0$.
\end{proposition}
\begin{proof}
We have
\begin{eqnarray*}
\Phi_n(n-1,R)&=&((R+n-1)(R+2n-2))^{n/2}-(R(R+n-1)))^{n/2}-n(n-1)(R+1)^{\overline{n-1}}.
\end{eqnarray*}
The proof for $R>0$ is divided into two cases, both of which are handled in the appendices. In Lemma~\ref{Dlemma2} in Appendix~\ref{Sec 8}, 
we show directly that $\Phi_n(p,R)\geq 0$ for all $R\geq 0$, $p\in \N_2$ and $n=2, \ldots, 8$. 

For the remaining cases with $R=0$, we first note that $\Phi_n(n-1,0)\geq 0$  is equivalent to $(n-1)^{(n-1)}2^{n/2}\geq n(n-1)!$.
Since $(n-1)^{n-1}\geq (n-1)!$, this holds if $2^{n/2}\geq n$, that is, $n\geq 4$. This proves the result for vanishing $R$.

For positive $R$, from~\eqref{A.1} we see that $\Phi_n(n-1,R)\geq 0$ holds
 if the stronger inequality
 \begin{equation}\label{Ineq}
 \sqrt{n-1+R}\left( (R+2n-2)^{\frac{n}{2}}-R^{\frac{n}{2}}\right)
 \geq n(n-1)\left(1+R+\frac{(n-3)(n-2)}{6(R+n-1)}	\right)^{\frac{(n-1)}{2}}.
 \end{equation}
 holds. This is proven in Lemma~\ref{lmx1}, and the proposition follows.
 \end{proof}
\subsection{Proof of Theorem~\ref{thmxquantitative}}

 Theorem~\ref{thmxquantitative} follows immediately from the next two theorems.
\begin{thm} \label{thmC}
\begin{enumerate}[{\rm 1)}]
\item If  $n=2,3, \ldots, 8$, for any $p\geq 2$ the Dirichlet eigenvalues of $\luned{n}{p}$
satisfy  $\Phi_n(p,R)\geq 0$ for all $R\geq 0$. In particular  
P\'{o}lya's  conjecture  is satisfied and   $\piphi(n)=\pin(n)=2$.\\[1mm]
\item  If  $n=9$ {and $p=2$} then all Dirichlet  eigenvalues satisfy P\'{o}lya inequality except for the first. On the other hand, for any  $p\geq 3$,
$\Phi_9(p,R)\geq 0$ for all $R\geq 0$. In particular   
$\luned{9}{p}$ satisfies P\'{o}lya conjecture,
and $\piphi(9)=\pin(9)=3$. 
\end{enumerate}
\end{thm}
\begin{proof} The proofs of non-negativity of $\Phi_n(p,R)$ in both cases  (1) and (2) are given in Appendix~\ref{Sec 8}. It remains to prove the first statement of (2) case $n=9$ with $p=2$ that requires to compute $M_{9,2,r}(m)$, $r=0,1$.
We use the notations in (\ref{M}) and Lemma~\ref{3.4}.
\begin{gather*}
\begin{array}{ll}
M_{9,2,0}(m)=\!\!&-14727577600 - 153205309440 m - 109688868864 m^2 +  2982747746304 m^3\\
 & + 14549522402304 m^4 + 34877052684288 m^5 + 52813694047232 m^6\\
 & + 55190653820928 m^7 + 41552485536768 m^8 + 
 23061618984960 m^9\\
 & + 9541581634560 m^{10} + 
 2950812942336 m^{11}
 + 678376890368 m^{12}
 + 114172305408 m^{13}\\
 & + 13655801856 m^{14}
  + 1098842112 m^{15}
 + 53329920 m^{16} + 1179648 m^{17},\\[1mm]
M_{9,2,1}(m)=\!\!&3746550616353 + 39785540336892 m + 182331210037860 m^2 + 486289372922496 m^3\\ 
&+ 854300999930688 m^4 + 1056833334726912 m^5 + 
 958330908488960 m^6\\
 & + 653186000517120 m^7
  + 339811011691008 m^8 + 136045798322176 m^9 + 42010244133888 m^{10}\\
 & + 9970025840640 m^{11} + 
 1799957405696 m^{12} + 242516090880 m^{13}\\
 & + 23597285376 m^{14} + 
 1565589504 m^{15} + 63356928 m^{16} + 1179648 m^{17}.\end{array}
\end{gather*}
It follows that $M_{9,2,r}(m)>0$  except  for $m=r=0$ . So,  P\'{o}lya's inequality holds for all eigenvalues except for the first one.
\end{proof}
The next result considers larger values of $n$ and provides an estimate for the values of $p$ for which the lunes $\luned{n}{p}$
satisfy P\'{o}lya's conjecture. We also determine the corresponding optimal values of $p$ for $n$ up to $50$, showing that
$\piphi(n)=\pin(n)$ in this range, indicating that $\Phi_n(p,R)\geq 0$ is a close approximation of $M_{n,p,r}(m)\geq 0$.
\begin{thm}\label{thmD}  
For $n\geq 10$  the  Dirichlet eigenvalues of  $\luned{n}{p}$ satisfy $\Phi_n(p,R)\geq 0$ for all $R\geq 0$, and  all $p\geq \pnphi=n-1 $ (not optimal).
These lunes satisfy P\'{o}lya's conjecture.  For $n\leq 50$ the optimal values satisfy $\piphi(n)=\pin(n)$.
{\footnotesize 
\begin{table}[h]
\centering
\begin{tabular}{|c|c||c|c||c|c||c|c|}
\hline
$n$ & $\pin(n)$ & $n$ & $\pin(n)$ & $n$ & $\pin(n)$ & $n$ & $\pin(n)$\\
\hline
$2$ &$1$ &$9-16$ &$3$ &$24-31$ &$5$ &$40-47$ &$7$ \\
\hline
$3-8$ &$2$ &$17-23$ &$4$ &$32-39$ &$6$ &$48-50$ &$8$\\
 \hline
 \end{tabular}
 \vspace*{3mm}
 \caption{Values of $\pin(n)$ for $n=1\dots 50$ -- see also Tables~\ref{evenp} and \ref{oddp}.}
 \end{table}}  
\end{thm}
\begin{proof} From Lemma~\ref{Dlemma1},  $\hat{p}(n)\geq n-1$, that is $\Phi_n(p,R)$ increases on $p$ for $p\geq n-1$. Applying Proposition~\ref{start} for any $R\geq 0$, if $p\geq n-1$, then  $\Phi_n(p,R)\geq \Phi_n(n-1,R)\geq 0$. The computations of the optimal $\piphi(n)$ and  $\pin(n)$ for $n \leq 50$ are given in Appendix~\ref{Sec 8},
where we prove that for $n\geq 50$, $\pin(n)=\piphi(n)$.
\end{proof}

\appendix

\section{Reflection Groups and Hilbert-Poincar\'{e} series }\label{ApA}

{A finite Coxeter group $\mathcal{W}$ is a finite group of reflections (\cite{H})},  generated by a set $S$ of $l$ reflections on $\mathbb{R}^{l}$ that defines a reduced root system
$\mathcal{R}$ composed by the unit roots $\pm \alpha$ 
 of any reflection $s_{\alpha}\in\mathcal{W}_s$ (see (\ref{R2redution})) where $\mathcal{W}_s$ is the subset  of all reflections of $\mathcal{W}$. This system   satisfies the following three conditions  (\cite[p.6-7]{H}). (a) $\mathrm{Span}~ \mathcal{R}=\mathbb{R}^l$; (b) For all $\alpha \in \mathcal{R}$, $~\mathcal{R}\cap L_{\alpha}=\{-\alpha,\alpha\}$ where $L_{\alpha}=\mathbb{R}\alpha$ is a line of roots; (c) For all $\alpha\in \mathcal{R}$, $s_{\alpha}(\mathcal{R})\subset \mathcal{R}$. 
The Weyl chambers are  given by  the connected components of the complement of the set of reflecting walls,  $\mathfrak{H}=\{H_{\alpha}=[\alpha]^{\bot}\}$, where $\alpha\in \mathcal{R}$ is a chosen unit of each element $s_{\alpha}$. This is extended to $\mathbb{R}^{n+1}=\R^{n-1}\times \R^l$ (\cite{bb}), defining a  root system for the extended reflection group $\tilde{\mathcal{W}}=\{\id_{\mathbb{R}^{n+1-l}}\}\times \mathcal{W}$, not satisfying (a) (\cite[p.\ 7]{H}). A dihedral  $n$-lune $\lune{n}{\pi/p}$ is the intersection of  $\mathbb{S}^n$ with a sector of $\R^{n+1}$ defined by the walls of a chamber modelled by  $\tilde{\mathcal{W}}$ which extends the chamber $\mathfrak{C}$ modelled by the  dihedral group $\mathcal{W}=\D_p$  acting on $\R^2$.
 
 \subsection{The dihedral group $\D_p$.}\label{Dp} Given a unit vector $\beta\in \R^2$, let
$\sigma_{\beta}(u)=u-2\langle u,\beta\rangle \beta$, and set 
\begin{gather*}
\begin{array}{l}
\beta_0=e^{i\pi/2}=(0,1)=i,\quad \beta_1=e^{i\theta}\beta_0, \quad \theta=\frac{\pi}{p},\\
s_0=\sigma_{\beta_0}=S^0, \quad s_1=\sigma_{\beta_1}=e^{2i\theta}s_0=S^{2\theta},
\quad
r=s_1s_0=e^{2i\theta}=R^{2\theta}, \quad r^p=Id.
\end{array}
\end{gather*}
For all integer $k$ define
\[\beta_{k}=e^{i\theta}\beta_{k-1}= e^{ik\theta}\beta_0, \quad s_k=r^ks_0=\sigma_{\beta_k}.\]
Then $\D_p$ is generated by two reflections 
$S=\{s_0,s_1\}$ with respect to  the walls $[\beta_0]^{\bot}$ and $[\beta_1]^{\bot}$
making the angle $\pi/p$  of a chamber $\mathfrak{C}$. Then 
$\mathcal{W}=\mathcal{W}_s\cup \mathcal{W}_r$
with  
$\mathcal{W}_s=\{s_0,s_1, \ldots, s_{p-1}\}$,  $\mathcal{W}_r= 
\{\mathbbm{1}=r^0, r,r^2,\ldots, r^{p-1}\}$,
and 
$\mathfrak{H}:= \{H_{\beta_k}=[\beta_k]^{\bot}=\mathbb{R}e^{i(k\theta+\pi)}, \quad k=0,\ldots, p-1\}$.
 Hence 
$\#(\mathfrak{H})=p$,  
the Coxeter transformation 
$c=s_1s_0=r$ has order $h=p$, and the  exponents are given by $m_1=1$ and $m_2=h-1=p-1$ (\cite{BOU}, p.\ 119
Th\'{e}or\`{e}me 1).
For all $k,j\in \mathbb{N}$ we have,
{ \begin{equation}\label{relations}
 \begin{array}{l}
  s_jr^k=r^{p-k}s_j, \quad
  s_ks_j=r^{k-j}, \quad 
s_j(\beta_k)=\beta_{p-k+2j}\\
-\beta_k=\beta_{p+k}, \quad \beta_{-k}=\beta_{2p-k}=
\beta_{p+(p-k)}=-\beta_{p-k}, \quad \beta_{2p}=\beta_0
\end{array}
\end{equation}
}
Thus, we obtain a cyclic $2p$ sequence of distinct vectors that closes at $\beta_{p+p}=\beta_0$, 
\begin{equation}\label{anticlockwiseorder} 
\beta_0, ~\beta_1, ~\ldots, ~\beta_{p-1}, ~\beta_p=-\beta_0, ~\beta_{p+1}=-\beta_1,~\beta_{p+2}= -\beta_2, \ldots, ~\beta_{p+p-1}=-\beta_{p-1}
\end{equation}
 Moreover, they satisfy 
 $\angle(\beta_k,\beta_{k+1})=\theta = \angle([\beta_k]^{\bot},[\beta_{k+1}]^{\bot})$, and $\angle(\beta_0,\beta_{k})=k\theta = \angle([\beta_0]^{\bot},[\beta_{k}]^{\bot})$.
We take the reductive root system
$\mathcal{R}=\mathrm{I}_2(p):=\{\pm\beta_0, \dots, \pm\beta_{p-1}\}$.
This Coxeter system is irreducible if $p\geq 3$.
Note that $\mathcal{R}$ is a crystallographic system only for $p=2,3,4$ and $6$. In these cases $\mathbb{D}_p$ is a Weyl group.  

\subsection{The Hilbert-Poincar\'{e} series}\label{Sec. 6.2}
Given a graded algebra over $\mathbb{R}$ of homogeneous polynomials,
$\mathcal{M}=\oplus_{k\geq 0}M_k$,
where $k$ is the degree, we consider the Hilbert-Poincar\'{e} series of $\mathcal{M}$, 
$\mathcal{HP}_{\mathcal{M}}(T)=\sum_{k\geq 0} dim_{\mathbb{R}}(M_k)T^k$. 
In case of the graded algebra of homogeneous polynomials on $\mathbb{R}^l$, $P(l)=\oplus_{k\geq 0}P_k(l)$ the series is given by~\cite{amd} 
\begin{equation}\label{PS1}
\mathcal{HP}_{P(l)}(T)=\sum_{k\geq 0}\binom{l-1+k}{l-1}T^k=\frac{1}{(1-T)^l}.
\end{equation}
Let $P^{i}_k(l)$ and $P^{a}_k(l)$, $k\geq 0$, be the graded sub-algebras of homogeneous polynomials that are $\mathcal{W}-$invariant
and anti-invariant, respectively. The corresponding graded sub-algebras of homogeneous harmonic polynomials,
$H_k(l)$, $H^{a}_k(a)$ and $H^{a}_k(l)$, satisfy $\dim(H_k(l))=\dim(P_k(l))-\dim(P_{k-2}(l))$, with the same identities also holding in the invariant and anti-invariant cases. These identities also holds  replacing $l$ by $n+1$
and $\mathcal{W}$ by $\tilde{\mathcal{W}}$~\cite[B.12]{bb}. 
In particular
\begin{equation}\label{PSH}
 \mathcal{HP}_{H(l)}(T)=(1-T^2)\mathcal{HP}_{ P(l)}(T),\quad
  \mathcal{HP}_{H(n+1)}(T)=(1-T^2)\mathcal{HP}_{P(n+1)}(T),
\end{equation}
(\cite[Section C.2]{bb}) 
with similar identities for the invariant and anti-invariant cases.
{Note that $\tilde{\mathcal{W}}$ is reducible and acts as the identity map on the graded sub-algebra $P(n+1-l)$ of $P(n+1)$ given by
the homogeneous polynomials on $\mathbb{R}^{n+1-l}$.
Therefore $P^{i}(n+1-l)=P^{a}(n+1-l)=P(n+1-l)$ (the same holding in the harmonic case), and
 (\ref{PS1}), (\ref{PSH}) hold replacing $l$ by $n+1-l$. In particular $P^{i}(n+1)=P(n-1)\otimes P^{i}(2)$ and similarly for anti-invariant
 polynomials. Since the Poincar\'{e} series is given by the product of the series of the components (\cite[C.5. and  Prop. 6, first equality
 p.243]{bb},~\cite[V.5.1, p.103, (2)]{BOU}), we have}
\begin{equation}\label{NOTIRREDU}
\begin{array}{l}
\mathcal{HP}_{P^{i}(n+1)}(T)=
\mathcal{HP}_{P(n+1-l)}(T)\mathcal{HP}_{P^{i}(l)}(T)\eqskip
\mathcal{HP}_{H^{i}(n+1)}(T)=(1-T^2)
\mathcal{HP}_{P(n+1-l)}(T)\mathcal{HP}_{P^{i}(l)}(T),
\end{array}
\end{equation}
with similar identities holding for the anti-invariant case.
The dimension $\dim(P^i_k(l))$ can be derived from the exponents  $m_j$ of $\mathcal{W}$, given by the roots of the characteristic
polynomial of the Coxeter transformation $c=\sigma_1\cdot \ldots\cdot \sigma_l$ of order $h$
\begin{equation}\label{EXPO}
 P(t):=\det[ t\cdot \mathbbm{1}-c]={\ds \prod_{j=1}^l}\left(t- e^{2i\pi m_j/h}\right), \quad
0\leq m_1\leq m_2\leq\ldots \leq m_l<h.
\end{equation}
It is not  required  $\mathcal{W}$ to be irreducible, but requires to be essential in the sense of~\cite[V.3.7 p.83]{BOU},  
which is the case we treat here.
The Coxeter transformation of  $\mathcal{W}$ is given by the product of the ones of the  components  with order $h$ given by the smallest common multiple of the orders of the components,  and $P(t)$ is the product of the respective characteristic polynomials. 
For $\mathcal{W}$ irreducible we have
$ m_1=1$, $m_l=h-1$, and $\#(\mathfrak{H})=\#\mathcal{W}_s=\frac{1}{2}lh$
~\cite[Theorem 1, p.119]{BOU}.
Furthermore, the following equality holds with  $k_j=m_j+1$.
\begin{equation}\label{NOTYET}
\begin{array}{l}
\mathcal{HP}_{P^i(l)}(T) = \Big{(}{\ds \prod_{j=1}^l}(1-T^{k_j})\Big{)}^{-1}={\ds \prod_{j=1}^l}\frac{1}{(1-T^{m_j+1})}=
{\ds \prod_{j=1}^l}\left(\dsum_{k\geq 0} T^{(m_j+1)k}\right)
\end{array}
\end{equation}
(\cite[p.103, p.107, Prop. 3 p. 121]{BOU} and \cite[C.2, p.\ 241 and 244 Preuve, Remarque]{bb}). 
Note that $\tilde{\mathcal{W}}$ and $\mathcal{W}$ have the same cardinality and in the reducible case of $\mathcal{W}$ the exponents are the union of the exponents  of the components.

 To compute $\dim(P^a_k(n+1))$ we use the same argument given in  \cite{bb}(C.3.).
From now on we are assuming $l=2$, $\mathcal{W}=\mathbb{D}_p$, and consider the minimal homogeneous  polynomial  function 
$\mathcal{E}:\mathbb{R}^2\to \mathbb{R}$  that vanishes on each hyperplane $H_{\beta_k}\in \mathfrak{H}$, necessarily of degree $p$, 
\begin{equation}\label{Epoly}
\mathcal{E}(\xi)=g(\xi,\beta_0)g(\xi, \beta_1)\ldots g(\xi,\beta_{p-1}), 
\end{equation}
where $g$ is the Euclidean metric on $\R^2$. Its  extension  to $\mathbb{R}^{n+1}=\mathbb{R}^{n-1}\times\mathbb{R}^2$,  $\hat{\mathcal{E}}(w,\xi)=\mathcal{E}(\xi)$, vanishes on the
hyperplanes  $\mathbb{R}^{n-1}\times H_{\beta_k}$. The next lemma is a particular case  of
a more general framework given in  \cite{BOU}, Prop.\ 5, p.\ 113, Prop.\ 6, p.\ 116.,  following \cite{bb} Prop.\ C.3, p.\ 242, adapted now to $\mathbb{D}_p$,  choosing the root system $\mathcal{R}$  described in Section~\ref{Dp}, that does not need to be crystallographic.
\begin{lemma}\label{BOU BB}
 Let $\mathcal{W}=\D_p$. Then we have: 
 \begin{enumerate}[{\rm 1)}]
\item $\mathcal{E}\in P_p^a(2)$, and  $\hat{\mathcal{E}}\in P_p^a(n+1)$.\\
\item $P^a(2)=P^i(2)\cdot\mathcal{E}$
and $P^a(n+1)=P^i(n+1)\cdot\hat{\mathcal{E}}=P(n-1)\otimes P^i(2)\cdot \hat{\mathcal{E}}$.
\end{enumerate}
\end{lemma}
\begin{proof} 
 The vectors $\beta_k$, $k\in \mathbb{Z}$,  satisfy the relations given in \eqref{relations} and \eqref{anticlockwiseorder}.
Moreover,  for any $j,k$,   
$s_j(\beta_k)=\beta_{p-k+2j}$. 
Since $\mathcal{W}$ is generated by $S=\{s_0,s_1\}$, to prove (1) it is sufficient to show  that
$\mathcal{E}(s_0(\xi))=- \mathcal{E}(\xi)$ and  
$\mathcal{E}(s_1(\xi))=- \mathcal{E}(\xi)$.
Reflections are isometric involutions, so
\begin{eqnarray*}
\mathcal{E}(s_0(\xi))&=& g(s_0(\xi),\beta_0)\ldots g(s_0(\xi),\beta_{p-1})\\
&=&g(\xi,s_0(\beta_0))\ldots g(\xi,s_0(\beta_{p-1}))\\
&=&g(\xi,\beta_p)g(\xi,\beta_{p-1})\ldots g(\xi,\beta_1)\\
&=&g(\xi, -\beta_0)g(\xi,\beta_{p-1})\ldots g(\xi,\beta_1)=
-\mathcal{E}(\xi),\\[1mm]
\mathcal{E}(s_1(\xi))
&=&g(\xi,s_1(\beta_0))\ldots g(\xi,s_1(\beta_{p-1}))\\
&=&g(\xi,\beta_{p+2})g(\xi,\beta_{p+1})\ldots 
g(\xi,\beta_3)\\
&=&g(\xi, -\beta_2)g(\xi,-\beta_1)g(\xi,-\beta_0)g(\xi,\beta_{p-1})\ldots g(\xi,\beta_3)=
-\mathcal{E}(\xi).
\end{eqnarray*}
Now we prove (2). 
 If $f\in P^a_k(2)$, then $f(\sigma_{\beta_j}(\xi))=-f(\xi)$, $j=0, \ldots, p-1$. On the other hand
 $f(\sigma_{\beta_j}(\xi))-f(\xi)=0$ for $\xi\in [\beta_j]^{\bot}$. Then $g(\xi,\beta_j)$ divides $f\circ \sigma_{\beta_j} -f$, that is
 $f(\sigma_{\beta_j}(\xi))-f(\xi)=g(\xi,\beta_j)f_j(\xi)$ 
 for some $f_j\in P_{k-1}(2)$. Hence,
 $-2f(\xi)=g(\xi,\beta_j)f_j(\xi)$, that is $g(\xi,\beta_j)f_j(\xi)=g(\xi,\beta_r)f_r(\xi)$,  for all $j, r$.
   Thus, if  $j\neq r$ and $\xi\in [\beta_j]^{\bot}\backslash [\beta_r]^{\bot}$, then
 $f_r(\xi)=0$, that is  $g(\xi,\beta_j)$ divides $f_r(\xi)$.
 Consequently, 
 $\mathcal{E}(\xi)$ divides $f(\xi)$, that is, 
  $f(\xi)=\mathcal{E}(\xi)Q(\xi)$, {for a unique homogeneous polynomial $Q(\xi)$ of degree $p-k$. It is clear that $Q$ must be invariant. }
Hence $P^a_k(2)\subset \mathcal{E}P^i_{k-p}(2)$, for all $k\geq p$, and obviously $\mathcal{E}P^i_{k-p}(2)\subset P^a_k(2)$, and so we have an equality. 
The same conclusions hold for the natural  extensions on $(w,\xi)\in \mathbb{R}^{n+1}=\mathbb{R}^{n-1}\times \mathbb{R}^2$. 
\end{proof}
\begin{remark} {Consider the holomorphic function $(\xi)^p=(x_n+ix_{n+1})^{p}=u_p(\xi)+iv_p(\xi)=\rho^p(\cos(p\theta)+i\sin(p\theta))$,
where $\xi=\rho e^{i\theta}$. It follows that  $\mathcal{E}$ given in Lemma~\ref{BOU BB}, $u_p$, $v_p$, and $x_iu_pv_p$ with $1\leq i\leq n-1$,
are all harmonic homogeneous polynomials. The first three of degree $p$ and the remaining of degree $2p+1$, all them except $u_p=x_n^2-x_{n+1}^2$
vanishing at $\beta_k$, $k=0, \ldots, p-1$, hence they are
eigenfunctions of $\lune{n}{\pi/p}$, with $\mathcal{E}=c v_p$, $c\neq 0$ constant, the principal eigenfunction.}
\end{remark} 

Now {we  derive} the Hilbert-Poincar\'{e} series.
\begin{proposition}
\label{Inv-Anti}
Let $\mathcal{W}=\D_p$ where $p\geq 3$, and
$\tilde{\mathcal{W}}$ its extension to $\mathbb{R}^{n+1}$.
Then for all $k\geq 0$
$\dim(P^a_{p+k}(2))=\dim(P^i_k(2))$ and $\dim(P^a_{p+k}(n+1))=\dim(P^i_k(n+1))$.
Therefore  we have  
$\mathcal{HP}_{P^a(2)}(T)=T^p\mathcal{HP}_{P^i(2)}(T)$, and $  \mathcal{HP}_{P^a(n+1)}(T)=T^p\mathcal{HP}_{P^i(n+1)}(T)$. 
 Consequently,
\begin{gather}
\begin{array}{ll}
\mathcal{HP}_{P^i(2)}(T) =\fr{1}{(1-T^2)(1-T^p)},&
\mathcal{HP}_{P^i(n+1)}(T) =\fr{1}{(1-T)^{{n-1}}(1-T^2)(1-T^p)},\eqskip
\mathcal{HP}_{P^a(2)}(T) =\fr{T^p}{(1-T^2)(1-T^p)},&
\mathcal{HP}_{P^a(n+1)}(T) =\fr{T^p}{(1-T)^{{n-1}}(1-T^2)(1-T^p)},\eqskip
 \mathcal{HP}_{H^i(2)}(T) = \fr{1}{(1-T^p)},&
\mathcal{HP}_{H^a(2)}(T)= \fr{T^p}{(1-T^p)},
\end{array}\nonumber \eqskip
\mathcal{HP}_{H^i(n+1)}(T) = \frac{1}{(1-T)^{n-1}(1-T^p)}, \quad
 \mathcal{HP}_{H^a(n+1)}(T)
=\fr{T^p}{(1-T)^{n-1}(1-T^p)}. \label{Fi}
\end{gather}
\end{proposition}
\begin{proof}
From  Appendix \ref{Dp} and \ref{Sec. 6.2}, the exponents are $m_1=1$, $m_2=p-1$, and by
(\ref{NOTYET}) with $k_j=m_j+1$  we get the expressions of the Poincar\'{e} series of $P^i(2)$ and from (\ref{NOTIRREDU}) the one of $P^i(n+1)$, and by Lemma~\ref{BOU BB} we have 
\begin{eqnarray*}
\begin{array}{l}
\mathcal{HP}_{P^a(2)}(T)=\dsum_{k\geq 0}\dim(P^a_k(2))T^k=\sum_{k\geq 0}\dim(P^i_{k-p}(2))T^k=T^p\mathcal{HP}_{P^i(2)}(T).
\end{array}
\end{eqnarray*}
The series of $H^i(l)$ follows from (\ref{PSH}) what implies  the one of  $H^i(n+1)$  by using (\ref{PS1}) replacing $l$ by $n+1-l$,  and the formulas of the series on decomposition of {reducible} groups into the product of irreducible ones given in (\ref{NOTIRREDU}). Similarly
$\mathcal{HP}_{H^a(n+1)}(T)$ is obtained from
\[
\begin{array}{lllll}
\mathcal{HP}_{H^a(2)}(T)&=&(1-T^2)\mathcal{HP}_{P^a(2)}(T)&=&(1-T^2)T^p\, \mathcal{HP}_{P^i(2)}(T).
\end{array}
\]
\end{proof}

The case of $n$-hemispheres ($l=1$) is already known  (see e.g \cite{bb}), with
\[ \mathcal{HP}_{H^i(n+1)}(T)=\sum_{k\geq 0}\binom{n-1+k}{n-1}T^{k} = \frac{1}{(1-T)^n}\] and
\[
 \mathcal{HP}_{H^a(n+1)}(T)=
  \sum_{k\geq 0}\binom{n-1+k}{n-1}T^{k+1}=\frac{T}{(1-T)^n}.
\]
These series extend the case of $\lune{n}{\pi/p}$ to $p=1$.

{\subsection{ $\luned{n}{2}$: The case $l=2$ with {reducible} $A_1\times A_1=\mathbb{D}_2$    
} \label{Sec 2.4}
The two Coxeter systems on $\mathbb{R}^2$, 
$A_1\times A_1$ and $\mathbb{D}_2$,  are the same  since $\mathbb{D}_2$ 
is generated by $S=\{s_0, s_1=e^{i\pi}s_0=-s_0\}$, where  $s_0$ and $s_1$ commute. Then the split $S=\{s_0\} \cup\{-s_0\}$ defines $\mathcal{W}$ homeomorphic  to $\mathcal{W}_1\times \mathcal{W}_1$ where $\mathcal{W}_1$ is the previous case for $l=1$. 
The homogeneous polynomial  of degree two $\mathcal{E}(u,v)=\mathcal{E}_1(u)\mathcal{E}_2(v)=g(u,1)g(v,1)=uv$  is anti-invariant  with respect to $\mathcal{W}$. Applying
the same argument from proof of Lemma~\ref{BOU BB}, then 
a similar conclusion to Prop.\ ~\ref{Inv-Anti} follows, 
$P^i(2)=P^i(\mathbb{R}_u)\otimes P^i(
\mathbb{R}_v)$, and $P^a(2)= P^i(2)\otimes \mathcal{E}$.
Consequently
$dim (P^a_{2+k}(2))=dim (P^i_{k}(2))$ and
$dim (P^a_{2+k}(n+1))=dim (P^i_{k}(n+1))$, and 
we obtain} 
\begin{gather*}\begin{array}{l}
\mathcal{HP}_{P^i(2)}(T)=(\mathcal{HP}_{P^i(1)}(T))^2=\frac{1}{(1-T^2)^2}, \quad 
\mathcal{HP}_{P^a(2)}(T)=
\frac{T^2}{(1-T^2)^2},\quad 
 \mathcal{HP}_{H^a(2)}(T)=\frac{T^2}{(1-T^2)}.,
 \\
\mathcal{HP}_{H^i(n+1)}(T)=\frac{1}{(1-T)^{n-1}(1-T^2)}, \quad 
\mathcal{HP}_{H^a(n+1)}(T)=\frac{1}{(1-T)^{n-1}}\frac{T^2}{(1-T^2)}.
\end{array}
\end{gather*}
These series extends to $p=2$ the ones  given in  Proposition~{\ref{Inv-Anti}}  for $p\geq 3$.

\section{Optimal $\hat{p}(n)$, $\piphi(n)$ and $\pin(n)$,  for $2\leq n\leq 50$\label{Sec 8}}
Recall that $M_{n,p,r}(m)$ given in (\ref{M})  and for any $p\geq 1$, $n\geq 2$, we have
  \[\begin{array}{l}
  \Phi_n(p,0)=(p(p+n-1))^{n/2}-(n!p);\\
  M_{n,p,0}(0)=\Psi_n(p,0)=(p(p+n-1))^{n}-(n!p)^2
  \end{array}\]
 Moreover, the following statements are equivalent. (a) $\Phi_n(p,0)<0$; (b) $\Psi_n(p,0)<0$; (c) The first Dirichlet eigenvalue $\lambda_1$ of $\lune{n}{\pi/p}$ does not satisfy P\'{o}lya inequality (cf.\ Proposition~\ref{First})
  
For any $n\geq 2$ and $p\geq 1$ (even or odd) consider the following expressions
\begin{gather*}
\begin{array}{ll}
H_n(p,R)&=\Big{(}(R+p)(R+p+n-1)\Big{)}^{n}-
 \Big{(}R(R+n-1)\Big{)}^{n}
 - \left((np)(R+1)^{\overline{n-1}}\right)^2\eqskip
J_n(p,R) &= 2(R(R+n-1))^{n/2}(np)(R+1)^{\overline{n-1}}\eqskip
\Theta_n(p,R)&= H_n(p,R)^2-J_n(p,R)^2
   \end{array}
   \end{gather*}
 Recall that $\Phi_n(p,R)\geq 0$ is equivalent to $\Psi_n(p,R)\geq 0$ (\ref{PhiPsi}). Moreover, $\Psi_n(p,R)=H_n(p,R)-J_n(p,R)$, and so
 if $H_n(p,R)\geq 0$, then  $\Psi_n(p,R)\geq 0$
 is equivalent to $\Theta_n(p,R)\geq 0$.
\begin{lemma}\label{Dlemma2}
For all $p\geq 2$, $2\leq n\leq 8$,  we have $\Phi_n(p,R)\geq 0$ for all $R\geq 0$. This no longer holds with $n=9$ for
$p=2$, since the first eigenvalue does not satisfy P\'{o}lya's inequality. For $2\leq n\leq 50$ we derive the optimal
values of $\piphi(n)$ and $\pin(n)$ showing they coincide -- see also Tables~\ref{evenp} and~\ref{oddp}.
\end{lemma}

{\footnotesize 
\begin{table}[h!]
\centering
\begin{tabular}{|c|c|c|c|c|}
\hline
\scriptsize{$n$ even} &\scriptsize{$\hat{p}(n)$}&
\scriptsize{$(p,R):\Phi(p,R)<0$} &\scriptsize{$\min p:\Phi(p,R)\geq 0 ~\forall R$} & \scriptsize{$\piphi(n)$}\\ 
\hline
 $2$ & $1$  &  {}   & $1$ & $1$\\
\hline
 $4$ & $1$ & $(1,0)$ & $2$  & $2$\\
 \hline
 $6,8$ & $2$ & $(1,0)$ & $2$  & $2$\\
  \hline
  $10$ & $2$ & $(2,0)$ & $3$  & $3$\\
   \hline
   $12,14,16$ &3& $(2,0)$ & $3$ & $3$\\ 
    \hline
   $18$ & $3$ & $(3,0)$ & $4$  & $4$\\  
    \hline
  $20,22$ &4& $(3,0)$ & $4$ & $4$\\
   \hline
 $24$ & $4$ & $(3,0),(4,1)$ & $5$ & $5$\\
  \hline
 $26,28,30$ & $5$ & $(4,0)$ & $5$ & $5$\\
  \hline
 $32$ & $5$ & $(5,0)$ & $6$  & $6$\\
  \hline
  $34,36,38$ & $6$ &  $(5,0)$ & $6$  & $6$\\
   \hline
   $40,42,44,46$ &$7$&  $(6,0)$ & $7$  & $7$\\
    \hline
    $48,50$ & $8$ & $(7,0)$ & $8$ & $8$\\
 \hline
 \end{tabular}
 \vspace*{3mm}
 \caption{Values of $\hat{p}(n)$ and $\piphi(n)$ for $n$ even}\label{evenp}
 \end{table}}
{\footnotesize 
\begin{table}[h!]
\centering
\begin{tabular}{|c|c|c|c|c|c|}
\hline
\scriptsize{$n$ odd}&\scriptsize{$\hat{p}(n)$}&\scriptsize{$(p,R):H(p,R)<0$} &\scriptsize{$\min   p:H(p,R)\geq 0 ~\forall R$}& \scriptsize{$\min p:\Theta(p,R)\geq 0 ~\forall R$} & 
\scriptsize{$\piphi(n)$}\\ 
\hline
 $3$ & $2$  & $(1,0)$   & $2$ & $2$ & $2$\\
\hline
 $5,7$ & $2$  & $(1,0)$   & $2$ & $2$ & $2$\\
 \hline
  $9,11$ & $2$  & $(2,0)$   & $3$ & $3$ & $3$\\
 \hline
   $13,15$ & $3$  & $(2,0)$   & $3$ & $3$ & $3$\\
 \hline
  $17$ & $3$  & $(3,0)$   & $4$ & $4$ & $4$\\
 \hline
  $19,21,23$ & $4$  & $(3,0)$   & $4$ & $4$ & $4$\\
 \hline
  $25$ & $4$  & $(4,0)$   & $5$ & $5$ & $5$\\
 \hline
   $27,29,31$ & $5$  & $(4,0)$   & $5$ & $5$ & $5$\\
 \hline
   $33,35,37,39$ & $6$  & $(5,0)$   & $6$ & $6$ & $6$\\
 \hline
   $41,43$ & $7$  & $(6,0)$   & $7$ & $7$ & $7$\\
 \hline
   $45$ & $7$  & $(6,0)$   & $7$ & $7$ & $7$\\
 \hline
    $47$ & $8$  & $(6,0)$   & $7$ & $7$ & $7$\\
    \hline
 $49$ & $8$  & $(7,0)$   & $8$ & $8$ & $8$\\
  \hline     
\end{tabular}
\vspace*{3mm}
 \caption{Values of $\hat{p}(n)$ and $\piphi(n)$ for $n$ odd}\label{oddp}
 \end{table}
}

\begin{proof}
We obtain for $n$ even (the same for $n$ odd) $\piphi(n)$ using $\hat{p}(n)$ (Remark~\ref{8.1}) to reduce $p_{\Phi}(n)=n-1$ by checking when $\Phi_n(p,R)\geq 0$ for $2\leq n\leq 50$. Then we use the fact that $\pin(n)\leq \piphi(n)$, or inspection of $M_{n,p,r}(m)$ to determine $\pin(n)$.

 For example, if $n=2$, $\Phi(p,R)= 0$. Hence $\hat{p}(2)=\piphi(2)=\pin(2)=1$. If $n=4$, $\hat{p}(4)=1$, $\Phi(1,R)<0$,
$\Phi(2,R)=52+52R+12 R^2$ and  so $\piphi(4)=2$. Necessarily $\pin(4)=2$. This is the generic procedure. However, the case of $n=24$ is different: $\hat{p}(24)=4$,  $\Phi(3,0)<0$, $\Phi(4,0)>0$, but $\Phi(4,1)<0$, and $\Phi(5,R)\geq 0$. Hence $\piphi(24)=5$. Since   $M_{24,4,1}(0)<0$ we  conclude that $\pin(24)=5$ as well.
Most cases in the table all coefficients of $\Phi(p,R)$ are positive, for $p\geq p^*$.
If $n=3$, $\hat{p}(n)=2$, $H(1,0)=-9$, hence $\Psi(1,0)<0$. Now $H(2,R)\geq 0$, $\Theta(2,R)\geq 0$. Hence $\piphi(3)=\pin(3)=2$. This is the usual case. If $n=47$ we have a special case. $\Psi(7,R)$ is positive for $R=0,1$ and all $R\geq 2$ but it is negative for some non-integer $R\in (1.2,1.4)$. However, $H(8,R)\geq 0$ and $\Theta(8,R)\geq 0$  for all $R\geq 0$. Moreover, $\piphi(47)=\pin(47)=7$ and not $8$ because only integer $R$  matters for P\'{o}lya's conjecture.
 \end{proof}
\section{The counting function of a dihedral lune via geodesic billiards\label{ApC}} The purpose of this appendix is to derive a counterpart of
Theorem~\ref{thmxeventually} for the counting function, following the approach developed in~\cite{sava}.
\begin{proposition}\label{2termsava} The counting function for the Dirichlet eigenvalues of 
$\lune{n}{\pi/p}$, $N(\lambda)=\#\{j:\lambda_j< \lambda\}$,  satisfies the following inequalities for each $\varepsilon>0$ when $\lambda \to +\infty$, \cite[(1.7.4), p. 50]{sava},
\begin{equation}\label{1.7.4}
\begin{array}{l}
c_0(\lambda-\varepsilon)^{n/2}+\left(\ci +Q(\lambda-\varepsilon))\lambda^{(n-1)/2} +\so(\lambda^{(n-1)/2}\right)\eqskip
\quad\quad\leq \quad N(\lambda)\quad \leq \quad c_0(\lambda+\varepsilon)^{n/2}+\left(\ci +Q(\lambda+\varepsilon))\lambda^{((n-1)/2}
+\so(\lambda^{(n-1)/2}\right),
\end{array}
\end{equation}
where
\[Q(\lambda)=
\frac{1}{{(2\pi)^n}}|\mathbb{L}^n_{\pi/p}||\mathbb{S}^{n-1}|\left(\left\{{1}-\sqrt{\lambda}+\frac{2p+n-1}{2}\right\}-\frac{1}{2}\right),\]
and
\begin{equation}\label{cosntantn}
\begin{array}{lll}
c_0=\fr{|S^*(\lune{n}{\pi/p})|}{n(2\pi)^n}, & \ci =-\fr{1}{4}\fr{\omega_{n-1}|\mathbb{S}^{n-1}|}{(2\pi)^{n-1}},&
|S^*(\lune{n}{\pi/p})|= |\lune{n}{\pi/p}||\mathbb{S}^{n-1}|=\fr{|\mathbb{S}^{n}|}{2p}|\mathbb{S}^{n-1}|.
\end{array}
\end{equation}

\end{proposition}
The above function $Q(\lambda)$ is defined in~\cite[Theorem~1.7.6 (1.7.9)]{sava} as the total shift along each periodic trajectory $\Gamma=(x^*(t, y, \eta), \xi^*(t,y, \eta))$
given by the sum of three shifts \cite[Definition  1.7.2]{sava},
namely, the phase shifts $f_{r}$, $f_{c}$, and $f_{s}$, generated respectively by the reflections, 
the passage of $\Gamma$ through caustics, and by the subprincipal symbol. Writing
  \[
  \boldsymbol{q}(y, \eta)=f(\boldsymbol{T}(y,\eta); y, \eta)= f_r(\boldsymbol{T}(y,\eta); y, \eta)+f_c(\boldsymbol{T}(y,\eta); y, \eta)+ f_s(\boldsymbol{T}(y,\eta); y, \eta),
  \]
$Q$ is then given by the  integral formula
\begin{equation}\label{integral}
Q(\lambda)=
\int_{S^*(\mathbb{L}^n_{\pi/p})}\frac{\{\pi-\boldsymbol{q}(y,\eta)-\sqrt{\lambda} \boldsymbol{T}(y,\eta)\}_{2\pi}}{2\pi} dy \dbar\tilde{\eta},
\end{equation}
where $\{\tau\}_{2\pi}=\tau+2k\pi
\in [-\pi, \pi)$ for some $k\in \Z$,  $\boldsymbol{T}(y,\eta)
=2\pi$ is the $t$-period of the trajectory and
\begin{equation}\label{Qbounds}
\frac{\{\pi-q-\sqrt{\lambda} T\}_{2\pi}}{2\pi}=\left\{{1}-\frac{q }{2\pi}-\fr{\sqrt{\lambda} T}{2\pi}\right\}-\frac{1}{2}
=
\frac{1}{\pi}\sum_{k=1}^{\infty} \frac{\sin(k(\sqrt{\lambda} T +q))}{k}\in \left[-\frac{1}{2}, \frac{1}{2}\right].
\end{equation}
Here $\{\cdot\}$ denotes the fractional part of a real number which, for a negative number, we define by the  sawtooth wave function. $Q(\lambda)$ is then extended as a left-countinuous function. 

The  constant  on the second term of (\ref{2termweyl})  (Dirichlet case) is  in \cite[Theorem 1.6.1]{sava} related to
$\ci $.
We derive the three shifts for any $n$-dimensional dihedral lune in Section~\ref{3shift}, $n\geq 2$. When $n$ equals $2$ we obtain the above inequalities in the next section,
using formulas from Section~\ref{Sec Third term}, related to the  generalised two-term asymptotic formula given in Theorem~\ref{thmxeventually}. \\[2mm]
\subsection{ $n=2$}\label{C1} We consider the Dirichlet eigenvalues of $\lune{2}{\pi/p}$ with $p\geq 1$.
Then, given $K$ and $k$ in the corresponding $K$-chain, we have  $\lambda_k=\lambda_{k_-(K)}=\lambda_{k_+(K)}$
and so for $ \epsilon$ such that $0\leq \epsilon < 2K$ and $0<\epsilon \leq 2(K+1)$ in the following first and
second sets of identities,
\begin{gather}\label{mp}
\begin{array}{l}
N(\bar{\lambda}_K-\epsilon)=N(\bar{\lambda}_K)=k_-(K)-1\eqskip
N(\bar{\lambda}_K+\epsilon)=k_+(K)=k_-(K)-1 +m_{2,p}(K),
\end{array}
\end{gather}
respectively. Since $c_0>0$, for any real constants $\tilde{c}$
and $\varepsilon$, with $\varepsilon\geq 0$ (variable independent of $\epsilon$), for  $\lambda$ large enough the following inequalities hold,
\begin{eqnarray}\nonumber
(\lambda+\varepsilon)^{1/2}( c_0(\lambda+\varepsilon)^{1/2} +\tilde{c})+ \so(\lambda^{1/2})&\geq& 
\lambda^{1/2}( c_0\lambda^{1/2} +\tilde{c})+ \so(\lambda^{1/2})\eqskip
&\geq&(\lambda-\varepsilon)^{1/2}( c_0(\lambda-\varepsilon)^{1/2} +\tilde{c})+  \so(\lambda^{1/2}).
\label{trivial2}
\end{eqnarray}

Given $\lambda>p(p+1)$ let $K_{\lambda}\in \N$ such that
$\bar{\lambda}_{K_{\lambda}-1}<\lambda \leq \bar{\lambda}_{K_{\lambda}}$.  Then $\lambda=\bar{\lambda}_{K_{\lambda}}-\epsilon$, for a unique $0\leq \epsilon <2K$. Thus,  $N(\lambda)=N(\bar{\lambda}_{K_{\lambda}})$. 
Reciprocally, if $\lambda=\bar{\lambda}_K-\epsilon$ and $0\leq \epsilon< 2K$ then $K_{\lambda}=K$. If  
 $\lambda=\bar{\lambda}_K+\epsilon$ and $0<\epsilon\leq 2(K+1)$ then $\bar{\lambda}_K<\lambda\leq \bar{\lambda}_{K+1}$, with $K_{\lambda}=K+1$, and so on.
Note that, for $\epsilon$ fixed, or $\epsilon$ bounded from above by a linear map of $K$,  $\so((\lambda_k\pm \epsilon)^{1/2})=\so(\lambda_k^{1/2})$ holds when $k\to +\infty$ with $k$ in the corresponding $K$-chain.

When $n$ is $2$ we have
\begin{equation}\label{Qlambda}
Q(\lambda)=
\frac{1}{p}\left(\left\{1-\sqrt{\lambda} +\frac{2p+1}{2}\right\}-\frac{1}{2}\right),
\end{equation}
where for each real number $x$ we denote by $\{x\}=x-\lfloor x\rfloor$ its fractional part, with values on the interval $[0,1]$.
Then $Q(\lambda)\in [-\frac{1}{2p},\frac{1}{2p}]$ (see \ref{Qbounds} below) is a  sawtooth wave function with non constant sign.
 
 At points of discontinuity $Q(\lambda)$ is extended as left-continuous. Moreover, 
   $Q(\lambda)$ increases for  $\sqrt{\lambda}\in (j +\frac{1}{2}, j+1 + \frac{1}{2}]$, and $(j +\frac{1}{2})^2$ are the points of
   discontinuity points  with  $Q((j^++\frac{1}{2})^2)=-\frac{1}{2p}$,  and  $Q(((j+1)^{-}+\frac{1}{2})^2)= Q((j+1+\frac{1}{2})^2)=\frac{1}{2p}$.
\begin{proposition} Let $n=2$, $p\in \N$, and consider the  strict counting function $N(\lambda)$ on the Dirichlet eigenvalues of the dihedral lune $\lune{2}{\pi/p}$.
\begin{enumerate}[{\rm 1)}]
\item The following equality and inequalities hold  for all $k\to +\infty$ and $0\leq \epsilon<2K$, where $\lambda_k$ is in the $K$-chain
\begin{eqnarray*}
N(\lambda_k-\epsilon)=N(\lambda_k)
&=& c_0\lambda_k+ \left( \ci -\frac{1}{2p}\right)\lambda_k^{1/2}+
\so(\lambda_k^{1/2})\nonumber \\[1mm]
&\geq &c_0(\lambda_k-\epsilon)+ \left(\ci -\frac{1}{2p}\right)(\lambda_k-\epsilon)^{1/2}+\so(\lambda_k^{1/2})\label{N1}\\
&\geq &c_0(\lambda_k-\epsilon)+ \left(\ci -\frac{1}{2p}\right)\lambda_k^{1/2}+\so(\lambda_k^{1/2}).\nonumber
\end{eqnarray*}
 If $k\to +\infty$ and $0\leq \epsilon<2(K+1)$, where $\lambda_k$ is in the $K$-chain, the following equality and inequalities hold, 
 \begin{eqnarray*}
N(\lambda_k+\epsilon)=
N(\lambda_k)+ m_{2,p}(K) &=&  c_0\lambda_k+ \left( \ci +\frac{1}{2p}\right)\lambda_k^{1/2}+\so(\lambda_k^{1/2})\\
&\leq& c_0(\lambda_k+\epsilon)+\left( \ci +\frac{1}{2p}\right)(\bar{\lambda}_K+\epsilon)^{1/2}+\so((\bar{\lambda}_K)^{1/2})\\
&\leq& c_0(\lambda_k+\epsilon)+\left( \ci +\frac{1}{2p}\right)\bar{\lambda}_K^{1/2}+
\so(\bar{\lambda}_K^{1/2}).
\end{eqnarray*}
\item The following inequality holds for all non-negative $\varepsilon$ and positive $\lambda$
 \[Q((\lambda+\varepsilon)^2)-Q(\lambda^2)\quad \geq -2\varepsilon c_0=-\frac{\varepsilon}{p}.
\]
That is, \cite[(1.7.2), p. 50]{sava} holds for all $p\geq 1$ 
\item Since the Hamiltonian billiard of the dihedral lunes satisfies the simple reflection principle, then \eqref{2termsava} holds. These inequalities can also be proved directly from {\rm 1)}.
\end{enumerate}
  \end{proposition}
%
%
\begin{remark}\label{newC1}
  Note that, however,  $Q(\lambda)$ is not continuous, and then the quasi-Weyl asymptotic formula \cite[(1.7.5), p. 50]{sava}, $N(\lambda)=c_0\lambda+(\ci +Q(\lambda))\lambda^{1/2} +\so(\lambda^{1/2})$, does not hold.
 Furthermore,  from \cite[Corollary 1.7.9, p.\ 53, and (1.7.7), p.\ 50]{sava}, with $q=-(2p+1)\pi$, for
 $\varepsilon_j \to 0^+$, $j\in \N_0$, and for $K\to +\infty$, $\varepsilon_K\to 0^+$
\begin{gather*} \begin{array}{ll}
{\rm 4)}\hspace*{3mm} Q\left((-\frac{(2p+1)}{2}+j +\frac{1}{2} +\varepsilon_j)^2\right)- Q\left((-\frac{(2p+1)}{2}+j+ \frac{1}{2})^2\right) &=~~\frac{(1-\varepsilon_j)}{p}\eqskip
{\rm 5)}\hspace*{3mm} Q\left((-\frac{(2p+1)}{2}+j +\frac{1}{2}+\varepsilon_j)^2\right)- Q\left((-\frac{(2p+1)}{2}+j +\frac{1}{2} -\varepsilon_j)^2\right) &=~~\frac{(1-2\varepsilon_j)}{p}\eqskip
{\rm 6)}\hspace*{3mm}
N\left( K(K+1)+\varepsilon_K\right)-
 N\left( K(K+1)-\varepsilon_K\right) &\geq~~ \frac{1}{2p}
 \left(K(K+1)\right)^{1/2}
+\so\left( (K(K+1))^{1/2} \right)
\end{array}
\end{gather*}
In particular, conditions 1. to 6. in~\cite[p.\ 53]{sava} are satisfied and the spectrum contains clusters. See the last paragraph of the proof below for the proof
of the above statements.
\end{remark}

\begin{proof} 1) We show the first identity, for $k_-(K)-1=N(\lambda_k-\epsilon)=N(\lambda_k)$ (\ref{mp}, Sec.\ \ref{C6.1}), that is
\begin{eqnarray*}
\frac{m(m+1)}{2}p +(m+1)r
&=& \frac{1}{2p}\left( m^2p^2+ mp(2p+2r+1) + (p+r)(p+r+1)\right)\\
&&+\left( -\frac{1}{2} -\frac{1}{2p}\right)\left( m^2p^2+ mp(2p+2r+1) + (p+r)(p+r+1)\right)^{1/2}\\
&&
+\so\left(\left( m^2p^2+ mp(2p+2r+1) + (p+r)(p+r+1)\right)^{1/2}\right).
\end{eqnarray*}
We have to consider the following limit and see if it evaluates to zero when $m\to +\infty$
\begin{gather*}
\begin{array}{l}
{\ds \lim_{m\to +\infty}}\left(\fr{\frac{m(m+1)}{2}p +(m+1)r
-\fr{1}{2p}\left( m^2p^2+ mp(2p+2r+1) + (p+r)(p+r+1)\right)}{\left( m^2p^2+ mp(2p+2r+1) + (p+r)(p+r+1)\right)^{1/2}} + \left( \frac{1}{2}+\frac{1}{2p}\right)\right)\eqskip
\hspace*{1cm}={\ds \lim_{m\to +\infty}}
\left(\fr{m\left( -\frac{p}{2}-\frac{1}{2}\right)}{mp\left( 1+ \frac{1}{2}\left(\frac{(2p+2r+1)}{mp} + \frac{(p+r)(p+r+1)}{m^2p^2}\right)+\so\left(\frac{1}{m}\right)\right)}
+ \left( \frac{1}{2}+\frac{1}{2p}\right)
\right)\eqskip
\hspace*{1cm}= -\fr{1}{p}\left( \fr{p}{2}+\fr{1}{2}\right)
+ \left( \fr{1}{2}+\fr{1}{2p}\right)\eqskip \hspace*{1cm}=0.
\end{array}
\end{gather*}

Now we show the second identity, for $N(\lambda_k+\epsilon)$ (cf.\ {\ref{mp}}, Sec.\ \ref{C6.1}) given by
\begin{gather*}
\begin{array}{lll}
\frac{m(m+1)}{2}p +(m+1)(r+1)
&=& \frac{1}{2p}\left( m^2p^2+ mp(2p+2r+1) + (p+r)(p+r+1)\right)\eqskip
&&\hspace*{5mm}+\left( -\frac{1}{2} +\frac{1}{2p}\right)\left( m^2p^2+ mp(2p+2r+1) + (p+r)(p+r+1)\right)^{1/2}\eqskip
&& \hspace*{10mm}
+\so\left(\left( m^2p^2+ mp(2p+2r+1) + (p+r)(p+r+1)\right)^{1/2}\right).
\end{array}
\end{gather*}
That is, 
\begin{gather*}
\begin{array}{l}
{\ds \lim_{m\to +\infty}}\left(\fr{\frac{m(m+1)}{2}p +(m+1)(r+1)
-\frac{1}{2p}\left( m^2p^2+ mp(2p+2r+1) + (p+r)(p+r+1)\right)}{\left( m^2p^2+ mp(2p+2r+1) + (p+r)(p+r+1)\right)^{1/2}} + \left( \frac{1}{2}-\frac{1}{2p}\right)\right)\eqskip
\hspace*{1cm}={\ds \lim_{m\to +\infty}}
\left(\fr{m\left( -\frac{p}{2}+\frac{1}{2}\right)}{mp\left( 1+ \frac{1}{2}\left(\frac{(2p+2r+1)}{mp} + \frac{(p+r)(p+r+1)}{m^2p^2}\right)+o\left(\frac{1}{m}\right)\right)}
+ \left( \frac{1}{2}-\frac{1}{2p}\right)
\right)\eqskip
\hspace*{1cm}= \fr{1}{p}\left( -\fr{p}{2}+\fr{1}{2}\right)
+ \left( \frac{1}{2}-\frac{1}{2p}\right)\eqskip
\hspace*{1cm}=0
\end{array}
\end{gather*} 
The other inequalities follow trivially from (\ref{trivial2}).\\[1mm]
2) The properties of $Q(\lambda)$ are elementary. We may write $\lambda=(j+\frac{1}{2}+\epsilon)^2$, where  $0\leq \epsilon <1$, $j\in \mathbb{N}_0$, and so $Q(\lambda)=-\frac{1}{2p}$ if $\epsilon=0$ and $Q(\lambda)=\frac{(1-2\epsilon)}{2p}$ if $0<\epsilon<1$,
Then for any $\varepsilon\geq 0$, writing
$\lambda + \varepsilon= (j'+\frac{1}{2}+ \epsilon')^2$. It follows that
\[ Q(\lambda+\varepsilon)-Q(\lambda)= \left\{
\begin{array}{cll}
0& \mbox{if~} j'=j \mbox{~and~} \varepsilon=\epsilon'=\epsilon=0\\
\fr{1-\epsilon'}{p} & \mbox{if~} j'=j \mbox{~and~} \epsilon'>\epsilon=0\\
\fr{\epsilon-\epsilon'}{p} & \mbox{if~} j'=j \mbox{~and~} \epsilon'>\epsilon>0\\
0& \mbox{if~} j'\geq j+1 \mbox{~and~} \varepsilon=\epsilon'=\epsilon=0\\
\fr{1-\epsilon'}{p} & \mbox{if~} j'\geq j+1 \mbox{~and~} \epsilon'>\epsilon=0\\
-\fr{(1-\epsilon)}{p} & \mbox{if~} j'\geq j+1 \mbox{~and~} \epsilon>\epsilon'=0\\
\fr{\epsilon-\epsilon'}{p} & \mbox{if~} { j'\geq j+1}  \mbox{~and~} \epsilon'\geq \epsilon>0 \mbox{~or~}
 \epsilon\geq \epsilon'>0 
\end{array}\right.\]
For $j=j'$ we have $\varepsilon=\epsilon'(\epsilon' +1)-\epsilon(\epsilon+1)$. If $j'\geq j+1$ we have
$\varepsilon\geq (1+\epsilon'-\epsilon)(2 +\epsilon'+\epsilon)$, in particular $\varepsilon \geq (1-\epsilon)(2+\epsilon)\geq (1-\epsilon)$.
Thus, if either $j'= j$ or $j'\geq j+1$, we always have $\epsilon-\epsilon'\geq -\varepsilon$ provided $\epsilon \epsilon'\neq 0$. It follows that in all cases the inequality
$Q(\lambda +\varepsilon)-Q(\lambda)
\geq -\frac{\varepsilon}{p}$ holds. 
\\[1mm]
3) All the statements in this point are a consequence of the conditions in~\cite[Theorem 1.7.6., p.\ 52]{sava} being satisfied.
First we prove inequality (\ref{1.7.4}) from 1). 
From the introduction just before the Proposition,  we take unique $K_{\lambda}$ such that $\bar{\lambda}_{K_{\lambda}-1}<\lambda\leq \bar{\lambda}_{K_{\lambda}}$,  and so $\lambda=\bar{\lambda}_{K_{\lambda}}-\epsilon$ for some $0\leq \epsilon< 2K_{\lambda}$. 
We have $\bar{\lambda}_K=\lambda + \epsilon$. Then applying the first equality of 1), 
\begin{eqnarray*}
N(\lambda)&=& N(\bar{\lambda}_K)=c_0\bar{\lambda}_K+\left(-\frac{1}{2}-\frac{1}{2p}\right)\bar{\lambda}^{1/2}_K+ \so(\bar{\lambda}_K^{1/2})\\
&=&c_0(\lambda +\epsilon)+ \left(-\frac{1}{2}-\frac{1}{2p}\right)(\lambda+\epsilon)^{1/2}+ \so(\lambda^{1/2})\\
&\leq& c_0(\lambda +\epsilon)+ \left(-\frac{1}{2}-\frac{1}{2p}\right)\lambda^{1/2}+ \so(\lambda^{1/2})\\
&\leq & c_0(\lambda +\epsilon)+ \left(-\frac{1}{2}+Q(\lambda+\epsilon)\right)\lambda^{1/2}+ \so(\lambda^{1/2})
 \end{eqnarray*} 
 Note that if we take any  $\varepsilon> \epsilon$
 we may replace $\epsilon$ by $\varepsilon$ in the last inequality.
 
We take $\bar{\lambda}_K<\lambda \leq \bar{\lambda}_{K+1}$ and so $\lambda =\bar{\lambda}_K+\epsilon$ with $\epsilon <2(K+1)$ and using the second inequality of 1) we get, 
 \begin{eqnarray*}
N(\lambda )&=& N(\bar{\lambda}_K+\epsilon) =
N(\bar{\lambda}_K)+m_{2,p}(K)\\
&=&c_0 \bar{\lambda}_K+\left(-\frac{1}{2}+ \frac{1}{2p}\right)\bar{\lambda}^{1/2}_K+ \so(\bar{\lambda}_K^{1/2})\\
&=& c_0(\lambda -\epsilon)+ \left(-\frac{1}{2}+\frac{1}{2p}\right)(\lambda-\epsilon)^{1/2}+ \so(\lambda^{1/2})\\
&\geq &c_0(\lambda -\epsilon)+ \left(-\frac{1}{2}+ \frac{1}{2p}\right)\lambda^{1/2}+ \so(\lambda^{1/2})\\
&\geq & c_0(\lambda -\epsilon)+ \left(-\frac{1}{2}+Q(\lambda-\epsilon)\right)\lambda^{1/2}+ \so(\lambda^{1/2}).
 \end{eqnarray*} 
 Again we may replace the last inequality by taking $\varepsilon >\epsilon$.
The next two equalities with respect to the $Q$ follows trivially from the definition of $Q(\lambda)$.
The second one and (\ref{1.7.4}) implies the third inequality taking $j$ large enough so that $1/(2p)< (1-2\varepsilon_j)/p$.\\

The simple reflection is satisfied
 (see next subsection). Now we verify that conditions $1.$ to $5.$ of $Q(\lambda)$ given in \cite[p.53]{sava}  are satisfied.
From (\ref{Qlambda}) and (\ref{Qbounds}) given above,
we obtain $|Q(\lambda)|\leq \frac{1}{2p}= \frac{\pi}{\boldsymbol{T}}
\int_{S^*\lune{2}{p}}dy\dbar\tilde{\eta}$. Hence $1.$ is satisfied. It is also clear from the expression for $Q(\lambda)$ that this is oscillatory  and left-continuous, so $2.$ and $3.$ are also satisfied.  Condition $4.$ follows from 2).   5. is a consequence of 3) where
  $\Pi^a(\lambda)$ is defined with 
  $q(y,\eta)=-(2p+1)\pi$ constant.  
Finally, for $6.$  we observe first that if we denote by
  $\tilde{Q}(\tilde{\lambda})$ the function defined in \cite{sava}[ (1.7.9)]  (for the eigenvalue problem $\Delta v+\tilde{\lambda}^2 v=0$, with
  $\tilde{\lambda}=\sqrt{\lambda}$), then for the dihedral lunes our function $Q(\lambda)$ is given by $ Q(\lambda)=\tilde{Q}(\sqrt{\lambda})$,. While
  $\tilde{Q}$ is $1$-periodic in the variable
  $\tilde{\lambda}$,   $Q(\lambda)$  is oscillatory in the real variable $\lambda$ and periodic over the two discrete sequences
  $\lambda=(j^{\pm}+\frac{1}{2})^2$,  with $Q((j^{\pm}+\frac{1}{2})^2)=\mp\frac{1}{2p}$ holding for all $j\in \mathbb{Z}$. The proof of these six
conditions implies the statements in Remark~\ref{newC1}. More precisely, the first identity 4) of Remark~\ref{newC1}
corresponds to the identity in \cite[Corollary 1.7.9]{sava} where $C_q= |S^*\lune{2}{\pi/p}|=(2\pi)^2/p$ with
$T=2\pi$, and $n=2$, and inequalities 5) and 6) agree with inequalities~(1.7.6) and~(1.7.7) in~\cite{sava}  (note that
$\frac{(1-2\varepsilon)}{p}\geq \frac{1}{2p}$ for $\varepsilon\to 0^+$), confirming the existence of clusters of eigenvalues lying in an
$\varepsilon_K$-neighbourhood of $K(K+1)$.  Furthermore, Using (\ref{mp})  we get for any $K\ to +\infty$ and 
$\varepsilon_K\to 0^+$, 
\begin{eqnarray*}
\begin{array}{lcl}
N( K(K+1)+\varepsilon_K)-N (K(K+1)-\varepsilon_K)
&=& k_+(K)-k_-(K)+1\eqskip
&=&\fr{(K-r)}{p}\eqskip
&\geq & \fr{1}{2p}\sqrt{K(K+1)}
+ \so\left(\frac{(K-r)}{p} \right)
\end{array}
\end{eqnarray*} 
\end{proof}

\subsection{$n\geq 2$. Trajectory phase shifts~\cite{sava}}\label{3shift}
In what follows we describe in more detail certain concepts used in the derivation of an expression for $Q(\lambda)$ using the shifts of the periodic trajectories
at the boundary hitting points. We then apply~\cite[Theorem 1.7.6]{sava} to the case of $\mathbb{L}^n_{\pi/p}$ and interpret the two-term asymptotic inequalities
using the periodic trajectories.

The billiard trajectories of an admissible initial  geodesic in a dihedral lune $\lune{n}{\pi/p}$ are not necessarily unit-speed but preserve the Hamiltonian $h(y, \eta)=\|\eta\|_{g(y)}$, where $g(y)$ is the Riemannian metric of the tangent space of the lune at $y\in \lune{n}{\pi/p}$ and $|\eta|$ is the Hilbert Schmidt norm on  the cotangent space $T^*_y \lune{n}{\pi/p}$, that corresponds to the dual inner product with respect to $g(y)$. Let $\theta=\pi/p$ and a trajectory
$\Gamma(t,y, \eta)=(x^*(t,y, \eta), \xi^*(t,y,\eta))\in T^*(\lune{n}{\pi/p})$, where $(y,\eta)$ is the value of $\Gamma$ at $t=0$.
Then $\Gamma$  has period $\boldsymbol{T}=2\pi/|\eta|$ on the variable $t$ and hits the boundary at $2p$ points, points where it reflects and closes at the $2p$-th point.
Let $v^*(t,y,\eta)=\frac{\partial}{\partial t}x^*(t,y,\eta)\in T_{x^*(t,y,\eta)}(\lune{n}{\pi/p})$, not continuous at  boundary points.\\

From now on we only consider unit-speed trajectories.\\

 The initial segment $\Gamma(t,y,\eta)$ is of the form, 
\begin{equation}\label{inicialA}\begin{array}{l}
x^*(t,y,\eta)=\cos(t)y+\sin(t)\eta^{\sharp}\\
v^*(t,y,\eta)=\left(-\sin(t)y+\cos(t)\eta^{\sharp}\right)\\
\xi^*(t,y,\eta)=(v^*(t,y,\eta))^{\flat},
\end{array}
\end{equation}
where $\eta(u)=g(u, \eta^{\sharp})$ for $u\in T_y(\lune{n}{\pi/p})$, and
 $\xi^*(t,y,\eta)(u)=g(u, v^*(t,y, \eta))$
 for $u\in T_{x^*(t,y,\eta)}(\lune{n}{\pi/p})$,
  where $g$ is the metric of $\mathbb{S}^n$.

   We take initial values with $y\in \partial_0$ and $(y,\eta)\in S^*{\lune{n}{\pi/p}}$
  with  $\eta$  not tangent to $\partial \mathbb{L}^n_{\pi/p}$,  so  that $x^*(t, y, \eta)$ is an admissible geodesic.
After reaching $\partial_{\theta}$ at time $t_1$ the trajectories is of the  form
\begin{equation}\label{inicialB}
x^*(t,y,\eta)= \cos(t)Q(y)+ \sin(t) Q\left(\eta^{\sharp}\right), 
\end{equation}
 where $Q=S_{2\theta}$ (\ref{R2redution}) and so successively on $t_j$,  multiplying alternately $S_{2\theta}$ and  $S_0$, obtaining an isometry $Q$ that preserves the polar axis. All trajectories but one are periodic of period $2p$, and for those starting at $y$ a boundary point we have $2p$ hitting points and $2p$ reflections up to the closing point that is the initial point $y$. We are considering the $2p$-th reflection at the closing point $t=t_{2p}=t_0+2\pi$ in order to consider the right derivative at $y$ as a complete end of a cycle. Let $t_0=0<t_1<\ldots< t_{2p-1}<t_{2p}$ be the boundary reflection points of $x^*(t, y, \eta)$, where  $t_{2p}=t_{0}+\boldsymbol{T}$, (Lemma~\ref{thmAA} and Lemma~\ref{thmAB}). We do not lose generality by taking initial conditions at the boundary, since
 any trajectory starting at an interior point $y_{\delta}$ can always be written using a initial geodesic of the form  
 $x_{\delta}^*(s,y_{\delta},\eta_{\delta})=
 x^*(s+\delta, y, \eta)$ with $\delta$
 sufficiently small. 
The set of admissible unit-speed absolutely periodic trajectories is $S^*(\mathbb{L}^n_{\pi/p})$ up to a zero measure set. Then we have $\boldsymbol{T}=2\pi$, except for the equatorial trajectory.

Geodesic billiards satisfy the strong simple reflection given in \cite[Definition 1.3.32]{sava} with $h(x,\xi)=(g^{ij}(x)\xi^i\xi^j)^{1/2}$, since the non-admissible geodesics form a set of zero  measure (see Remark~\ref{zeromeasure} 1)). The  boundary is not smooth, but this is not a problem if  $\theta< \pi$ (cf.\ \cite{va}), that is the  case $p\geq 2$,  neither a problem for $p=1$, that is  $\theta= \pi$, since the boundary is smooth.
Following \cite[Definition 1.7.2, p.\ 52]{sava},  for each $T$-admissible trajectory $\Gamma(t,y,\eta)$, for $0\leq t\leq T$,
some phase shifts are defined as follows. Consider the following functions. \\[1mm]
\begin{enumerate}[1)]
\item $f_r(\boldsymbol{T},y,\eta)$ the total phase shift generated
 by the reflections of the trajectory $\Gamma$ , that is given by the sum of the phase shifts generated by
the reflections on $\Gamma$;
\item If $x_{\eta}(\boldsymbol{T}, y, \eta)=0$ then define
$f_c(T; y, \eta)=-\alpha_{\Gamma} \pi/2$ the phase shift generated by the passage of the trajectory $\Gamma$ through the caustics  where $\alpha_{\Gamma}$ is the Maslov index.
\item $f_s(\boldsymbol{T}, y, \eta)=-\frac{1}{2}
\int_0^T A_{sub}(x^*(t; y, \eta), \xi^*(t,y,\eta))dt$ is the phase shift of the subprincipal symbol $A_{sub}$ of $-\Delta u$,
$ A(x,\xi)=g^{ij}(x) \xi_i\xi_j+ \sqrt{-1}g^{ij}(x)\Gamma^k_{ij}(x)\xi_k= A_{pr}(x,\xi)+ \sqrt{-1}A_{sub}(x,\xi).$
\end{enumerate}
Explicitly, we have the following.
(1) An auxiliary one-dimensional eigenvalue problem associated to the trajectories $\Gamma(t,y,\eta)$ is defined as follows on a Riemannian domain $\Omega$.  Given $y\in \partial \Omega$ a smooth boundary point,  we take a local  coordinate system
 $x'$ of $ \partial\Omega$ normal at  $y$, and set $x_n(\cdot)=d(\cdot,\partial\Omega)$ defining a Fermi coordinate system $x=(x', x_n)$ on $\Omega\cup \partial \Omega$ ($y$ is identified
 with  $(y,0)=x(y)=(x'(y),0)$). Hence $g_{ij}(y)=\delta_{ij}$ for all $i,j$, and $\nabla x_n=\partial/\partial_{x_n}$ is a local unit normal to $\partial \Omega$ (\cite{gray}), and $dx_n$
 is the co-normal. The auxiliary eigenvalue problem at fixed $(y,\eta')\in T^*\partial \Omega$, where $\eta'$ in these coordinates is the orthogonal projection of $\eta^{\sharp}$ on the
 tangent space of the boundary, given by
 $A_{pr} \left(y,0,\eta', -i{\partial}/{\partial x_n}\right)u(t)
 =-u''(t)+|\eta'|^2u(t)=\nu u(t)$, with $ u(0)=0$ and for $t=x_n\in [0, +\infty)$. There are no true eigenvalues, i.e. there are no $L^2$ solutions,
 and so its counting function vanishes $N^+=0$. However  the continuous spectrum is given by $[\nu_1^{st}, +\infty)$, where $\nu_1^{st}=|\eta'|^2=\min_{\xi_n}A_{pr}(y,0,\eta',\xi_n)$ is a threshold of the symbol in the sense that $A_{pr}(\xi_n)=\nu_1^{st}$  has a multiple real $\xi_n$-root, while  each $\nu> \nu^{st}_1$ has multiplicity one with a reflection matrix  $R(\nu)$ that satisfies   $\mathrm{arg}(R(\nu))= -\pi$.
 (cf. \cite[p.\ 28-32, (1.6.15) p.\ 43, Example 1.6.14  p.44-46]{sava}).

On $\lune{n}{\pi/p}$,  $\eta=(\eta',\eta_n)$ where $\eta_n\neq 0$ by  admissibility, hence $|\eta'|<1$, and  we have $\mathrm{arg}(R(1))=-{\pi}$.  The total phase shift $f_r(\boldsymbol{T},y,\eta)$ is the sum of the phase shifts generated by the reflections of $T(t)$, that are $2p$ in number.  Hence, 
\begin{equation}\label{fr}
f_r(\boldsymbol{T},y,\eta)= 2p\left(\mathrm{arg}( R(1))\right)=-2p\pi.
 \end{equation}
\noindent
(2) Now we compute the Maslov index $\alpha_{\Gamma}$ and $f_c(\boldsymbol{T}, y, \eta)$. On each geodesic segment $t\in [t_j, t_{j+1}]$, we have  
$x^*(t,y,\eta)= \cos(t)Q(y)+\sin(t)Q(\eta^{\sharp})$, where $Q=Q_j$ is an isometry of $\mathbb{R}^{n+1}$ that fixes the polar axis. We take an  orthonormal basis $e_k\in T_y\Se^n$, $k=1,\ldots, n$ where $e_n=\eta^{\sharp}$. The derivative of $x^*(t,y, \cdot)$ at $\eta$ is zero in the $e^{\flat}_n$ direction. More precisely, any variation of $\eta^{\sharp}$ by unit elements in $T_y\Se^n$ is of the form  $\eta^{\sharp}_s=(\eta^{\sharp} + se_k)/|\eta^{\sharp} + se_k|$, $s\in \R$, with $\eta_0=\eta $, and  $V_k:=\frac{d}{ds}|_{s=0} \eta^{\sharp}_s=
e_k- g(e_k,\eta^{\sharp})\eta^{\sharp}$, and  this is $e_k$ for $k\neq n$ and zero for $k=n$.   Consequently, 
\begin{equation}
d x^*(t,y,\cdot)(V_k^{\flat})=\frac{d}{ds}|_{s=0}x^*(t,y, \eta_s)= \sin(t)\, Q(e_k^* -g(e_k, \eta^{\sharp})\eta), 
\end{equation} 
we conclude  $\mathrm{rank}\left(\frac{\partial}{\partial \eta_k} x^*(t,y,\eta)\right) $ is $n-1$ for $t\neq m\pi$, and $0$ for $t=m\pi$. 
The  conjugate points are the points $t$ where $\mathrm{rank}\left(\frac{\partial}{\partial \eta_k} x^*(t,y,\eta)\right) $ is $< n-1$.
By definition (cf.~\cite[Definition 1.7.2, p.\ 52]{sava}),  $f_c(\boldsymbol{T},y,\eta)=-\alpha_{\Gamma}\pi/2$, where $\alpha_{\Gamma}$ is the Maslov index that is the number of
conjugate points on $(0,\boldsymbol{T}]$ counted with their multiplicities~\cite[Example 1.5.8., p.\ 36]{sava}, and so they are at $t=\pi$ and $t= 2\pi=\boldsymbol{T}$, each one with
multiplicity given by the formula $n-1-\mathrm{rank}=n-1$, \cite[p.\ 34]{sava}. Therefore,
\begin{equation}\label{fc}
\alpha_{\Gamma}=2(n-1), \quad\quad f_c(\boldsymbol{T}, y,\eta)=-\alpha_{\Gamma}\frac{\pi}{2}= -(n-1)\pi.
\end{equation}

\noindent
(3) Next consider the stereographic projection
$x: \Se^n\to \mathbb{R}^n$, $x(y)= \frac{y'}{(1-y_0)}=z$ with $x^{-1}(z)=\left(\frac{|z|^2-1}{|z|^2+1}, \frac{2z}{(1+|z|^2)}\right)$.  We have a coordinate chart  $(x,\xi):T^*\Se^n\to \R^n\times (\R^n)^*$ given by $(x,\xi)(y,\eta)=(x^i(y), \xi_i(y,\eta))$ where $\eta_i=\xi_i(y,\eta)$ is the components of $\eta$ in the basis $dx^i(y)$, that is $\eta=\sum_i\eta_idx^i(y)$, 
$\eta_i=\eta\left(\partial_{x_i}(y)\right)$ where $\partial_{x_i}(y):=\frac{\partial}{\partial x_i}(y)= dx^{-1}(x(y))(\epsilon_i)$, and $dx^i(y)\circ dx^{-1}(x(y))=\epsilon_i^*\in (\mathbb{R}^n)^*$. 
Denoting by  $g$ and $\langle, \rangle$  the Riemannian metrics of $\Se^n$, and $\R^n$ respectively,   $x$ is a conformal map  with coefficient of conformality $ \frac{4}{(1+|z|^2)^2}$, that is
\begin{eqnarray*}
g(dx^{-1}(z)(X), dx^{-1}(z)(Y)) &=&\frac{4}{(1+|z|^2)^2}\langle X,Y\rangle,
\end{eqnarray*}
and so, $e_i(y)=\frac{1+|x(y)|^2}{2}\partial_{x_i}(y)$ is a o.n. basis of $T_{y}\Se^n$. We use the following notations,   $g_{ij}(x(y))=g\left( \partial_{x_i}(y),  \partial_{x_i}(y)\right)=\frac{4}{(1+ |x(y)|^2)^2}\delta_{ij}$, 
and for  short $x^*=x^*(t,y,\eta)$,
{$\xi^*=\xi^*(t,y,\eta)\in T^*_{x^*}\mathbb{S}^n$ with corresponding dual
$\xi^{*\sharp}\in T_{x^*}\mathbb{S}^n$}. So,  
\[(x(x^*), \xi(x^*, \xi^*))=\left(\frac{ x_1^*}{1-x^*_0}, \ldots \frac{x_n^*}{1-x^*_0}, \xi^*(\partial_{x_1}(x^*)), \ldots, \xi^*(\partial_{x_n}(x^*)\right),\]
with $\xi^*\left(\partial_{x_k}(x^*)\right)=\frac{2}{1+|x(x^*)|^2}\xi^*(e_k(x^*))$. 
The trajectories preserve the  norm $|\xi^*|$ along the time  $t$ , and so $|\eta|=|\xi^*|$. 
 The Christoffel symbols are given by
\[\Gamma^i_{ii}(x)=\Gamma^j_{ij}(x)=\Gamma^j_{ji}(x)= -\frac{2 x^i}{1+|x|^2}~\mbox{~~if}~i\neq j, \quad\quad \Gamma^k_{ii}=\frac{2 x^k}{1+|x|^2}~\mbox{~~if}~i\neq k.\]
Hence, and using 
$dx^{-1}(z)(z) = \left(\frac{4|z|^2}{(|z|^2+1)^2}, \frac{ 2z(1-|z|^2)}{(|z|^2+1)^2}\right)$, we have 
\begin{eqnarray*}
\lefteqn{A_{sub}(x^*(t,y,\eta), \xi^*(t,y,\eta)) = g^{ij}(x(x^*))\Gamma^{k}_{ij}(x(x^*)) \xi^*(\partial_{x_k}(x^*))}\\
& =& \frac{(1+|x(x^*)|^2)^2}{4}{\sum_{k,s}} \Gamma^k_{ss}(x(x^*))\xi^*(\partial_{x_k}(x^*))\\
& =&  \frac{(1+|x(x^*)|^2)^2}{4}\sum_s\left(\sum_{k\neq s} \frac{2 x^k(x^*)}{1+|x^*|^2}\xi^*(\partial_{x_k}(x^*))-\frac{2 x^s(x^*)}{1+|x^*(x^*)|^2}\xi^*(\partial_{x_s}(x^*)))\right)\\
&=&(n-2)  \frac{(1+|x(x^*)|^2)}{2}\sum_{k=1}^n (x^k(x^*)\xi^*(\partial_{x_k}(x^*))\\
&=& (n-2)\sum_k {g((\xi^{*})^{\sharp }, e_k(x^*))}\langle x(x^*), \epsilon_k\rangle \quad\quad  
\\
&=& (n-2) \frac{(1+|x(x^*)|^2)^2}{4}g((\xi^{*})^{\sharp }, e_k(x^*))g\left( dx^{-1}(x(x^*))(x(x^*)), dx^{-1}(x(x^*))(\epsilon_k)\right)\\
&=&(n-2) \frac{(1+|x(x^*)|^2)}{2}g((\xi^{*})^{\sharp }, e_k(x^*))g\left( dx^{-1}(x(x^*))(x(x^*)), e_k(x^*)\right)\\
&=&(n-2) \frac{(1+|x(x^*)|^2)}{2}g\left((\xi^*)^{\sharp}~, ~ \left(\frac{4|x(x^*)|^2}{(1+|x(x^*)|^2)^2},
2x(x^*)\frac{(1-|x(x^*)|^2)}{(1+|x(x^*)|^2)^2}\right)\right)\\
&=&(n-2)~ g\left(v^*, \left(\frac{|x(x^*)|^2}{(1+|x(x^*)|^2)}, 2x(x^*)\frac{(1-|x(x^*)|^2)}{(1+|x(x^*)|^2)}\right)\right).
\end{eqnarray*}
 Let us assume $n\geq 3$ and $|\eta|=1$.
Note that $x^*_0\neq 1$ by admissibility. We have 
\[ |x(x^*)|^2=\frac{1+x^*_0}{1-x^*_0}, \quad 2x(x^*)=\frac{2 (x^*)'}{1-x^*_0}, \quad 1+|x(x^*)|^2=\frac{2}{1-x^*_0}, \quad \frac{1-|x(x^*)|^2}{1+|x(x^*)|^2}=-x^*_0\]
\begin{eqnarray*}
f_s(\boldsymbol{T},y,\eta)&=&-\frac{(n-2)}{2}\int_0^{2\pi} 
g\left(v^*, \left(\frac{|x(x^*)|^2}{(1+|x(x^*)|^2)}, 2x(x^*)\frac{(1-|x(x^*)|^2)}{(1+|x(x^*)|^2)}\right)\right)dt\nonumber\\
&=&-\frac{(n-2)}{2}\int_0^{2\pi} 
g\left(v^*, \left(\frac{1+x^*_0}{2}, -\frac{x^*_0}{1-x^*_0}2(x^*)'\right)\right)dt
\end{eqnarray*}
We take the reflection points $t_1, \ldots, t_{2p}=2\pi$., ordered as  $0=t_0<t_1<\ldots<t_{2p-1}<t_{2p}$.  Then
\begin{gather*}
f_s(\boldsymbol{T},y,\eta)=\sum_{j=0}^{2p-1}-\frac{(n-2)}{2}\int_{t_j}^{t_{j+1}}
 g\left( -\sin(t)Q_j(y)+\cos(t)Q_j(\eta^{\sharp})~,~ \right.\\
 \left.\left(
 \frac{(1+\cos(t)Q_i(y)_0 + \sin(t)Q_j(\eta^{\sharp})_0)}{{2}},
 -\frac{(\cos(t)Q_j(y)_0 + \sin(t)Q_j(\eta^{\sharp})_0)}{(1-\cos(t)Q_j(y)_0-\sin(t)Q_j(\eta^{\sharp})_0)}2
 \left(\cos(t)Q_j(y)'+\sin(t)Q_j(\eta^{\sharp})'\right)\right)
 \right) dt\\[1mm]
 = \sum_{j=0}^{2p-1}-\frac{(n-2)}{2}\int_{t_j}^{t_{j+1}}
 \left[\left( -\sin(t)Q_j(y)_0+\cos(t)Q_j(\eta^{\sharp})_0\right)
 \times \left(
 \frac{(1+\cos(t)Q_j(y)_0 + \sin(t)Q_j(\eta^{\sharp})_0)}{{2}}\right)\right.\\
 \left.+ \left\langle  \left(-\sin(t)Q_j(y)'+\cos(t)Q_j(\eta^{\sharp})'\right),\left(
 -\frac{(\cos(t)Q_j(y)_0 + \sin(t)Q_j(\eta^{\sharp})_0)}{(1-\cos(t)Q_j(y)_0-\sin(t)Q_j(\eta^{\sharp})_0)} \times 2
 \left(\cos(t)Q_j(y)'+\sin(t)Q_j(\eta^{\sharp})'\right)\right)\right\rangle \right]dt
\\[1mm]
=\sum_{j=0}^{2p-1}-\frac{(n-2)}{2}\int_{t_j}^{t_{j+1}}
\left[ \left( -\sin(t)Q_j(y)_0+\cos(t)Q_j(\eta^{\sharp})_0\right)
 \times \left(
 \frac{(1+\cos(t)Q_j(y)_0 + \sin(t)Q_j(\eta^{\sharp})_0)}{{2}}\right)\right.\\[1mm]
\left. -2\left(\frac{(\cos(t)Q_j(y)_0 + \sin(t)Q_j(\eta^{\sharp})_0)}{(1-\cos(t)Q_j(y)_0-\sin(t)Q_j(\eta^{\sharp})_0)}\right)  \left\langle  \left(-\sin(t)Q_j(y)'+\cos(t)Q_j(\eta^{\sharp})'\right),\left(
\cos(t)Q_j(y)'+\sin(t)Q_j(\eta^{\sharp})'\right)\right\rangle\right] dt\\[1mm]
=\sum_{j=0}^{2p-1}-\frac{(n-2)}{2}\int_{t_j}^{t_{j+1}}
 \left[\left( -\sin(t)Q_j(y)_0+\cos(t)Q_j(\eta^{\sharp})_0\right)
 \times \left(
 \frac{(1+\cos(t)Q_j(y)_0 + \sin(t)Q_j(\eta^{\sharp})_0)}{{2}}\right)\right.\\[1mm]
 \left.+2\left(\frac{(\cos(t)Q_j(y)_0 + \sin(t)Q_j(\eta^{\sharp})_0)}{(1-\cos(t)Q_j(y)_0-\sin(t)Q_j(\eta^{\sharp})_0)}\right)    \left(-\sin(t)Q_j(y)_0+\cos(t)Q_j(\eta^{\sharp})_0\right)\times\left(
\cos(t)Q_j(y)_0+\sin(t)Q_j(\eta^{\sharp})_0\right)\right] dt\\[1mm]
=\sum_{j=0}^{2p-1}-\frac{(n-2)}{2}\int_{t_j}^{t_{j+1}}
 \left( -\sin(t)Q_j(y)_0+\cos(t)Q_j(\eta^{\sharp})_0\right)
 \times {\frac{1+3\left(\cos(t)Q_j(y)_0 + \sin(t)Q_j(\eta^{\sharp})_0\right)^2}{2(1-\cos(t)Q_j(y)_0 - \sin(t)Q_j(\eta^{\sharp})_0)}}dt.
\end{gather*}
Set for $t\in [t_j,t_{j+1}]$, 
$ F_j(t)=(Q_j(\gamma(t)))_0=
 \cos(t)Q_j(y)_0 +\sin(t)Q_j(\eta^{\sharp})_0$
  that never takes the value $1$ by admissibility.
Using the representation  (\ref{representation}) of $\gamma(t)$ an the fact that $Q_j$ preserves $w(t)$,  we see that $F_j(t)=w(t)_0$ for all $j$.
 We are thus lead to
\begin{eqnarray*}
f_s(\boldsymbol{T},y,\eta)  	
 &=&-\frac{(n-2)}{4}\sum_{j=0}^{2p-1}\int_{t_j}^{t_{j+1}}\frac{ F'_j(t)(1+3F_j^3(t))}{(1-F_j(t))}dt\\
 &=&-\frac{(n-2)}{4}\sum_{j=0}^{2p-1}\left.\left( 6(1-F_j(t))+\frac{3}{2}(1-F_j(t))^2-4\log(1-F_j(t))\right)\right]^{t_{j+1}}_{t_j}\nonumber\\ 
 &=& -\frac{(n-2)}{4}\sum_{j=0}^{2p-1}\left.\left( 6(1-w(t)_0)+\frac{3}{2}(1-w(t)_0)^2-4\log(1-w(t)_0)\right)\right]^{t_{j+1}}_{t_j}\nonumber\\
 &=&  -\frac{(n-2)}{4}\sum_{j=0}^{2p-1}\left( 6(1-w(t_{j+1})_0)+\frac{3}{2}(1-w(t_{j+1})_0)^2-4\log(1-w(t_{j+1})_0)\right.\nonumber\\
&& \quad\quad\quad\quad\quad\quad\left.-6(1-w(t_{j})_0)+\frac{3}{2}(1-w(t_{j})_0)^2-4\log(1-w(t_{j})_0)\right)\nonumber\\
&=&-\frac{(n-2)}{4}\left( 6(1-w(t_{2p})_0)+\frac{3}{2}(1-w(t_{2p})_0)^2-4\log(1-w(t_{2p})_0)\right.\nonumber\\
&&  \quad\quad\quad\quad\quad\quad\left.-6(1-w(t_{0})_0)+\frac{3}{2}(1-w(t_{0})_0)^2-4\log(1-w(t_{0})_0)\right)\nonumber\\
&=& 0 \nonumber
\end{eqnarray*}
\begin{remark}
We considered the points  $0=t_0<t_1<\ldots<t_{2p-1}<t_{2p}=t_{0}+2\pi$ corresponding to orbits whose first point is on the boundary, while
in~\cite[p.33]{sava}) $t_{0}$ is taken to be positive and the initial point does not belong to the boundary. However the two approaches
are equivalent in this instance, as trajectories are periodic of period $2\pi$ and so the integrals involved have the same value on any interval
of length $2\pi$.
\end{remark}

 We are ready to derive the expressions of $\boldsymbol{q}(y,\eta)$ and $Q(\lambda)$ in the next lemma.
\begin{lemma} The total shift  of a $n$-dimensional dihedral lune  $\lune{n}{\pi/p}$ 
is given by the sum of the 3 shifts, with $f_s(2\pi, y,\eta)=0$, 
\begin{equation}\label{boldq}
\boldsymbol{q}(y, \eta)=f_r(2\pi, y,\eta)+ f_c(2\pi,y,\eta)=-2p\pi -(n-1)\pi=-(n+2p-1)\pi.
\end{equation}
Moreover,
\begin{eqnarray}\label{Q2}
Q(\lambda)&=& {\frac{1}{p}}\left(\left\{{1}-\sqrt{\lambda}+\frac{2p+1}{2}\right\}-\frac{1}{2}\right),
\hspace*{3cm} \mbox{ if~}n=2\\
\label{Qn}
Q(\lambda)&=&\frac{1}{(2\pi)^n} {\frac{|\mathbb{S}^n|}{2p}}|\mathbb{S}^{n-1}|\left(\left\{{1}-\sqrt{\lambda}+\frac{2p+n-1}{2}\right\}-\frac{1}{2}\right)  \mbox{, if~}n\geq 2.
\end{eqnarray}
\end{lemma}
\begin{proof}
For any $n\geq 2$, 
\begin{eqnarray*}
\frac{\{\pi-\boldsymbol{q}-\sqrt{\lambda} \boldsymbol{T}\}_{2\pi}}{2\pi}&=&
\left\{ 1-\sqrt{\lambda} +\frac{2p+n-1}{2}\right\}-\frac{1}{2}.
\end{eqnarray*}
Therefore,
\begin{eqnarray*}
Q(\lambda)&=&
\int_{S^*\mathbb{L}^n_{\pi/p}}\left(\left\{{1}-\sqrt{\lambda}+\frac{2p+n-1}{2}\right\}-\frac{1}{2}\right) dy{\dbar\tilde{\eta}}\eqskip
&=&\frac{1}{{(2\pi)^n}}|S^*\mathbb{L}^n_{\pi/p}|\left(\left\{{1}- \sqrt{\lambda} +\frac{2p+n-1}{2}\right\}-\frac{1}{2}\right)\nonumber \eqskip
&=&
\frac{1}{{(2\pi)^n}}|\mathbb{L}^n_{\pi/p}||\mathbb{S}^{n-1}|\left(\left\{{1}-\sqrt{\lambda}+\frac{2p+n-1}{2}\right\}-\frac{1}{2}\right)\nonumber \eqskip
&=&\frac{1}{(2\pi)^n} {\frac{|\mathbb{S}^n|}{2p}}|\mathbb{S}^{n-1}|\left(\left\{{1}-\sqrt{\lambda}+\frac{2p+n-1}{2}\right\}-\frac{1}{2}\right)
\end{eqnarray*}
The case $n=2$ follows from $|\lune{2}{\pi/p}|=|\mathbb{S}^2|/(2p)=4\pi/(2p)=2\pi/p$.
\end{proof}
\begin{remark}
 Taking $\varepsilon\to 0^+$ at (\ref{1.7.4}) and the fact that 
$Q(\lambda)\leq \frac{1}{2p}$  we get  for $p\geq 2$, 
$N(\lambda)+\frac{|\mathbb{S}^{n-1}|}{(2\pi)^{n-1}}\frac{\omega_{n-1}}{4}(1-\frac{1}{p})<c_0\lambda^{n/2}$
when $\lambda $ sufficiently large.  This also provides another proof that dihedral lunes with $p\geq 2$
satisfy P\'{o}lya's conjecture eventually, as stated in Corollary~\ref{CorC} in Section~\ref{Sec Third term}.
More generally, for any domain $\Omega$ of a Riemannian manifold satisfying  $\sup Q(\lambda)<-c_1$ we conclude the Dirichlet eigenvalues of $\Omega$  satisfy P\'{o}lya conjecture eventually.
This is the case when 
$\frac{|\partial \Omega|}{|\Omega|}> 2 \frac{|\mathbb{S}^{n-1}|}{|\mathbb{S}^{n}|}$,
independently of the size of the  subset of initial conditions of the periodic trajectores  $\Pi^a\subset S^*(\Omega)$,  or when the canonic measure of $\Pi^a$ as a subset of $S^*(\Omega)$ satisfies 
 $|\Pi^a|< -2c_1=\frac{\omega_{n-1}|\partial \Omega|}{2(2\pi)^{n-1}}$. 
\end{remark}
 
\section{Proof of inequality~(\ref{Ineq})\label{prooflemma64}}

\begin{lemma}\label{lmx1}
For all positive integer $n$ and positive real $R$ we have
\[
\sqrt{n-1+R}\left[ (R+2n-2)^{\frac{n}{2}}-R^{\frac{n}{2}}\right]
\geq n(n-1)\left[1+R+\frac{(n-3)(n-2)}{6(R+n-1)} \right]^{\frac{n-1}{2}},
\]
with equality for $n=1,2$.
\end{lemma}

\begin{proof}
 Define the function
 \[
  f(n,R) = \sqrt{n-1+R}\left[ (R+2n-2)^{\frac{n}{2}}-R^{\frac{n}{2}}\right] - n(n-1)\left[1+R+\frac{(n-3)(n-2)}{6(R+n-1)} \right]^{\frac{n-1}{2}}
 \]
 for all positive integer $n$ and positive real $R$. Direct calculations yield $f(1,R) = f(2,R) = 0$, and so the result holds
 with equality for $n=1,2$.

 We postpone the proof of the case of $n=3$ till later, and shall now prove positivity of $f$ for all $n$ greater than
 or equal to $4$ and all positive $R$.
 Let $d=n-1$ and $M=R+n-1$. Then $n-1+R = M$, $R+2(n-1) = M+d$, and $R=M-d$. We now consider the function
 \[
 \begin{array}{lll}
  F(d,M) & = & f(d+1,M-n+1)\eqskip
  & = & M^{1/2} \left[ (M+d)^{\frac{d+1}{2}} - (M-d)^{\frac{d+1}{2}} \right]
  - d(d+1)\left[ M+1-d +\fr{(d-1)(d-2)}{6M} \right]^{\frac{d}{2}}.
 \end{array}
 \]
 Clearly $f$ is positive for $n$ greater than or equal to $4$ and all positive $R$ if we show
 that $F$ is positive for all $d$ greater than or equal to $3$ and all positive $M$. We first note
 that, for $d$ larger than $3$, the function $g:\R^{+}\to \R^{+}$ defined by $g(x) = x^{\frac{d+1}{2}}$ has
 a first derivative which is continuous and convex. Then, by Lemma~\ref{auxconvlemma} below,
 we have
 \[
  g(M+d)-g(M-d) \geq 2d g'(M),
 \]
 and so
 \begin{equation}
 \begin{array}{lll}
  F(d,M) & \geq & M^{1/2} \left[d(d+1) M^{\frac{d-1}{2}}\right]-d(d+1) \left[ M+1-d + \fr{(d-1)(d-2)}{6M}\right]^{\frac{d}{2}}\eqskip
  & = & d(d+1) \left( M^{\frac{d}{2}} - \left[M+1-d + \fr{(d-1)(d-2)}{6M}\right]^{\frac{d}{2}}\right).\label{flowb}
 \end{array}
 \end{equation}
 This will be positive provided that
 \[
  \begin{array}{lcll}
   & M & > & M+1-d + \fr{(d-1)(d-2)}{6M} \eqskip
   \Leftrightarrow & d-1 & > & \fr{(d-1)(d-2)}{6M} \eqskip
   \Leftrightarrow & 1 & > & \fr{(d-2)}{6M} \eqskip
   \Leftrightarrow & 6M & > & d-2 \eqskip
   \Leftrightarrow & 6(R+n-1) & > & n-3 \eqskip
   \Leftrightarrow & 6R + 5 n - 3 & > &0,
  \end{array}
 \]
 which is clearly true for $n$ larger than $3$.

 Finally we will consider the case $n$ equal to three. In this instance Lemma~\ref{auxconvlemma} may not be applied directly,
 as the correcponding function $g(x) = x^{3/2}$ does not have a convex derivative. We will, however, be able to establish the
 result in this case by following a similar path to that used in the proof of the lemma. When $d=2$ the function $F$ defined above
 becomes
 \[
 \begin{array}{lll}
  F(2,M) & = & M^{\frac{1}{2}}\left[ (M+2)^{\frac{3}{2}} - (M-2)^{\frac{3}{2}} \right] -6(M-1)\eqskip
  & = & \fr{3}{2} M^{\frac{1}{2}} \dint_{M-2}^{M+2} t^{1/2} \ {\rm d}t - 6(M-1)\eqskip
  & = & 3M^{\frac{1}{2}} \dint_{(M-2)^{1/2}}^{(M+2)^{1/2}} y^2 \ {\rm d}y - 6(M-1) \eqskip
  & \geq & \fr{12 M^{\frac{1}{2}}}{(M+2)^{1/2}-(M-2)^{1/2}}- 6(M-1)\eqskip
  & = & 3M^{\frac{1}{2}} \Big[(M+2)^{1/2}+(M-2)^{1/2}\Big]- 6(M-1),
 \end{array}
 \]
 where the inequality follows from Jensen's inequality. To see that this is positive for all $M$ larger than $2$ ($R>0$),
 it is now enough to consider the equivalent inequaliy $3M^{\frac{1}{2}} \Big[(M+2)^{1/2}+(M-2)^{1/2}\Big] > 6(M-1)$.
 After some tedious but standard calculations, we conclude that this holds for all $M$ larger than or equal to $2$.
\end{proof}
\begin{remark}
 By applying Lemma~\ref{auxconvlemma} to~\eqref{flowb} again, now with $g(x) = x^{d/2}$ and $d$ greater than or equal to $4$,
 it is possible to obtain another, perhaps more explicit, lower bound for $F(d,M)$ and, as a consequence, for the function $f(n,R)$.
 More precisely, for $n\geq5$ and all positive $R$ we have
 \[
  f(n,R) \geq \fr{1}{2}n(n-1)^2(n-2) \left(1-\fr{n-3}{6(R+n-1)} \right) \left( R+\fr{n}{2} +\fr{n(n-1)}{12(R+n-1)}\right).
 \]
 This may then be used to bound the function $\Phi_n$ used in the proof of Theorem~\ref{thmxquantitative}, but we will not pursue
 this approach further.
\end{remark}

The following lemma is a variation of the Hermite-Hadamard inequality for convex functions, written in differential form. 
For completeness, we provide the statement and the corresponding proof here.
\begin{lemma}\label{auxconvlemma}
 Let $g:\R^{+}\to \R^{+}$ have a continuous first derivative which is convex. Then
 \[
  g(x+y) - g(x-y) \geq 2y g'(x),
 \]
 for all positive numbers $x,y$ such that $y<x$.
\end{lemma}
\begin{proof}
 We have $g(x+y) -g(x-y) = \dint_{x-y}^{x+y} g'(t) \, {\rm d} t.$ Writing $\varphi(t) = g'(t)$, by Jensen's inequality
 we have
 \[
 \begin{array}{lcll}
  & \varphi\left( \fr{1}{2y} \dint_{x-y}^{x+y} t \ {\rm d}t\right) & \leq  & \fr{1}{2y} \dint_{x-y}^{x+y} \varphi(t) \, {\rm d}t\eqskip
  \Leftrightarrow & \varphi(x) & \leq & \fr{1}{2y} \dint_{x-y}^{x+y} \varphi(t) \, {\rm d}t\eqskip
  \Leftrightarrow & g'(x) & \leq & \fr{1}{2y} \left( g(x+y) - g(x-y)\right),
 \end{array}
 \]
as desired.
\end{proof}

\section*{Acknowledgements} This work was partially supported by the Funda\c c\~{a}o para a Ci\^{e}ncia e a Tecnologia (Portugal) through project UID/00208/2025.
The second author is grateful to Professor Liu Xiang for his hospitality during a visit to the Lanzhou Center for Theoretical Physics, School of
Physical Science and Technology, Lanzhou University, where part of this work was carried out.

\end{document}